\documentclass[11pt]{article}

\usepackage[a4paper,margin=1in]{geometry}
\usepackage{amsmath,amssymb,amsthm,mathtools}
\usepackage{enumitem}
\usepackage{microtype}
\usepackage{graphicx}
\usepackage{tikz}
\usetikzlibrary{positioning,calc}
\usepackage[hidelinks]{hyperref}
\usepackage[nameinlink,capitalize,noabbrev]{cleveref}

\newtheorem{theorem}{Theorem}[section]
\newtheorem{proposition}[theorem]{Proposition}
\newtheorem{lemma}[theorem]{Lemma}
\newtheorem{corollary}[theorem]{Corollary}
\theoremstyle{definition}
\newtheorem{definition}[theorem]{Definition}
\theoremstyle{remark}
\newtheorem{remark}[theorem]{Remark}

\title{Uniform High-Frequency Localization on Quantum Graphs}
\author{Binh T. Nguyen\\
Faculty of Mathematics and Computer Science, University of Science,\\
Vietnam National University Ho Chi Minh City, Ho Chi Minh City, Vietnam\\
\texttt{ngtbinh@hcmus.edu.vn}}
\date{}

\begin{document}
\maketitle

\begin{abstract}
We study uniform high-frequency localization for the Laplacian on compact
metric graphs through the least \(L^2\)-mass that eigenfunctions must place in
a prescribed measurable observation set. We first identify this asymptotic
localization constant with the minimum of a linear functional over the
attainable edge-intensity set; in the generic standard-Kirchhoff setting,
this set is governed by the regular Gauss image of the secular manifold.
Primitive cycles and exterior-to-exterior paths consequently determine the
positivity threshold, but not, in general, the positive numerical value. We
introduce a boundary-aware singular-completion cone and prove that it contains
all regular secular edge-energy vectors for trees, unicyclic graphs, and
closed graphs of cycle rank two. We then construct a cycle-rank-two graph
with two Dirichlet leaves for which this completion principle fails. An exact
rational separator, combined with a validated Krawczyk enclosure, yields a
nonsingular scalar secular state lying outside every boundary-compatible
singular sector. A positive radial derivative identity and recurrence in the
compact orbit closure convert this local separation into an exact
high-frequency eigensequence for a single fixed metric. For a suitable
measurable observation set, the true high-frequency localization constant
\(C_\infty(\omega;\ell)\) and its singular-completion counterpart
\(C_{\mathrm{sing}}(\omega;\ell)\) satisfy
$$
C_\infty(\omega;\ell)<\frac12<C_{\mathrm{sing}}(\omega;\ell).
$$
Thus, singular completion captures the quantitative localization geometry in
several low-complexity classes but does not, in general, determine the
high-frequency variational problem on a fixed quantum graph.
\end{abstract}



\section{Introduction}
\label{sec:introduction}

Let \(G\) be a finite compact metric graph with edge-length vector
\(\ell=(\ell_e)_{e\in E}\), standard Kirchhoff conditions at internal
vertices, and prescribed Dirichlet or Neumann conditions at exterior
endpoints.  For a measurable observation set \(\omega\subset G\), define
\begin{equation}
C_\infty(\omega;\ell)
:=
\lim_{\Lambda\to\infty}
\inf_{\substack{
-\Delta_{G,\ell}u=\lambda u\\
\lambda\ge\Lambda,\ \|u\|_{L^2(G)}=1}}
\int_\omega |u|^2.
\label{eq:intro-Cinf}
\end{equation}
This is the least asymptotic fraction of \(L^2\)-mass that high-frequency
eigenfunctions are forced to place in \(\omega\).  In the language of
mathematical physics, it is a fixed-system high-energy localization
observable: the metric graph is held fixed while the spectral parameter tends
to infinity.

There are two different questions hidden in
\eqref{eq:intro-Cinf}.  The first is qualitative:
when is \(C_\infty(\omega;\ell)>0\)?  The second is quantitative:
when the constant is positive, what spectral geometry determines its value?
The distinction is essential on quantum graphs.  Minimal-support scars can
detect whether the constant vanishes, while the minimizing positive
high-frequency intensity can have larger support and need not arise from a
single primitive scar.  The purpose of this paper is to identify the quantitative spectral geometry
behind this distinction, determine low-complexity regimes in which singular
states already control the variational problem, and exhibit a fixed quantum
graph on which interior regular secular geometry produces strictly stronger
high-energy localization than every boundary-compatible singular
completion.

\subsection{High-frequency intensities and the quantitative problem}

On each edge \(e\), an eigenfunction with frequency \(k\) can be written as
\[
u_e(x)=A_e\cos(kx)+B_e\sin(kx).
\]
We associate with \(u_e\) the homogenized edge intensity
\begin{equation}
q_e:=\frac{|A_e|^2+|B_e|^2}{2}.
\label{eq:intro-q}
\end{equation}
For a measurable observation set \(\omega\subset G\), let
\[
\alpha_e:=|\omega\cap e|.
\]
Oscillatory averaging shows that, in the high-frequency limit, the
\(L^2\)-mass observed on \(\omega\) is represented by the linear functional
\(q\mapsto\alpha\cdot q\). This motivates the attainable intensity set
\[
I(\ell)\subset
P_\ell:=
\{q\in\mathbb{R}_{\geq 0}^{E}:\ell\cdot q=1\}.
\]
Writing \(C_\infty(\omega;\ell)\) for the uniform high-frequency localization
constant, \cref{thm:exact-variational-principle} yields the exact
finite-dimensional variational formula
\begin{equation}
C_\infty(\omega;\ell)
=
\min_{q\in I(\ell)}\alpha\cdot q.
\label{eq:intro-variational}
\end{equation}
Thus, the high-frequency localization problem is reduced to the geometry of
the attainable intensity set \(I(\ell)\).

In the generic standard-Kirchhoff setting, the description of semiclassical
measures due to Colin de Verdi\`ere \cite{CdV2015} identifies \(I(\ell)\) with
the closure of the normalized regular Gauss image of the secular determinant
manifold; see \cref{cor:generic-gauss-characterization}. Its minimal-support
directions are represented by simple cycles and exterior-to-exterior paths,
which we refer to as primitive supports. These supports determine precisely
whether the localization constant vanishes: in the same generic setting,
\cref{cor:primitive-positivity} gives
\begin{equation}
C_\infty(\omega;\ell)>0
\quad\Longleftrightarrow\quad
|\omega\cap P|>0
\quad\text{for every primitive support }P.
\label{eq:intro-primitive-positivity}
\end{equation}
They do not, however, determine the positive value of
\(C_\infty(\omega;\ell)\) in general. Indeed, the parallel-edge examples of
\cref{sec:primitive-supports} exhibit regular full-support intensities whose
observation ratios are strictly smaller than the corresponding ratios on
every primitive support.

\subsection{Boundary-aware singular completion}

The preceding examples show that primitive supports alone do not determine
the positive value of the high-frequency localization constant. This
motivates a larger comparison class arising from singular points of the
secular problem. At such points, the admissible edge phases depend on the
exterior boundary conditions: internal--internal and Dirichlet-terminal
nodal modes occur at integer multiples of \(\pi\), whereas Neumann-terminal
nodal modes occur at half-integer multiples of \(\pi\). For each
boundary-compatible signed Kirchhoff sector \(\sigma\), the corresponding
singular amplitude vector \(A\) satisfies
\[
B_{\sigma,\mathrm{bc}}A=0.
\]
We therefore define the boundary-aware singular-completion cone by
\begin{equation}
\mathcal M_G^{\mathrm{bc}}
:=
\operatorname{cone}
\bigcup_{\sigma\in\Sigma_{\mathrm{bc}}}
\left\{
A^{\circ 2}:A\in\ker B_{\sigma,\mathrm{bc}}
\right\},
\label{eq:intro-singular-cone}
\end{equation}
where \(\Sigma_{\mathrm{bc}}\) denotes the finite family of
boundary-compatible singular sectors and \(A^{\circ 2}\) denotes the
componentwise squared amplitude vector. In particular, the energy vectors
associated with primitive cycle and exterior-to-exterior scars belong to
\(\mathcal M_G^{\mathrm{bc}}\).

For a measurable observation set \(\omega\subset G\), with
\(\alpha_e=|\omega\cap e|\), we associate with
\(\mathcal M_G^{\mathrm{bc}}\) the singular-completion variational constant
\begin{equation}
C_{\mathrm{sing}}(\omega;\ell)
:=
\min_{\substack{
s\in\mathcal M_G^{\mathrm{bc}}\\
\ell\cdot s=1}}
\alpha\cdot s.
\label{eq:intro-singular-value}
\end{equation}
The structural question underlying this relaxation is whether every regular
scalar secular edge-energy vector admits such a singular completion. Writing
\begin{equation}
\mathcal R_G^{\mathrm{reg}}
:=
\operatorname{cone}
\left\{
\bigl(|A_e|^2+|B_e|^2\bigr)_{e\in E}:
(A,B)\ \text{arises from a regular scalar secular state}
\right\},
\label{eq:intro-regular-energy-set}
\end{equation}
this amounts to asking whether
\begin{equation}
\mathcal R_G^{\mathrm{reg}}
\stackrel{?}{\subseteq}
\mathcal M_G^{\mathrm{bc}}.
\label{eq:intro-central-inclusion}
\end{equation}
For an individual regular state, its edge-energy vector
\(r=(r_e)_{e\in E}\) satisfies
\[
r_e=|A_e|^2+|B_e|^2=2q_e,
\]
where \(q\) is the corresponding homogenized intensity.

If \eqref{eq:intro-central-inclusion} holds, every linear observation
functional evaluated on a regular secular energy vector is bounded below by
the corresponding singular-completion variational problem. The inclusion
itself, however, is a statement about secular geometry and does not by itself
identify the singular-completion value with the fixed-metric high-frequency
constant \(C_\infty(\omega;\ell)\); the latter additionally requires an
accessibility argument for the relevant regular secular states.

\subsection{Low-complexity sufficiency and the six-terminal obstruction}

The first part of the structural theory is positive.
Section~\ref{sec:low-complexity-sufficiency} proves
\eqref{eq:intro-central-inclusion} for three classes:
trees, unicyclic graphs, and closed cycle-rank-two graphs.  The common mechanism is a
positive-semidefinite covariance decomposition.  Local Kirchhoff data can be
realized by zero-sum vector flows on a cut tree, and sufficiently low
cycle--boundary complexity forces those vector flows to decompose into
boundary-compatible parity sectors.

To test the limit of this mechanism in the presence of exterior flux, we
cut the two cycles of a particular cycle-rank-two graph and obtain a crossed
six-terminal tree.  At the covariance
level, an exact rational separator distinguishes a regular positive
semidefinite covariance from every parity-supported covariance.  Since a
covariance witness need not be generated by one scalar quantum-graph state,
this separation alone is not enough.  We therefore reconstruct the scalar
graph, solve its nonlinear endpoint matching equations, and validate the
solution by a Krawczyk enclosure.  The resulting phase point is nonsingular,
its secular kernel is one dimensional, and its physical seven-edge energy
\(r^\star\) satisfies
\begin{equation}
\widehat d\cdot r^\star<0,
\label{eq:intro-regular-negative}
\end{equation}
where the exact physical separator is
\begin{equation}
\widehat d
=
\left(
\frac{117}{25},
\frac{106}{25},
\frac35,
\frac{109}{50},
\frac{33}{20},
-\frac{149}{100},
-\frac{13}{100}
\right).
\label{eq:intro-dhat}
\end{equation}
By contrast,
\begin{equation}
\widehat d\cdot s>0
\qquad
\forall\,0\ne s\in\mathcal M_G^{\mathrm{bc}}.
\label{eq:intro-singular-positive}
\end{equation}
Consequently,
\[
\mathcal R_G^{\rm reg}
\not\subseteq
\mathcal M_G^{\mathrm{bc}}
\]
for this graph.  This is an actual scalar secular obstruction, not merely a
failure of a covariance relaxation.

\subsection{Main theorem: fixed-metric breakdown}

A separated secular point is still not a high-frequency counterexample:
\(C_\infty\) is defined using exact eigenfunctions of one fixed metric at
frequencies tending to infinity.  Section~\ref{sec:fixed-metric-counterexample}
closes this final gap.
Choose a positive \(2\pi\)-lattice lift
$
\ell=\vartheta+2\pi N
$
of a separated regular phase point \(\vartheta\).  Then, \(k=1\) is already an
exact regular eigenfrequency for the fixed metric \(\ell\).  The crucial
radial identity is
\begin{equation}
\phi^*
\frac{d}{dk}M(k\ell)\bigg|_{k=1}\phi
=
\sum_{e\in E}\ell_e r_e
>0.
\label{eq:intro-radial-identity}
\end{equation}
Hence, the simple secular zero is transverse to the fixed radial phase flow.
Recurrence in the compact orbit closure of
\(k\mapsto k\ell\bmod2\pi\), followed by a one-dimensional transverse
correction, produces exact eigenfrequencies \(k_n\to\infty\) with
$
k_n\ell\bmod2\pi\longrightarrow\vartheta.
$
The corresponding normalized intensities converge to the separated regular
intensity.
This yields the main negative theorem of the paper.
\begin{theorem}[Fixed-metric breakdown of singular completion]
\label{thm:intro-main-breakdown}
There exist a compact cycle-rank-two metric graph \(G\), with standard
Kirchhoff conditions at its internal vertices and Dirichlet conditions at two
exterior endpoints, a fixed metric
\(\ell\in(0,\infty)^7\), and a measurable observation set
\(\omega\subset G\) such that
\begin{equation}
ed{
C_\infty(\omega;\ell)
<
C_{\rm sing}(\omega;\ell).
}
\label{eq:intro-main-gap}
\end{equation}
More precisely, one may choose the observation lengths in the form
\begin{equation}
\alpha
=
\frac12\ell+\varepsilon\widehat d,
\qquad \varepsilon>0,
\label{eq:intro-alpha}
\end{equation}
with \(0<\alpha_e<\ell_e\) for every edge, so that
\begin{equation}
C_\infty(\omega;\ell)
<
\frac12
<
C_{\rm sing}(\omega;\ell).
\label{eq:intro-half-gap}
\end{equation}
Moreover, the upper bound on \(C_\infty\) is realized asymptotically by an
exact eigensequence \(k_n\to\infty\) whose regular intensities converge to the
validated separated secular intensity.
\end{theorem}

\begin{proof}
This is \cref{thm:fixed-metric-breakdown}; the construction and recurrence
argument are given in Section~\ref{sec:fixed-metric-counterexample}.
\end{proof}

The theorem gives a precise limit to the singular-completion principle:
singular completion is a powerful quantitative enlargement of primitive
scars and is sufficient throughout the low-complexity regimes proved here,
but it is not a universal variational description of uniform localization on
compact quantum graphs.

\subsection{Contributions}

The contributions of this paper can be listed as follows. 

\paragraph{a. Exact variational reduction.}
For every fixed metric and measurable observation set, the high-frequency
constant is the minimum of the edgewise observation functional over the
attainable intensity set; see
\cref{thm:exact-variational-principle}.  In particular, the observation set
enters the asymptotic problem only through its edgewise measures.

\paragraph{b. Positivity is simpler than the positive value.}
Primitive cycles and exterior-to-exterior paths determine the positivity
threshold in the generic secular setting, but explicit parallel-edge graphs
show that the positive numerical value can be strictly smaller than every
primitive-support benchmark.

\paragraph{c. Boundary-aware singular completion.}
The singular cone
\(\mathcal M_G^{\mathrm{bc}}\) incorporates the correct Dirichlet/Neumann
phase classes, is invariant under switching, contains all
boundary-compatible primitive scars, and admits generalized-eigenvalue and
dual-semidefinite formulations.

\paragraph{d. Low-complexity sufficiency.}
By \cref{cor:low-complexity-singular-sufficiency}, every regular scalar
secular edge-energy vector belongs to the singular-completion cone for trees, unicyclic graphs, and closed cycle-rank-two graphs.  This is a geometric inclusion theorem.  An
identity with \(C_\infty\) additionally requires the relevant singular
minimizer to be accessible along the fixed metric flow.

\paragraph{e. A physically realizable six-terminal escape.}
The six-terminal covariance obstruction is upgraded in Section~7 to a
validated scalar regular secular point.  The nine-coordinate cut-tree
separator is pushed down exactly to the seven physical edges as
\(\widehat d\), eliminating any distinction between the finite-dimensional
certificate and the physical observation functional.

\paragraph{f. Fixed-metric high-frequency realization.}
The radial identity \eqref{eq:intro-radial-identity} and compact
orbit-closure recurrence turn the separated regular point into an exact
high-frequency eigensequence for one fixed metric, yielding
\eqref{eq:intro-main-gap}.

\subsection{Related work}

The problem considered here belongs to the spectral and semiclassical analysis
of compact quantum graphs.  For background on self-adjoint realizations,
vertex conditions, and the basic spectral formalism, we refer to
\cite{KostrykinSchrader1999,Kuchment2004,BerkolaikoKuchment2013}; the
periodic-orbit and scattering formulations that underlie much of the quantum-
graph literature were developed in particular in
\cite{KottosSmilansky1999,GnutzmannSmilansky2006}.  Our asymptotic parameter,
however, is the eigenfrequency of one fixed compact metric graph, rather than
the number of edges or a varying family of graphs.  This fixed-system
viewpoint is what makes the recurrence argument in Section~\ref{sec:fixed-metric-counterexample}
essential: a separated regular phase must be approached by exact
eigenfrequencies of the same metric.

The geometric starting point is the description of semiclassical measures by
Colin de Verdi\`ere \cite{CdV2015}, where the regular part of the determinant
manifold and its Gau\ss{} map encode the attainable edge intensities.  The
regular secular geometry is closely related to generic simplicity and
non-vanishing questions for quantum-graph eigenfunctions; see
\cite{BerkolaikoLiu2017,PlumerTaufer2021}.  These works help distinguish the
regular secular states used here from eigenfunctions tied to singular or
special phase configurations.  Our use of the determinant-manifold viewpoint
is quantitative: we minimize an arbitrary edgewise observation functional
over the full attainable intensity set.  Primitive cycles and
exterior-to-exterior paths therefore determine the zero-versus-positive
threshold, but need not determine the value of the constant once it is
positive.

Localization and scarring on quantum graphs have been studied from several
other perspectives.  Berkolaiko and Winn \cite{BerkolaikoWinn2018} construct
maximally scarred eigenfunctions for particular non-Kirchhoff scattering
matrices on star graphs.  Harrell and Maltsev
\cite{HarrellMaltsev2018,HarrellMaltsev2020} develop Agmon and landscape
methods for localization of Schr\"odinger eigenfunctions on quantum graphs,
while Kravitz et al. \cite{KravitzBrioCaputo2023} analyze exact
resonant localized eigenvectors on tuned metric graphs.  Harrell and Maltsev
\cite{HarrellMaltsev2024} further relate topological bound states to
singularities of secular formulations.  The present problem is different in
that the observation set is fixed and measurable, and the object of interest
is the least asymptotic mass carried by exact high-frequency eigenfunctions
of one fixed graph.

A complementary literature concerns delocalization and eigenfunction
statistics for large graphs.  Quantum-ergodic behavior and eigenfunction
statistics were investigated in \cite{GnutzmannKeatingPiotet2008,GnutzmannKeatingPiotet2010},
and rigorous quantum-ergodicity results for large equilateral and expanding
quantum graphs were established in
\cite{AnantharamanIngremeauSabriWinn2021,IngremeauSabriWinn2020}.  Those
results concern graph-sequence limits and spatial equidistribution under
suitable hypotheses.  They are therefore complementary to the fixed-graph
optimization problem considered here, where nontrivial limiting intensity
geometry can persist at arbitrarily high frequency.

There is also a natural connection with observability and control.  Egidi et al. \cite{EgidiMugnoloSeelmann2024} obtain
Logvinenko--Sereda-type estimates and applications to quantum graphs with
measurable distributed control sets, while Duca \cite{Duca2020} studies exact
controllability for bilinear Schr\"odinger dynamics on compact graphs.  More
recently, Ammari et al. 
\cite{AmmariDucaJolyLeBalch2025} introduced the graph geometric control
condition (GGCC): for the wave equation with internal controls and Dirichlet
exterior conditions it gives a necessary and sufficient criterion, whereas
for the Schr\"odinger equation and mixed boundary conditions it is sufficient
but not necessary.  The relation to the present work is primarily
qualitative.  Both settings reveal cycle and path mechanisms that permit
high-frequency mass to avoid an observation region, but here the objective is
the exact asymptotic observation constant and the geometry of the intensity
set that determines it.



\section{High-frequency intensity geometry and the variational principle}
\label{sec:variational}

\subsection{Global conventions}
\label{subsec:global-conventions}

We use three related but distinct edgewise objects throughout the paper.
For an edge expansion
\[
u_e(x)=A_e\cos(kx)+B_e\sin(kx),
\]
the \emph{homogenized intensity} is
\[
q_e:=\frac{|A_e|^2+|B_e|^2}{2},
\]
whereas the \emph{unnormalized edge energy} is
\[
r_e:=|A_e|^2+|B_e|^2=2q_e.
\]
The high-frequency variational problem is stated in \(q\)-coordinates and
uses the normalization \(\ell\cdot q=1\).  The structural secular and
singular-cone arguments are stated in homogeneous \(r\)-coordinates.
Because all cone comparisons and Rayleigh quotients are homogeneous, the
factor \(2\) never changes a separator sign or a normalized variational
value.

For an observation set \(\omega\), we write
\[
\alpha_e:=|\omega\cap e|,
\qquad
D_\ell:=\operatorname{diag}(\ell_e),
\qquad
W_\alpha:=\operatorname{diag}(\alpha_e).
\]
The notation \(I(\ell)\) is reserved for attainable normalized
high-frequency intensities \(q\); \(\mathcal R_G^{\rm reg}\) is reserved for
unnormalized regular scalar secular edge energies \(r\);
\(\mathcal M_G^{\rm bc}\) denotes the boundary-aware singular-completion
cone in the same \(r\)-coordinates.  Covariance matrices \(Q\) and their
tree measurements \(H_T(Q)\) are used only in the cut-tree relaxation of
Sections~6--7.

Throughout this section, \(G=(V,E)\) is a finite compact metric graph with edge
lengths \(\ell=(\ell_e)_{e\in E}\), and \(-\Delta_{G,\ell}\) is a self-adjoint
metric-graph Laplacian with standard Kirchhoff conditions at interior vertices
and fixed self-adjoint separated conditions at exterior vertices. The
variational statements below use only the one-dimensional form of an
eigenfunction on each edge and therefore apply to this general boundary
setting. The generic Gauss-map corollary at the end of the section is stated
separately under the standard Kirchhoff hypotheses of Colin de Verdi\`ere \cite{CdV2015}.

\subsection{Edge intensities and oscillatory averaging}

Let \(u\) be an eigenfunction with eigenvalue \(k^2>0\). After orienting each
edge and identifying it with \([0,\ell_e]\), we write
\begin{equation}
u_e(x)=A_e\cos(kx)+B_e\sin(kx),
\qquad 0\le x\le\ell_e,
\label{eq:edge-expansion}
\end{equation}
with \(A_e,B_e\in\mathbb C\). The quantity in the following definition is
invariant under a change of the origin on the edge, since such a change rotates
the coefficient pair \((A_e,B_e)\) by a real orthogonal matrix.

\begin{definition}[Homogenized edge intensity]
For an eigenfunction \(u\) of frequency \(k>0\), define
\begin{equation}
q_e(u):=\frac{|A_e|^2+|B_e|^2}{2},
\qquad
q(u):=(q_e(u))_{e\in E}\in[0,\infty)^E.
\label{eq:edge-intensity}
\end{equation}
\end{definition}
The next elementary estimate is the normalization input needed later. It is
important that its conclusion is uniform over all normalized eigenfunctions at
a fixed large frequency.

\begin{lemma}[Uniform coefficient bounds]
\label{lem:uniform-coefficient-bounds}
Let \(u\) satisfy
\(-\Delta_{G,\ell}u=k^2u\) and \(\|u\|_{L^2(G)}=1\).
Then, for every edge \(e\),
\begin{equation}
\int_0^{\ell_e}|u_e(x)|^2\,dx
=
\ell_e q_e(u)+R_e(u,k),
\qquad
|R_e(u,k)|\le \frac{2}{k}q_e(u).
\label{eq:edge-mass-estimate}
\end{equation}
Consequently, if \(k\ge4/\ell_{\min}\), where
\(\ell_{\min}:=\min_{e\in E}\ell_e\), then
\begin{equation}
0\le q_e(u)\le \frac{2}{\ell_e}
\qquad(e\in E),
\label{eq:qe-bound}
\end{equation}
and
\begin{equation}
\sum_{e\in E}\ell_e q_e(u)=1+O(k^{-1}),
\label{eq:normalization-asymptotic}
\end{equation}
where the implicit constant depends only on \(G\) and \(\ell\).
\end{lemma}

\begin{proof}
From \eqref{eq:edge-expansion},
\[
|u_e(x)|^2
=
q_e(u)
+
\frac{|A_e|^2-|B_e|^2}{2}\cos(2kx)
+
\operatorname{Re}(A_e\overline{B_e})\sin(2kx).
\]
The two oscillatory integrals over \([0,\ell_e]\) have absolute values at most
\(1/(2k)\) and \(1/k\), respectively. Moreover,
\[
\frac{\bigl||A_e|^2-|B_e|^2\bigr|}{2}
\le q_e(u),
\qquad
|\operatorname{Re}(A_e\overline{B_e})|
\le q_e(u).
\]
This gives \eqref{eq:edge-mass-estimate} with the stated slightly non-optimal
constant. For \(k\ge4/\ell_{\min}\),
\[
\int_0^{\ell_e}|u_e|^2
\ge
\left(\ell_e-\frac{2}{k}\right)q_e(u)
\ge
\frac{\ell_e}{2}q_e(u).
\]
Since the left-hand side is at most \(1\), we obtain
\eqref{eq:qe-bound}. Summing \eqref{eq:edge-mass-estimate} over the finite edge
set and using \eqref{eq:qe-bound} gives
\eqref{eq:normalization-asymptotic}.
\end{proof}

\begin{lemma}[Averaging on arbitrary measurable edge sets]
\label{lem:measurable-averaging}
Let \(W_e\subset[0,\ell_e]\) be measurable and
\(\alpha_e:=|W_e|\). If \(u_n\) is any sequence of normalized eigenfunctions
with frequencies \(k_n\to\infty\), then
\begin{equation}
\int_{W_e}|u_{n,e}(x)|^2\,dx
-
\alpha_e q_e(u_n)
\longrightarrow 0
\qquad(e\in E).
\label{eq:measurable-averaging}
\end{equation}
The convergence is uniform with respect to the choice of normalized
eigenfunction at frequency \(k_n\).
\end{lemma}

\begin{proof}
Integrating the oscillatory expansion from the proof of
\cref{lem:uniform-coefficient-bounds} over \(W_e\) gives
\[
\int_{W_e}|u_{n,e}|^2
=
\alpha_eq_e(u_n)
+
\frac{|A_{n,e}|^2-|B_{n,e}|^2}{2}
\int_{W_e}\cos(2k_nx)\,dx
+
\operatorname{Re}(A_{n,e}\overline{B_{n,e}})
\int_{W_e}\sin(2k_nx)\,dx.
\]
Because \(\mathbf1_{W_e}\in L^1(0,\ell_e)\), the Riemann--Lebesgue lemma
implies that both Fourier integrals tend to zero. By
\cref{lem:uniform-coefficient-bounds}, for sufficiently large \(n\) the
coefficients in front of those integrals are bounded uniformly over the
normalized eigenspace. This proves \eqref{eq:measurable-averaging}.
\end{proof}
For an observation set \(\omega\subset G\), write
\begin{equation}
\alpha_e:=|\omega\cap e|,
\qquad
\alpha:=(\alpha_e)_{e\in E}.
\label{eq:alpha-vector}
\end{equation}
Summing \cref{lem:measurable-averaging} over the finite edge set yields
\begin{equation}
\int_\omega |u_n|^2
=
\alpha\cdot q(u_n)+o(1)
\qquad(k_n\to\infty).
\label{eq:global-averaging}
\end{equation}
Notice that no openness or regularity of \(\omega\) is used here;
measurability is enough.

\subsection{The attainable intensity set}

Define the weighted probability simplex
\begin{equation}
P_\ell
:=
\left\{
q\in[0,\infty)^E:
\sum_{e\in E}\ell_eq_e=1
\right\}.
\label{eq:weighted-simplex}
\end{equation}

\begin{definition}[Attainable high-frequency intensity]
A vector \(q\in[0,\infty)^E\) is attainable for the metric \(\ell\) if there
exist normalized exact eigenfunctions \(u_n\) and frequencies \(k_n\to\infty\)
such that
\[
-\Delta_{G,\ell}u_n=k_n^2u_n,
\qquad
\|u_n\|_{L^2(G)}=1,
\qquad
q(u_n)\longrightarrow q.
\]
The set of all attainable vectors is denoted by
$
I(\ell).
$
\end{definition}

\begin{lemma}[Compactness of the attainable set]
\label{lem:compact-attainable}
The set \(I(\ell)\) is a nonempty compact subset of \(P_\ell\).
\end{lemma}

\begin{proof}
The spectrum of a compact metric graph is discrete and unbounded, so there
exists a sequence of normalized eigenfunctions with frequencies tending to
infinity. By \cref{lem:uniform-coefficient-bounds}, the corresponding
intensity vectors are eventually contained in a fixed compact rectangle in
\(\mathbb R^E\); hence at least one convergent subsequence exists and
\(I(\ell)\ne\varnothing\).
If \(q\in I(\ell)\) is realized by \(u_n\) with \(k_n\to\infty\), then
\eqref{eq:normalization-asymptotic} and \(q(u_n)\to q\) imply
$
\sum_e\ell_eq_e=1.
$
Thus, \(I(\ell)\subset P_\ell\).

For closedness, let \(q^{(m)}\in I(\ell)\) and suppose \(q^{(m)}\to q\).
For each \(m\), choose a normalized exact eigenfunction \(u_m\) with frequency
\(k_m\ge m\) such that
\[
\|q(u_m)-q^{(m)}\|\le \frac1m.
\]
Then, \(k_m\to\infty\) and \(q(u_m)\to q\), so \(q\in I(\ell)\).
Hence, \(I(\ell)\) is closed, and therefore, compact as a closed subset of the
compact simplex \(P_\ell\).
\end{proof}

\subsection{Exact variational principle}

For \(\Lambda>0\), set
\begin{equation}
m_\omega(\Lambda)
:=
\inf_{\substack{
-\Delta_{G,\ell}u=\lambda u\\
\lambda\ge\Lambda,\ \|u\|_{L^2(G)}=1}}
\int_\omega |u|^2.
\label{eq:momega}
\end{equation}
The function \(m_\omega\) is nondecreasing and bounded between \(0\) and \(1\);
therefore,
\[
C_\infty(\omega;\ell)
=
\lim_{\Lambda\to\infty}m_\omega(\Lambda)
\]
always exists.

\begin{theorem}[Exact high-frequency variational principle]
\label{thm:exact-variational-principle}
Let \(\omega\subset G\) be measurable and let \(\alpha\) be given by
\eqref{eq:alpha-vector}. Then,
\begin{equation}
C_\infty(\omega;\ell)
=
\min_{q\in I(\ell)}\alpha\cdot q.
\label{eq:exact-variational-principle}
\end{equation}
In particular, the high-frequency localization constant depends on
\(\omega\) only through the edgewise measures
\((|\omega\cap e|)_{e\in E}\).
\end{theorem}

\begin{proof}
Because \(I(\ell)\) is compact by \cref{lem:compact-attainable}, the minimum
on the right-hand side is attained.
For the upper bound, fix \(q\in I(\ell)\) and let \(u_n\) be an exact
normalized eigensequence realizing it, with frequencies \(k_n\to\infty\).
By \eqref{eq:global-averaging},
\[
\int_\omega |u_n|^2\longrightarrow \alpha\cdot q.
\]
For every fixed \(\Lambda\), all sufficiently large \(n\) satisfy
\(k_n^2\ge\Lambda\), whence
\(m_\omega(\Lambda)\le\int_\omega|u_n|^2\).
Passing first to \(n\to\infty\) and then to \(\Lambda\to\infty\) gives
\[
C_\infty(\omega;\ell)\le\alpha\cdot q.
\]
Since \(q\) was arbitrary,
\[
C_\infty(\omega;\ell)
\le
\min_{q\in I(\ell)}\alpha\cdot q.
\]
For the reverse bound, choose \(\Lambda_n\to\infty\) and normalized exact
eigenfunctions \(u_n\), with eigenvalues \(k_n^2\ge\Lambda_n\), such that
\[
\int_\omega|u_n|^2
\le
m_\omega(\Lambda_n)+\frac1n.
\]
Then, \(k_n\to\infty\). By
\cref{lem:uniform-coefficient-bounds}, after extraction we may assume
\(q(u_n)\to q\) for some \(q\in I(\ell)\). Applying
\eqref{eq:global-averaging} along this subsequence gives
\[
\int_\omega|u_n|^2\longrightarrow\alpha\cdot q.
\]
On the other hand, \(m_\omega(\Lambda_n)\to C_\infty(\omega;\ell)\), and the
defining lower bound together with the almost-minimizing inequality implies
\[
\int_\omega|u_n|^2\longrightarrow C_\infty(\omega;\ell).
\]
Hence,
\[
C_\infty(\omega;\ell)
=
\alpha\cdot q
\ge
\min_{p\in I(\ell)}\alpha\cdot p.
\]
This proves \eqref{eq:exact-variational-principle}.
\end{proof}

\begin{corollary}[Structure of the observation functional]
\label{cor:observation-functional}
Define, for \(\alpha\in[0,\infty)^E\),
\[
F_\ell(\alpha):=
\min_{q\in I(\ell)}\alpha\cdot q.
\]
Then, \(F_\ell\) is positively homogeneous, concave, coordinatewise
nondecreasing, and globally Lipschitz on \([0,\infty)^E\). More precisely,
\begin{equation}
|F_\ell(\alpha)-F_\ell(\beta)|
\le
\frac1{\ell_{\min}}\|\alpha-\beta\|_\infty.
\label{eq:Fl-lipschitz}
\end{equation}
For every measurable observation set \(\omega\) with edgewise measure vector
\(\alpha\), \(F_\ell(\alpha)=C_\infty(\omega;\ell)\).
\end{corollary}

\begin{proof}
Positive homogeneity is immediate. Since \(F_\ell\) is the pointwise minimum
of linear functions of \(\alpha\), it is concave. Since every
\(q\in I(\ell)\) has nonnegative coordinates, it is coordinatewise
nondecreasing. Finally, every \(q\in P_\ell\) satisfies
\[
\sum_eq_e
\le
\frac1{\ell_{\min}}
\sum_e\ell_eq_e
=
\frac1{\ell_{\min}}.
\]
Therefore,
\[
|(\alpha-\beta)\cdot q|
\le
\|\alpha-\beta\|_\infty\sum_eq_e
\le
\frac1{\ell_{\min}}\|\alpha-\beta\|_\infty.
\]
Taking minima gives \eqref{eq:Fl-lipschitz}.
\end{proof}

\begin{remark}[Convexification]
\label{rem:convexification}
Only linear functionals of \(I(\ell)\) occur in
\eqref{eq:exact-variational-principle}. Hence, if
$
K(\ell):=\operatorname{conv}I(\ell),
$
then
$
C_\infty(\omega;\ell)
=
\min_{q\in K(\ell)}\alpha\cdot q.
$
Thus, the exact attainable set is the spectral object, while its convex hull is
the natural optimization object. This distinction becomes essential in the
singular-completion theory.
\end{remark}

\subsection{Generic regular Gauss image in the standard Kirchhoff setting}

We next relate the attainable intensity set to the determinant-manifold
description of semiclassical measures. The statement below is a direct
consequence of the theorem of Colin de Verdi`ere \cite{CdV2015}, specialized
to the present normalization.

Throughout this subsection, assume that the standard Kirchhoff condition is
imposed at every vertex, including vertices of degree one. Let
\(Z_G\subset\mathbb T^E\) denote the determinant manifold and
\(Z_G^{\mathrm{reg}}\) its regular part. The Gauss map associates with each
\(z\in Z_G^{\mathrm{reg}}\) a positive ray,
$$
\Gamma:
Z_G^{\mathrm{reg}}
\longrightarrow
\bigl([0,\infty)^E\setminus\{0\}\bigr)/\mathbb R_+.
$$
For \(z\in Z_G^{\mathrm{reg}}\), choose any representative
\(g(z)\in[0,\infty)^E\setminus\{0\}\) of the ray \(\Gamma(z)\). For a fixed
metric \(\ell\), define its normalized Gauss image by
\begin{equation}
\widehat\Gamma_\ell(z)
:=
\frac{g(z)}{\ell\cdot g(z)}
\in P_\ell.
\label{eq:normalized-gauss}
\end{equation}
This definition is independent of the choice of positive representative
\(g(z)\), since the normalization is invariant under multiplication by a
positive scalar.

Following Colin de Verdi\`ere \cite{CdV2015}, let \(\mathcal G_G\) denote the generic class of
metrics whose components are rationally independent and whose phase line does
not meet the singular part of the determinant manifold; equivalently, in that
setting the Laplace spectrum is simple.

\begin{corollary}[Generic Gauss-map characterization]
\label{cor:generic-gauss-characterization}
Assume the standard Kirchhoff setting of the preceding paragraph and let
\(\ell\in\mathcal G_G\). Then,
\begin{equation}
I(\ell)
=
\overline{\widehat\Gamma_\ell(Z_G^{\mathrm{reg}})}.
\label{eq:generic-gauss-characterization}
\end{equation}
Consequently, for every measurable \(\omega\subset G\),
\begin{equation}
C_\infty(\omega;\ell)
=
\min_{q\in
\overline{\widehat\Gamma_\ell(Z_G^{\mathrm{reg}})}}
\alpha\cdot q.
\label{eq:gauss-variational}
\end{equation}
\end{corollary}

\begin{proof}
For a normalized high-frequency eigensequence \(u_n\) with
\(q(u_n)\to q\), the same expansion used in
\cref{lem:measurable-averaging}, now tested against an arbitrary continuous
function on each edge, gives the weak convergence
\[
|u_n|^2\,dx
\rightharpoonup
\sum_{e\in E}q_e\,dx|_e.
\]
Conversely, if a normalized eigensequence has a semiclassical measure with
constant edge densities \(q_e\), then testing against continuous functions
supported in the interior of a single edge and using the oscillatory expansion
shows that, after extraction, its homogenized intensities converge to the same
vector \(q\). Thus, probability-valued semiclassical measures are in
one-to-one correspondence with \(I(\ell)\) through their constant edge-density
vectors.

For \(\ell\in\mathcal G_G\), Colin de Verdi\`ere's generic theorem \cite{CdV2015} identifies
the set of non-normalized semiclassical measures with the closure of the
regular Gauss image. Normalizing each Gauss ray to total mass one is exactly
the operation in \eqref{eq:normalized-gauss}, because a constant edge-density
vector \(m\) has total mass \(\ell\cdot m\). This yields
\eqref{eq:generic-gauss-characterization}, and
\eqref{eq:gauss-variational} follows from
\cref{thm:exact-variational-principle}.
\end{proof}

\begin{remark}[Scope of the generic characterization]
The exact variational principle of
\cref{thm:exact-variational-principle} does not require a genericity
assumption and applies to the fixed exterior boundary conditions considered
throughout the paper. By contrast,
\cref{cor:generic-gauss-characterization} is stated under the standard
Kirchhoff assumptions for which the determinant-manifold description of
Colin de Verdi\`ere \cite{CdV2015} applies. Accordingly, we use the generic
Gauss-map characterization only in this setting and make no corresponding
claim for mixed Dirichlet--Neumann exterior conditions.
\end{remark}

\begin{remark}[Role of the determinant-manifold description]
The determinant manifold, its Gauss map, and the generic characterization of
semiclassical measures are due to Colin de Verdi\`ere \cite{CdV2015}. The
contribution needed here is the exact measurable-observation variational
principle
\eqref{eq:exact-variational-principle}, which expresses the quantitative
localization constant as an optimization problem over the full attainable
intensity set. Consequently, primitive supports determine the positivity
threshold, whereas the positive numerical value generally depends on the
larger geometry of the attainable intensity set.
\end{remark}

\section{Primitive supports: positivity versus the value of the constant}
\label{sec:primitive-supports}

The variational principle of Section~\ref{sec:variational} separates two
questions that are naturally related but mathematically distinct. The first
is qualitative: can an attainable high-frequency intensity assign zero mass
to the observation set? The second is quantitative: what is the least
observed mass among all attainable high-frequency intensities? In the generic
standard-Kirchhoff setting, the first question is governed by the
minimal-support structure described by Colin de Verdi\`ere
\cite{CdV2015}, whereas the second depends on the geometry of the full
attainable intensity set. This section makes this distinction precise and
exhibits an explicit family for which the two levels of information are
strictly separated.

\subsection{Primitive supports characterize the positivity threshold}
Throughout this subsection, we work under the hypotheses of
\cref{cor:generic-gauss-characterization}: every vertex carries the standard
Kirchhoff condition, and \(\ell\in\mathcal G_G\) is generic in the sense of
Colin de Verdi\`ere \cite{CdV2015}. A \emph{primitive support} is the support
of a minimal semiclassical measure. Under these hypotheses, the primitive
supports of a finite metric graph are precisely the simple cycles and the
simple paths joining two exterior vertices.

For a primitive support \(P\subset G\), we write
\[
|\omega\cap P|
:=
\sum_{e\subset P}|\omega\cap e|
\]
for the total length of the portion of \(P\) contained in the observation
set \(\omega\).

\begin{corollary}[Primitive-support criterion]
\label{cor:primitive-positivity}
Assume that every vertex carries the standard Kirchhoff condition and that
\(\ell\in\mathcal G_G\) is generic in the sense of Colin de Verdi\`ere.
Then
\begin{equation}
C_\infty(\omega;\ell)>0
\quad\Longleftrightarrow\quad
|\omega\cap P|>0
\quad\text{for every primitive support }P.
\label{eq:primitive-positivity}
\end{equation}
Equivalently, the observation set has positive measure on every simple cycle
and every simple path joining two exterior vertices.
\end{corollary}

\begin{proof}
By \cref{cor:generic-gauss-characterization},
\[
C_\infty(\omega;\ell)
=
\min_{q\in
\overline{\widehat\Gamma_\ell(Z_G^{\mathrm{reg}})}}
\alpha\cdot q.
\]
Suppose first that a primitive support \(P\) satisfies
\(|\omega\cap P|=0\). The corresponding minimal semiclassical measure belongs
to the closure of the regular Gauss image, and its edge-density vector
\(q^P\) is supported on \(P\). Since \(\alpha_e=0\) on every edge of \(P\) on
which \(q_e^P>0\), we have \(\alpha\cdot q^P=0\). Hence,
\(C_\infty(\omega;\ell)=0\).

Conversely, assume \(C_\infty(\omega;\ell)=0\). Compactness of the attainable
intensity set and \cref{thm:exact-variational-principle} give an attainable
limit \(q\) with \(\alpha\cdot q=0\). Since both vectors are nonnegative,
every edge on which \(q_e>0\) has \(\alpha_e=0\).
Colin de Verdi\`ere's minimal-support theorem \cite{CdV2015} implies that the support of a
nonzero semiclassical measure contains a minimal support \(P\), which is a
simple cycle or a simple exterior-to-exterior path. Thus, \(\alpha_e=0\) on
every edge of \(P\), and therefore, \(|\omega\cap P|=0\).
\end{proof}

\begin{remark}[Positivity versus quantitative minimization]
\label{rem:positivity-not-value}
The criterion \eqref{eq:primitive-positivity} characterizes whether the
minimum of the observation functional is strictly positive. It does not imply
that, once this minimum is positive, it is attained on a primitive support.
Primitive supports therefore determine the zero-versus-positive threshold,
whereas the positive numerical value of \(C_\infty(\omega;\ell)\) depends on
the optimization of the observation functional over the full attainable
intensity set \(I(\ell)\).
\end{remark}

\subsection{The parallel-edge graph and its regular intensity geometry}

Let \(\Theta_m\) be the graph consisting of two vertices \(v_-\) and \(v_+\)
joined by \(m\ge3\) parallel edges of lengths
\(\ell_1,\ldots,\ell_m>0\), with standard Kirchhoff conditions at both
vertices. We first describe the closure of its regular Gauss intensities.
Let \(u\) be a real regular secular state at wave number \(k>0\), and orient
every edge from \(v_-\) to \(v_+\). Put
\[
\phi:=u(v_-),
\qquad
p_j:=k^{-1}\partial_\nu u_j(v_-),
\qquad
j=1,\ldots,m,
\]
where \(\partial_\nu\) denotes the derivative in the chosen outgoing
direction. On edge \(e_j\), using the phase variable \(t=kx\),
\begin{equation}
u_j(t)=\phi\cos t+p_j\sin t.
\label{eq:theta-edge-state}
\end{equation}
The nonoscillatory edge intensity is therefore
\begin{equation}
q_j=\frac{\phi^2+p_j^2}{2}.
\label{eq:theta-edge-intensity}
\end{equation}
The Kirchhoff condition at \(v_-\) gives
\[
\sum_{j=1}^m p_j=0.
\]
With
\[
s:=\frac{\phi}{\sqrt2},
\qquad
x_j:=\frac{p_j}{\sqrt2},
\]
we obtain
\begin{equation}
q_j=s^2+x_j^2,
\qquad
\sum_{j=1}^m x_j=0.
\label{eq:theta-intensity-param}
\end{equation}
The converse is true after taking the closure of the regular secular image.
Indeed, fix \(s\ne0\) and \(x\in\mathbb R^m\) with
\(\sum_jx_j=0\), initially assuming \(x_j\ne0\) for every \(j\). Set
\(\phi=\sqrt2\,s\), \(p_j=\sqrt2\,x_j\), prescribe the same vertex value
\(u(v_+)=\phi\), and choose the edge phase
\begin{equation}
\theta_j
=
2\arctan\!\left(\frac{p_j}{\phi}\right)
\pmod{2\pi}.
\label{eq:theta-phase-choice}
\end{equation}
Then, \eqref{eq:theta-edge-state} takes the value \(\phi\) at
\(t=\theta_j\). Its derivative in the chosen orientation equals \(-p_j\)
there, so the outward derivative at \(v_+\) equals \(p_j\). Hence, the
Kirchhoff condition at \(v_+\) is again \(\sum_jp_j=0\). The state is regular
because \(\sin\theta_j\ne0\). Vectors having zero coordinates are obtained by
approximation inside the hyperplane \(\sum_jx_j=0\); the regimes \(s=0\) and
\(x=0\) are obtained as limits.

Consequently, the closure of the normalized regular intensity image is
\begin{equation}
\mathcal Q_m(\ell)
=
\left\{
\frac{(s^2+x_1^2,\ldots,s^2+x_m^2)}
{\sum_{j=1}^m\ell_j(s^2+x_j^2)}
:
(s,x)\ne0,\ 
\sum_{j=1}^m x_j=0
\right\}.
\label{eq:theta-regular-intensities}
\end{equation}

\begin{proposition}[Exact localization formula on \(\Theta_m\)]
\label{prop:theta-m-formula}
Assume that the metric \(\ell\) satisfies the generic hypotheses of
\cref{cor:generic-gauss-characterization}. For an arbitrary measurable
observation set with edgewise measures
\(\alpha=(\alpha_1,\ldots,\alpha_m)\),
\begin{equation}
C_\infty(\alpha;\ell)
=
\min\left\{
\frac{\sum_{j=1}^m\alpha_j}{\sum_{j=1}^m\ell_j},
\ 
\min_{\substack{x\ne0\\ \sum_jx_j=0}}
\frac{\sum_{j=1}^m\alpha_jx_j^2}
{\sum_{j=1}^m\ell_jx_j^2}
\right\}.
\label{eq:theta-m-localization}
\end{equation}
Here, \(C_\infty(\alpha;\ell)\) denotes the common value of
\(C_\infty(\omega;\ell)\) for sets \(\omega\) having the prescribed edgewise
measures \(\alpha\).
\end{proposition}

\begin{proof}
By \cref{cor:generic-gauss-characterization} and
\eqref{eq:theta-regular-intensities}, the high-frequency constant is the
minimum, over \((s,x)\ne0\) with \(\sum_jx_j=0\), of
\[
R(s,x)
=
\frac{s^2\sum_j\alpha_j+\sum_j\alpha_jx_j^2}
{s^2\sum_j\ell_j+\sum_j\ell_jx_j^2}.
\]
Write
\[
A:=\sum_j\alpha_j,
\qquad
L:=\sum_j\ell_j,
\qquad
A_x:=\sum_j\alpha_jx_j^2,
\qquad
L_x:=\sum_j\ell_jx_j^2.
\]
If \(s\ne0\) and \(x\ne0\), then
\[
R(s,x)
=
\frac{s^2L}{s^2L+L_x}\frac{A}{L}
+
\frac{L_x}{s^2L+L_x}\frac{A_x}{L_x}.
\]
Thus, every mixed state gives a weighted average of the uniform quotient
\(A/L\) and a constrained quotient \(A_x/L_x\), so it cannot lie below their
minimum. Conversely, the limits \(x\to0\) and \(s\to0\) in
\eqref{eq:theta-regular-intensities} realize the two branches in the closure.
Minimizing over \(x\in\mathbf1^\perp\setminus\{0\}\) yields
\eqref{eq:theta-m-localization}.
\end{proof}

\begin{remark}[Constrained generalized eigenvalue]
Set
\[
D_\alpha=\operatorname{diag}(\alpha_1,\ldots,\alpha_m),
\qquad
D_\ell=\operatorname{diag}(\ell_1,\ldots,\ell_m).
\]
The second branch of \eqref{eq:theta-m-localization} is the smallest generalized
eigenvalue of \((D_\alpha,D_\ell)\) restricted to
\(\mathbf1^\perp\). Thus, even this elementary graph already replaces a finite
comparison of primitive supports by a genuine constrained spectral
optimization problem.
\end{remark}

\subsection{The three-edge graph: primitive scars do not determine the value}

We now specialize to \(\Theta_3\), write
$
(\ell_1,\ell_2,\ell_3)=(a,b,c),
$
and observe the second and third edges completely. Thus,
$
\alpha=(0,b,c).
$
The primitive supports are the three two-edge cycles. Their observation ratios
are
\[
\frac{b}{a+b},
\qquad
\frac{c}{a+c},
\qquad
1,
\]
so the best primitive-scar benchmark is
\begin{equation}
C_{\mathrm{prim}}
=
\min\left\{
\frac{b}{a+b},
\frac{c}{a+c}
\right\}
>0.
\label{eq:Cprim-theta3}
\end{equation}

\begin{proposition}[Exact \(\Theta_3\) constant and strict primitive-scar gap]
\label{prop:theta3-gap}
Under the generic metric hypothesis,
\begin{equation}
C_\infty
=
\frac{bc}{ab+ac+bc},
\label{eq:theta3-Cinf}
\end{equation}
and
\begin{equation}
C_\infty<C_{\mathrm{prim}}.
\label{eq:theta3-strict-gap}
\end{equation}
A minimizing intensity is proportional to
\begin{equation}
\bigl((b+c)^2,c^2,b^2\bigr),
\label{eq:theta3-min-intensity}
\end{equation}
which has support on all three edges and is therefore not a primitive scar.
\end{proposition}

\begin{proof}
By \cref{prop:theta-m-formula}, the constrained branch is
\begin{equation}
\inf_{x_1+x_2+x_3=0}
\frac{b x_2^2+c x_3^2}
{a x_1^2+b x_2^2+c x_3^2}.
\label{eq:theta3-constrained}
\end{equation}
Set \(x_1=-(x_2+x_3)\) and
\[
N:=b x_2^2+c x_3^2.
\]
Weighted Cauchy--Schwarz gives
\[
(x_2+x_3)^2
\le
\left(\frac1b+\frac1c\right)
(bx_2^2+cx_3^2)
=
\frac{b+c}{bc}N.
\]
Therefore,
\[
\frac{N}{a(x_2+x_3)^2+N}
\ge
\frac{1}{1+a(b+c)/(bc)}
=
\frac{bc}{ab+ac+bc}.
\]
Equality is attained, for example, by
\[
x_2=c,\qquad x_3=b,\qquad x_1=-(b+c),
\]
which yields \eqref{eq:theta3-min-intensity}.
The uniform branch equals
\[
\frac{b+c}{a+b+c},
\]
and
\[
\frac{bc}{ab+ac+bc}
<
\frac{b+c}{a+b+c}
\]
because the difference after clearing positive denominators is
\(a(b^2+bc+c^2)>0\). This proves \eqref{eq:theta3-Cinf}.
Finally,
\[
\frac{bc}{ab+ac+bc}<\frac{b}{a+b},
\qquad
\frac{bc}{ab+ac+bc}<\frac{c}{a+c},
\]
since the two cross-multiplied positive differences are respectively
\(ab\) and \(ac\). Hence, \eqref{eq:theta3-strict-gap} holds.
\end{proof}

\begin{remark}[Regular approximation versus primitive support]
\label{rem:theta3-regular-approx}
The vector in \eqref{eq:theta3-min-intensity} lies on the \(s=0\) boundary
of the regular-intensity family \eqref{eq:theta-regular-intensities}. It is
therefore a nonprimitive singular limit of regular secular intensities, and
its observation ratio can be approached arbitrarily closely by regular
secular states. This is precisely the phenomenon relevant at this stage:
primitive supports still characterize the zero-versus-positive threshold,
but they need not determine the positive value of
\(C_\infty(\omega;\ell)\).

This example should be distinguished from the stronger obstruction developed
later in the paper. Here, the improved observation ratio is obtained only in a
singular limit of regular intensities. In the later construction, by contrast,
a nonsingular regular secular state lies outside the entire boundary-aware
singular-completion cone and yields a strictly smaller observation value than
every boundary-compatible singular completion.
\end{remark}

\subsection{All-but-one observation on \(\Theta_m\)}

The same mechanism persists for every \(m\ge3\). Observe all edges except
\(e_1\), so that
\[
\alpha_1=0,
\qquad
\alpha_j=\ell_j
\quad(j\ge2),
\]
and set
\begin{equation}
H_1:=\sum_{j=2}^m\ell_j^{-1}.
\label{eq:H1}
\end{equation}

\begin{corollary}[All-but-one formula]
\label{cor:theta-m-all-but-one}
Under the generic metric hypothesis,
\begin{equation}
C_\infty
=
\frac{1}{1+\ell_1H_1}.
\label{eq:theta-m-all-but-one}
\end{equation}
Moreover, for \(m\ge3\),
\begin{equation}
C_\infty
<
\min_{2\le j\le m}
\frac{\ell_j}{\ell_1+\ell_j},
\label{eq:theta-m-primitive-gap}
\end{equation}
so the exact constant is strictly smaller than every primitive cycle-scar
benchmark involving the unobserved edge.
\end{corollary}

\begin{proof}
For \(x\in\mathbf1^\perp\), write
\[
N:=\sum_{j=2}^m\ell_jx_j^2,
\qquad
x_1=-\sum_{j=2}^m x_j.
\]
Weighted Cauchy--Schwarz gives
\[
x_1^2
\le
\left(\sum_{j=2}^m\ell_j^{-1}\right)
\left(\sum_{j=2}^m\ell_jx_j^2\right)
=
H_1N,
\]
with equality when \(x_j\propto\ell_j^{-1}\) for \(j\ge2\). Hence, the
constrained branch of \cref{prop:theta-m-formula} equals
\[
\inf
\frac{N}{\ell_1x_1^2+N}
=
\frac{1}{1+\ell_1H_1}.
\]
It is no larger than the uniform branch, and for \(m\ge3\) it is strictly
smaller because
\[
H_1>
\frac{1}{\sum_{j=2}^m\ell_j}.
\]
This proves \eqref{eq:theta-m-all-but-one}. Finally, for each \(j\ge2\),
\[
H_1>\ell_j^{-1},
\]
so
\[
\frac{1}{1+\ell_1H_1}
<
\frac{1}{1+\ell_1/\ell_j}
=
\frac{\ell_j}{\ell_1+\ell_j},
\]
which is \eqref{eq:theta-m-primitive-gap}.
\end{proof}

\begin{remark}[Section-level conclusion]
\label{rem:section3-conclusion}
For generic standard-Kirchhoff graphs, primitive cycles and
exterior-to-exterior paths give an exact zero-versus-positive criterion
through \cref{cor:primitive-positivity}. The \(\Theta_3\) example shows,
however, that the positive value of \(C_\infty(\omega;\ell)\) can lie
strictly below every primitive-support value, and
\cref{cor:theta-m-all-but-one} shows that this discrepancy persists
throughout the corresponding parallel-edge family. Thus, the qualitative
support criterion and the quantitative localization problem encode genuinely
different information.

This motivates enlarging the comparison class from primitive supports to the
larger family generated by boundary-compatible singular Kirchhoff
cancellations. The next section makes this singular-completion construction
algebraically precise. It also separates two logically distinct questions:
whether regular secular edge-energy vectors lie in the resulting singular
cone, and whether the relevant regular secular states are accessible along
the spectral orbit of a single fixed metric.
\end{remark}


\section{Boundary-aware singular completion}
\label{sec:singular-completion}

Section~\ref{sec:primitive-supports} shows that primitive supports characterize
the qualitative zero-versus-positive threshold but need not determine the
positive numerical value of the high-frequency localization constant. This
motivates enlarging the comparison class. The enlargement is not obtained by
simply listing more complicated supports. At a singular secular
configuration, several nodal edge modes may coexist and cancel through the
Kirchhoff conditions, and their squared amplitudes generate edge-energy
vectors that need not arise from primitive scars.

The exterior boundary conditions enter this construction through the
\emph{admissible singular phase classes}. Orient each terminal edge from its
internal endpoint toward the exterior vertex. Its nodal mode contributes a
single signed amplitude to the Kirchhoff equation at the internal endpoint;
no additional Kirchhoff equation is imposed at the exterior endpoint. The
distinction between Dirichlet and Neumann exterior conditions is instead
encoded in the phase class for which the nodal edge mode satisfies the
prescribed boundary condition. We now make this boundary-aware singular
construction precise.

\subsection{Boundary-compatible nodal singular modes}
\label{subsec:boundary-singular-modes}

Fix a frequency \(k>0\).  We describe the singular edge modes after choosing an
orientation on every edge.  Amplitude signs are auxiliary and will later be
shown to have no effect on the resulting intensity cone.

\paragraph{Internal--internal edges.}
Let \(e\) join internal vertices \(v_-\) and \(v_+\), and identify
\(e\) with \([0,\ell_e]\), oriented from \(v_-\) to \(v_+\).  A nodal singular
mode has the form
\begin{equation}
    u_e(x)=A_e\sin(kx),
    \qquad A_e\in\mathbb R.
    \label{eq:internal-internal-singular-mode}
\end{equation}
The condition \(u_e(v_+)=0\) is equivalent to
\begin{equation}
    \theta_e:=k\ell_e\in\pi\mathbb Z.
    \label{eq:internal-internal-singular-phase}
\end{equation}
Write
\[
    \sigma_e:=\cos\theta_e\in\{+1,-1\}.
\]
The normalized outward derivatives are then
\begin{equation}
   k^{-1}\partial_\nu u_e(v_-)=A_e,
   \qquad
   k^{-1}\partial_\nu u_e(v_+)=-\sigma_e A_e.
   \label{eq:internal-internal-fluxes}
\end{equation}
Thus, the corresponding column of the Kirchhoff matrix has the two signed
incidences \(+1\) and \(-\sigma_e\).  If \(e\) is a loop, both incidences occur
in the same row, and the total coefficient is \(1-\sigma_e\).

\paragraph{Internal--Dirichlet edges.}
Let \(e\) join an internal vertex \(v\) to an exterior Dirichlet vertex \(w\),
oriented from \(v\) to \(w\).  Again, take
$
    u_e(x)=A_e\sin(kx).
$
The Dirichlet condition at \(w\) gives
\begin{equation}
    \theta_e=k\ell_e\in\pi\mathbb Z.
    \label{eq:dirichlet-terminal-phase}
\end{equation}
At the only internal endpoint,
\begin{equation}
    k^{-1}\partial_\nu u_e(v)=A_e.
    \label{eq:dirichlet-terminal-flux}
\end{equation}
Hence, the Kirchhoff column contains a single nonzero incidence, which may be
taken to be \(+1\).  The parity of \(\theta_e/\pi\) changes the value of the
mode at the exterior side only by the already imposed zero and therefore does
not create a second Kirchhoff incidence.

\paragraph{Internal--Neumann edges.}
Let \(e\) join an internal vertex \(v\) to an exterior Neumann vertex \(w\),
again oriented from \(v\) to \(w\).  With
$
    u_e(x)=A_e\sin(kx),
$
the Neumann condition at \(w\) is
$
    u'_e(\ell_e)=kA_e\cos(k\ell_e)=0,
$
and hence,
\begin{equation}
    \theta_e=k\ell_e\in \frac{\pi}{2}+\pi\mathbb Z.
    \label{eq:neumann-terminal-phase}
\end{equation}
At the internal endpoint,
\begin{equation}
    k^{-1}\partial_\nu u_e(v)=A_e.
    \label{eq:neumann-terminal-flux}
\end{equation}
Thus, the Kirchhoff column again contains one nonzero incidence, taken to be
\(+1\).  The two phase classes
\(\theta_e\equiv\pi/2,3\pi/2\pmod{2\pi}\) differ only by the sign of the
exterior value and generate the same squared-amplitude geometry.

\begin{remark}[Where the boundary condition enters]
\label{rem:where-bc-enters}
The term \emph{boundary-aware} refers to the admissible singular phase
classes, not to the introduction of an additional Kirchhoff row at an
exterior vertex. For a terminal edge oriented from its internal endpoint
toward the exterior vertex, compatibility with a Dirichlet condition occurs
at phases \(0\) or \(\pi\) modulo \(2\pi\), whereas compatibility with a
Neumann condition occurs at phases \(\pi/2\) or \(3\pi/2\) modulo \(2\pi\).
In either case, the terminal edge contributes only its signed amplitude to
the Kirchhoff equation at the internal endpoint. No additional incidence
constraint is imposed at the exterior endpoint. Thus the exterior boundary
condition determines the admissible singular phase sector rather than adding
a row to the signed Kirchhoff incidence matrix.
\end{remark}

\subsection{Boundary-aware signed incidence sectors}
\label{subsec:signed-incidence-sectors}

Let \(V_{\rm int}\) denote the internal vertices.  A
\emph{boundary-compatible singular sector} \(\sigma\) consists of

\begin{enumerate}[label=(\roman*)]
\item for every internal--internal edge \(e\), a parity
      \(\sigma_e=\cos\theta_e\in\{\pm1\}\) with
      \(\theta_e\in\pi\mathbb Z\);
\item for every internal--Dirichlet edge, a phase class
      \(\theta_e\in\pi\mathbb Z\);
\item for every internal--Neumann edge, a phase class
      \(\theta_e\in\pi/2+\pi\mathbb Z\).
\end{enumerate}
For a fixed representative of the edge orientations, define
\(B_{\sigma,\mathrm{bc}}\in\mathbb R^{V_{\rm int}\times E}\) by summing all
internal endpoint flux coefficients from
\eqref{eq:internal-internal-fluxes},
\eqref{eq:dirichlet-terminal-flux}, and
\eqref{eq:neumann-terminal-flux}.  Explicitly, for a non-loop
internal--internal edge \(e=(v_-,v_+)\),
\begin{equation}
  (B_{\sigma,\mathrm{bc}})_{v_-,e}=1,
  \qquad
  (B_{\sigma,\mathrm{bc}})_{v_+,e}=-\sigma_e,
  \label{eq:B-internal-column}
\end{equation}
and all other entries in that column vanish.  For a loop based at \(v\),
\begin{equation}
  (B_{\sigma,\mathrm{bc}})_{v,e}=1-\sigma_e.
  \label{eq:B-loop-column}
\end{equation}
For either a Dirichlet or Neumann terminal edge incident to an internal
vertex \(v\),
\begin{equation}
  (B_{\sigma,\mathrm{bc}})_{v,e}=1,
  \label{eq:B-terminal-column}
\end{equation}
with all other entries zero.

\begin{proposition}[Kirchhoff kernel characterization]
\label{prop:kirchhoff-kernel-characterization}
Let \(\sigma\) be a boundary-compatible singular sector.  A real amplitude
vector \(A=(A_e)_{e\in E}\) produces a nodal singular state satisfying all
internal Kirchhoff conditions if and only if
\begin{equation}
     B_{\sigma,\mathrm{bc}}A=0.
     \label{eq:singular-kirchhoff-kernel}
\end{equation}
\end{proposition}

\begin{proof}
For each edge, the normalized outward derivatives at internal endpoints are
exactly the coefficients listed in
\eqref{eq:internal-internal-fluxes},
\eqref{eq:dirichlet-terminal-flux}, and
\eqref{eq:neumann-terminal-flux}.  The \(v\)-th component of
\(B_{\sigma,\mathrm{bc}}A\) is therefore
\(k^{-1}\sum_{e\sim v}\partial_\nu u_e(v)\), with loop incidences counted
twice.  Hence, \eqref{eq:singular-kirchhoff-kernel} is precisely the collection
of internal Kirchhoff flux conditions.
\end{proof}

\begin{remark}[Complex amplitudes add no new intensities]
\label{rem:complex-singular-amplitudes}
The matrix \(B_{\sigma,\mathrm{bc}}\) is real.  If
\(A=X+iY\in\ker B_{\sigma,\mathrm{bc}}\), then
\(X,Y\in\ker B_{\sigma,\mathrm{bc}}\), and
$
       |A|^{\circ2}=X^{\circ2}+Y^{\circ2}.
$
Thus, complex singular amplitudes generate no intensity vectors beyond the
conic hull generated by real amplitudes.
\end{remark}

\subsection{Switching invariance}
\label{subsec:switching-invariance}

The signed incidence representation depends on harmless choices: edge
orientation, amplitude sign, and the sign used to write an individual
Kirchhoff equation.  The singular intensity geometry does not.

\begin{proposition}[Switching invariance]
\label{prop:switching-invariance}
Suppose two representatives of the same boundary-compatible singular sector
satisfy
\begin{equation}
      B'_{\sigma,\mathrm{bc}}
      =R\,B_{\sigma,\mathrm{bc}}\,S,
      \label{eq:switching-transform}
\end{equation}
where \(R\) and \(S\) are diagonal matrices with diagonal entries in
\(\{\pm1\}\).  Then
\begin{equation}
 \left\{A^{\circ2}:A\in\ker B'_{\sigma,\mathrm{bc}}\right\}
 =
 \left\{A^{\circ2}:A\in\ker B_{\sigma,\mathrm{bc}}\right\}.
 \label{eq:switching-generator-invariance}
\end{equation}
Consequently the singular-completion cone, its normalized slice, the sector
Rayleigh quotient, and the dual positivity condition are invariant under
switching.
\end{proposition}

\begin{proof}
Since \(R\) is invertible,
\[
   B'_{\sigma,\mathrm{bc}}A=0
   \quad\Longleftrightarrow\quad
   B_{\sigma,\mathrm{bc}}SA=0.
\]
Hence,
$
   \ker B'_{\sigma,\mathrm{bc}}
   =S\,\ker B_{\sigma,\mathrm{bc}},
$
because \(S^{-1}=S\).  Coordinatewise squaring removes the signs:
\((SA)^{\circ2}=A^{\circ2}\), proving
\eqref{eq:switching-generator-invariance}.  Since
\(S D S=D\) for every diagonal matrix \(D\), all diagonal quadratic forms used
below are unchanged as well.
\end{proof}

\begin{remark}[Sector relabeling and intrinsicity]
\label{rem:sector-relabeling}
The choice of edge orientations and phase representatives is auxiliary.
Reversing the orientation of an internal--internal edge may change the
displayed parity representative and the corresponding column signs, while
shifting an admissible terminal phase by \(\pi\) may change the sign of its
edge amplitude. Such transformations only switch or relabel the
boundary-compatible sector representatives and do not change the associated
squared-amplitude vectors. Consequently, the union over all admissible
sectors, and hence the boundary-aware singular-completion cone
\(\mathcal M_G^{\mathrm{bc}}\), is independent of these choices.
\end{remark}

\subsection{The singular-completion cone}
\label{subsec:singular-completion-cone}

For \(A=(A_e)_{e\in E}\), write
\[
      A^{\circ2}:=(A_e^2)_{e\in E}.
\]

\begin{definition}[Boundary-aware singular-completion cone]
\label{def:singular-completion-cone}
Let \(\Sigma_{\mathrm{bc}}\) be the finite collection of
boundary-compatible singular sectors, modulo the switching equivalence of
\cref{prop:switching-invariance}.  Define
\begin{equation}
   \mathcal M_G^{\mathrm{bc}}
   :=
   \operatorname{cone}
   \bigcup_{\sigma\in\Sigma_{\mathrm{bc}}}
   \left\{
       A^{\circ2}:
       A\in\ker B_{\sigma,\mathrm{bc}}
   \right\}.
   \label{eq:singular-completion-cone}
\end{equation}
\end{definition}
Let
\[
      D_\ell:=\operatorname{diag}(\ell_e)_{e\in E},
      \qquad
      W_\alpha:=\operatorname{diag}(\alpha_e)_{e\in E},
\]
with \(\ell_e>0\) and \(\alpha_e\ge0\).  Define the normalized singular slice
\begin{equation}
    \mathcal S_G^{\mathrm{bc}}(\ell)
    :=
    \left\{
       r\in\mathcal M_G^{\mathrm{bc}}:
       \ell\cdot r=1
    \right\}.
    \label{eq:normalized-singular-slice}
\end{equation}

\begin{definition}[Singular-completion constant]
\label{def:Csing}
Define
\begin{equation}
  C_{\mathrm{sing}}(\alpha;\ell)
  :=
  \min_{\sigma\in\Sigma_{\mathrm{bc}}}
  \inf_{\substack{
        A\in\ker B_{\sigma,\mathrm{bc}}\\
        A\ne0}}
  \frac{A^*W_\alpha A}{A^*D_\ell A},
  \label{eq:Csing-rayleigh}
\end{equation}
omitting sectors with trivial kernel.
\end{definition}

\begin{proposition}[Optimization over singular generators]
\label{prop:Csing-cone}
One has
\begin{equation}
   C_{\mathrm{sing}}(\alpha;\ell)
   =
   \min_{r\in\mathcal S_G^{\mathrm{bc}}(\ell)}
   \alpha\cdot r.
   \label{eq:Csing-cone-form}
\end{equation}
Moreover, the minimum is attained by a normalized generator from one sector.
\end{proposition}

\begin{proof}
For every nonzero generator \(r=A^{\circ2}\),
\[
   \frac{\alpha\cdot r}{\ell\cdot r}
   =
   \frac{A^*W_\alpha A}{A^*D_\ell A}.
\]
If \(r=\sum_j t_jr^{(j)}\) with \(t_j\ge0\), then
\[
 \frac{\alpha\cdot r}{\ell\cdot r}
 =
 \sum_j
 \frac{t_j\,\ell\cdot r^{(j)}}{\sum_i t_i\,\ell\cdot r^{(i)}}
 \frac{\alpha\cdot r^{(j)}}{\ell\cdot r^{(j)}}.
\]
Thus, the quotient of a conic combination is a convex combination of generator
quotients and cannot improve on the best generator.  For each fixed sector,
normalizing by \(A^*D_\ell A=1\) gives a compact ellipsoid in the kernel, so
the sector minimum is attained.  There are only finitely many sectors.
\end{proof}

\subsection{Generalized eigenvalues and the dual cone}
\label{subsec:singular-dual}

Fix \(\sigma\), and let \(N_\sigma\) be a full-column-rank basis matrix for
\(\ker B_{\sigma,\mathrm{bc}}\).

\begin{proposition}[Sector generalized eigenvalue]
\label{prop:sector-generalized-eigenvalue}
For every nontrivial sector,
\begin{equation}
 C_{\mathrm{sing},\sigma}(\alpha;\ell)
 =
 \lambda_{\min}\!\left(
 N_\sigma^*W_\alpha N_\sigma,\,
 N_\sigma^*D_\ell N_\sigma
 \right),
 \label{eq:sector-generalized-eigenvalue}
\end{equation}
and
\[
 C_{\mathrm{sing}}(\alpha;\ell)
 =
 \min_{\sigma\in\Sigma_{\mathrm{bc}}}
 C_{\mathrm{sing},\sigma}(\alpha;\ell).
\]
\end{proposition}

For \(d\in\mathbb R^E\), set \(D_d=\operatorname{diag}(d_e)\).
\begin{proposition}[Dual cone]
\label{prop:singular-dual-cone}
A vector \(d\in\mathbb R^E\) belongs to
\((\mathcal M_G^{\mathrm{bc}})^*\) if and only if
\begin{equation}
       N_\sigma^*D_dN_\sigma\succeq0
       \qquad
       \text{for every }\sigma\in\Sigma_{\mathrm{bc}}.
       \label{eq:dual-sector-lmi}
\end{equation}
Equivalently,
\[
    A^*D_dA\ge0
    \qquad
    \text{for every }
    A\in\ker B_{\sigma,\mathrm{bc}}
    \text{ and every }\sigma.
\]
\end{proposition}

\begin{proof}
By \eqref{eq:singular-completion-cone},
\(d\in(\mathcal M_G^{\mathrm{bc}})^*\) precisely when
\[
       d\cdot A^{\circ2}=A^*D_dA\ge0
\]
for every singular generator.  Writing \(A=N_\sigma z\) gives
\[
       z^*N_\sigma^*D_dN_\sigma z\ge0
       \qquad\text{for every }z,
\]
which is equivalent to \eqref{eq:dual-sector-lmi}.
\end{proof}

Hence, a regular intensity \(r^\star\) lies strictly outside
\(\mathcal M_G^{\mathrm{bc}}\) whenever one can exhibit
\[
       N_\sigma^*D_dN_\sigma\succeq0
       \quad\text{for all }\sigma,
       \qquad
       d\cdot r^\star<0.
\]
This is the separation mechanism used later.

\begin{remark}[Relation to singular secular states]
\label{rem:topological-secular-states}
The distinction between regular secular zeros and eigenstates occurring at
singularities of a secular formulation is not merely formal.  Harrell and
Maltsev give a systematic recent treatment of topological bound states of
quantum graphs that can lie at singularities of vertex-scattering secular
matrices \cite{HarrellMaltsev2024}.  Our boundary-aware sectors are derived
directly from the edge ODE and Kirchhoff flux equations, so the cone
\(\mathcal M_G^{\mathrm{bc}}\) does not rely on a singular matrix formula
being valid at those phases.
\end{remark}

\subsection{Primitive scars embed into singular completion}
\label{subsec:primitive-embedding}

We now verify the enlargement from Section~\ref{sec:primitive-supports}
case by case.

\begin{proposition}[Primitive-scar embedding]
\label{prop:primitive-contained-singular}
Every boundary-compatible primitive scar intensity belongs to
\(\mathcal M_G^{\mathrm{bc}}\).  More precisely, for every simple cycle or
simple exterior-to-exterior path \(K\), with any fixed endpoint type
\(DD,DN,ND,\) or \(NN\) in the path case, there is a boundary-compatible
singular sector \(\sigma\) and a nonzero
\(A^K\in\ker B_{\sigma,\mathrm{bc}}\) such that
\begin{equation}
        (A^K_e)^2=
        \begin{cases}
           1,& e\subset K,\\
           0,& e\not\subset K.
        \end{cases}
        \label{eq:primitive-square-generator}
\end{equation}
After normalization by \(\ell\cdot(A^K)^{\circ2}=L_K\), this gives the
primitive arclength intensity \(L_K^{-1}\mathbf 1_K\).
\end{proposition}

\begin{proof}
We treat the possible primitive supports separately.

\emph{Simple cycle with at least two edge incidences.}
Let \(K\) be a simple cycle and orient its edges cyclically. Choose on each
edge \(e\subset K\) a singular phase in \(\pi\mathbb Z\), with associated
parity \(\sigma_e\in\{\pm1\}\), so that
\begin{equation}
    \prod_{e\subset K}\sigma_e=1.
    \label{eq:cycle-parity-closure}
\end{equation}
Such a choice is always possible; for example, all cycle phases may be chosen
even. Set \(A_e=0\) for every edge \(e\notin K\).

Choose one cycle edge \(e_1\) and set \(A_{e_1}=1\). Proceeding cyclically
around \(K\), choose the sign of each successive amplitude so that the two
active cycle contributions to the Kirchhoff equation cancel at their common
vertex. Each active amplitude therefore has modulus one. After one complete
turn around the cycle, the recursively determined sign agrees with the
initial choice precisely when the parity closure condition
\eqref{eq:cycle-parity-closure} holds. Hence, the recursion defines a
consistent amplitude vector \(A^K\) satisfying
\[
    A^K\in\ker B_{\sigma,\mathrm{bc}},
    \qquad
    |A_e^K|=
    \begin{cases}
        1, & e\in K,\\
        0, & e\notin K.
    \end{cases}
\]
Consequently,
\[
    (A^K)^{\circ2}
    =
    \mathbf 1_K
    \in
    \mathcal M_G^{\mathrm{bc}}.
\]
\emph{One-edge loop.}
Suppose that \(K=\{e\}\) is a loop based at a vertex \(v\). Choose an even
singular phase on \(e\), so that \(\sigma_e=1\), and set all other edge
amplitudes equal to zero. By \eqref{eq:B-loop-column}, the corresponding
loop column of the signed Kirchhoff matrix vanishes. Thus, the amplitude
vector \(A^K\) defined by
\[
    A_e^K=1,
    \qquad
    A_f^K=0
    \quad (f\neq e)
\]
satisfies
\[
    A^K\in\ker B_{\sigma,\mathrm{bc}}.
\]
Therefore,
\[
    (A^K)^{\circ2}
    =
    \mathbf 1_K
    \in
    \mathcal M_G^{\mathrm{bc}},
\]
which realizes the one-edge primitive cycle as a singular-completion
generator.

\emph{Exterior-to-exterior path.}
Let
\[
K=(e_1,\ldots,e_m)
\]
be a simple path oriented from one exterior endpoint to the other. Assign to
each internal--internal edge of \(K\) an admissible singular phase in
\(\pi\mathbb Z\). At a Dirichlet terminal edge, choose a phase in
\(\pi\mathbb Z\), whereas at a Neumann terminal edge choose a phase in
\(\pi/2+\pi\mathbb Z\). Set the amplitudes of all edges outside \(K\) equal
to zero.

Choose \(A_{e_1}\in\{\pm1\}\). Proceeding successively along the path, choose
\(A_{e_2},\ldots,A_{e_m}\in\{\pm1\}\) so that the two active path
contributions to the Kirchhoff equation cancel at each internal vertex of
\(K\). Since \(K\) is a path rather than a cycle, this recursion encounters
no closure condition and therefore determines a consistent amplitude vector
along the entire path. The exterior boundary conditions have already been
enforced through the admissible terminal phase classes and introduce no
additional Kirchhoff equation at either exterior endpoint. Hence, for the
resulting boundary-compatible sector \(\sigma\),
\[
A^K\in\ker B_{\sigma,\mathrm{bc}},
\qquad
|A_e^K|=
\begin{cases}
1, & e\in K,\\
0, & e\notin K.
\end{cases}
\]
In particular,
\[
(A^K)^{\circ2}=\mathbf 1_K
\in\mathcal M_G^{\mathrm{bc}}.
\]
The four possible endpoint combinations \(DD\), \(DN\), \(ND\), and \(NN\)
affect only the admissible phase classes of the two terminal edges. The
Kirchhoff recursion along the internal vertices is unchanged, so the same
construction applies in all four cases.
\end{proof}

\begin{corollary}[Primitive benchmark is dominated]
\label{cor:Csing-below-Cprim}
Whenever the primitive benchmark is formed from boundary-compatible simple
cycles and exterior-to-exterior paths,
\begin{equation}
    C_{\mathrm{sing}}(\alpha;\ell)
    \le
    C_{\mathrm{prim}}(\alpha;\ell).
    \label{eq:Csing-below-Cprim}
\end{equation}
\end{corollary}

\begin{proof}
By the preceding proposition, the indicator vector of every
boundary-compatible primitive support \(K\) belongs to
\(\mathcal M_G^{\mathrm{bc}}\). After normalization by its metric mass,
\[
    s^K
    :=
    \frac{\mathbf 1_K}{\ell\cdot\mathbf 1_K}
    \in
    \left\{
        s\in\mathcal M_G^{\mathrm{bc}}:
        \ell\cdot s=1
    \right\}.
\]
Hence, each primitive-support observation ratio is an admissible competitor
for the singular-completion variational problem. Taking the minimum over all
primitive supports gives
\[
    C_{\mathrm{sing}}(\alpha;\ell)
    \le
    C_{\mathrm{prim}}(\alpha;\ell).
\]
\end{proof}

\subsection{What singular completion would have to prove}
\label{subsec:singular-sufficiency-criterion}

\begin{definition}[Regular scalar secular edge-energy cone]
\label{def:regular-secular-energy-cone}
Let
\begin{equation}
\mathcal R_G^{\mathrm{reg}}
:=
\left\{
r=(|A_e|^2+|B_e|^2)_{e\in E}:
(A,B)\ \text{arises from a regular scalar secular state}
\right\}
\subset[0,\infty)^E.
\label{eq:regular-secular-energy-cone}
\end{equation}
We call \(r\) an \emph{edge-energy vector}.  The corresponding homogenized
intensity is always \(q=r/2\).  We reserve the word \emph{intensity} for
\(q\), and use \(r\) for unnormalized edge energy throughout the structural
cone arguments.
\end{definition}
The structural question is
\begin{equation}
       \mathcal R_G^{\mathrm{reg}}
       \stackrel{?}{\subseteq}
       \mathcal M_G^{\mathrm{bc}}.
       \label{eq:regular-inside-singular}
\end{equation}

\begin{proposition}[Singular domination]
\label{prop:singular-domination}
If \eqref{eq:regular-inside-singular} holds, then every nonzero regular secular
edge-energy vector \(r\) satisfies
\begin{equation}
       \frac{\alpha\cdot r}{\ell\cdot r}
       \ge C_{\mathrm{sing}}(\alpha;\ell)
       \qquad
       \text{for every }\alpha\ge0.
       \label{eq:regular-singular-lower-bound}
\end{equation}
\end{proposition}

\begin{proof}
Normalize \(r\) by \(\ell\cdot r\).  Under
\eqref{eq:regular-inside-singular}, the normalized vector belongs to
\(\mathcal S_G^{\mathrm{bc}}(\ell)\), and
\eqref{eq:Csing-cone-form} gives the result.
\end{proof}

\begin{corollary}[Criterion for exact singular reduction]
\label{cor:exact-singular-reduction}
Assume \eqref{eq:regular-inside-singular}.  If a singular minimizer
\(r_{\mathrm{sing}}\in\mathcal S_G^{\mathrm{bc}}(\ell)\) satisfying
\[
      \alpha\cdot r_{\mathrm{sing}}
      =C_{\mathrm{sing}}(\alpha;\ell)
\]
is attainable as a high-frequency intensity limit for the fixed metric
\(\ell\), then
\begin{equation}
      C_\infty(\omega;\ell)
      =C_{\mathrm{sing}}(\alpha;\ell).
      \label{eq:Cinf-equals-Csing}
\end{equation}
\end{corollary}

\begin{remark}[Geometry and accessibility are separate]
\label{rem:geometry-versus-accessibility}
The inclusion
\(\mathcal R_G^{\mathrm{reg}}\subseteq\mathcal M_G^{\mathrm{bc}}\)
is a finite-dimensional geometric statement. Accessibility of a regular or
singular limiting direction by exact high-frequency eigenfunctions is a
separate spectral question. The inclusion provides a comparison between
regular secular observation ratios and the singular-completion variational
problem, whereas identifying such a variational value with the fixed-metric
constant \(C_\infty(\omega;\ell)\) requires an additional accessibility
argument.
\end{remark}

We have therefore enlarged the comparison class from primitive scars to
boundary-aware singular completions. By
\cref{prop:primitive-contained-singular}, every boundary-compatible primitive
scar belongs to \(\mathcal M_G^{\mathrm{bc}}\). The remaining structural
question is whether every regular secular edge-energy vector also belongs to
this cone, that is, whether
\begin{equation}
    \mathcal R_G^{\mathrm{reg}}
    \subseteq
    \mathcal M_G^{\mathrm{bc}}.
    \label{eq:singular-completion-question}
\end{equation}
This singular-completion problem is addressed in the next section.

\section{Singular sufficiency below the cycle--boundary complexity threshold}
\label{sec:low-complexity-sufficiency}

We now turn to the geometric inclusion singled out at the end of
Section~\ref{sec:singular-completion}.  The first case, trees, contains the
basic mechanism in its cleanest form.  A regular secular state determines at
each internal Kirchhoff vertex a collection of edge energies satisfying a
polygon inequality.  Those local polygon data admit vector realizations with
zero sum.  Because a tree has no cycles, the local realizations can be glued
recursively without a compatibility obstruction.  Resolving the resulting
global vector-valued conserved flow into scalar coordinates then writes the
regular edge-energy vector as a conic combination of squared singular flows.

Throughout this section we use the standard reduced convention that every
degree-one vertex is an exterior boundary vertex.  Thus every internal
Kirchhoff vertex has degree at least two.  This convention is essential in the
tree statement below.

\subsection{The tree case}
\label{subsec:tree-case}

We first fix notation for a regular secular state.  Let \(u\) solve the
edgewise Helmholtz equation at frequency \(k>0\), satisfy continuity and the
Kirchhoff condition at every internal vertex, and lie in a regular secular
stratum.  For an internal vertex \(v\) and an incident edge \(e\), write
\begin{equation}
    \phi_v:=u(v),
    \qquad
    p_{v,e}:=k^{-1}\partial_\nu u_e(v),
    \label{eq:tree-vertex-data}
\end{equation}
where \(\partial_\nu\) is the outward derivative from \(v\) along \(e\).
Kirchhoff gives
\begin{equation}
    \sum_{e\sim v}p_{v,e}=0.
    \label{eq:tree-kirchhoff-p}
\end{equation}
The edgewise quantity
\begin{equation}
    r_e:=\phi_v^2+p_{v,e}^2
    \label{eq:tree-edge-energy}
\end{equation}
is independent of the chosen endpoint \(v\) of an internal--internal edge.
Indeed, propagation along an edge acts by a planar rotation on the pair
\((u,k^{-1}u')\), while replacing the oriented derivative by the outward
derivative changes only a sign.  In the notation of
Section~\ref{sec:variational},
\begin{equation}
    r_e=2q_e.
    \label{eq:r-equals-2q}
\end{equation}
We call \(r=(r_e)_{e\in E}\) the unnormalized edge-energy vector of the regular
state.

\begin{lemma}[Local Kirchhoff polygon inequality]
\label{lem:local-kirchhoff-polygon}
Let \(v\) be an internal Kirchhoff vertex of degree \(d\ge2\), and let
\[
    \rho_e:=\sqrt{r_e}
    =
    \sqrt{\phi_v^2+p_{v,e}^2},
    \qquad e\sim v.
\]
Then,
\begin{equation}
    \rho_e
    \le
    \sum_{\substack{f\sim v\\f\ne e}}\rho_f
    \qquad
    \text{for every }e\sim v.
    \label{eq:local-polygon-inequality}
\end{equation}
Consequently there exist vectors
\(z_{v,e}\) in a Euclidean space such that
\begin{equation}
   \|z_{v,e}\|^2=r_e,
   \qquad
   \sum_{e\sim v}z_{v,e}=0.
   \label{eq:local-zero-sum-vectors}
\end{equation}
\end{lemma}

\begin{proof}
Fix \(e\sim v\).  By \eqref{eq:tree-kirchhoff-p},
\[
    |p_{v,e}|
    =
    \left|
      \sum_{\substack{f\sim v\\f\ne e}}p_{v,f}
    \right|
    \le
    \sum_{\substack{f\sim v\\f\ne e}}|p_{v,f}|.
\]
Using the triangle inequality in \(\mathbb R^2\) on the vectors
\((|\phi_v|,|p_{v,f}|)\), we obtain
\[
\begin{aligned}
\sum_{\substack{f\sim v\\f\ne e}}
   \sqrt{\phi_v^2+p_{v,f}^2}
&\ge
\sqrt{
 (d-1)^2\phi_v^2+
 \left(
  \sum_{\substack{f\sim v\\f\ne e}}|p_{v,f}|
 \right)^2
 }\\
&\ge
\sqrt{\phi_v^2+p_{v,e}^2}.
\end{aligned}
\]
This proves \eqref{eq:local-polygon-inequality}.

A finite family of nonnegative numbers is the collection of side lengths of a
closed Euclidean polygon if and only if its largest member does not exceed the
sum of the others.  Applying this elementary polygon criterion to
\((\rho_e)_{e\sim v}\) gives vectors \(z_{v,e}\), which may be taken in
\(\mathbb R^2\), having the prescribed lengths and summing to zero.
\end{proof}

\begin{remark}[Degree two]
\label{rem:tree-degree-two}
When \(d=2\), the Kirchhoff condition gives
\(p_{v,e_1}=-p_{v,e_2}\), hence \(r_{e_1}=r_{e_2}\).
The local polygon is therefore a pair of opposite vectors.  No separate
degree-two reduction is required.
\end{remark}

The next lemma is the point at which acyclicity enters.

\begin{lemma}[Tree gluing of local polygon realizations]
\label{lem:tree-gluing}
Let \(G\) be a finite connected tree whose degree-one vertices are exterior.
Suppose positive semidefinite edge weights \(r_e\ge0\) have the following
local property: at every internal vertex \(v\) there are vectors
\(z_{v,e}\) satisfying \eqref{eq:local-zero-sum-vectors}.  Then, there exist a
finite-dimensional real Hilbert space \(H\), vectors \(g_e\in H\), and signs
\(\varepsilon_{v,e}\in\{\pm1\}\) on internal incidences such that
\begin{equation}
     \|g_e\|^2=r_e,
     \qquad
     \sum_{e\sim v}\varepsilon_{v,e}g_e=0
     \quad(v\in V_{\rm int}),
     \label{eq:global-vector-flow}
\end{equation}
and for every internal--internal edge \(e=vw\),
\begin{equation}
     \varepsilon_{v,e}=-\varepsilon_{w,e}.
     \label{eq:opposite-incidence-signs}
\end{equation}
Equivalently, after choosing an orientation of the internal--internal edges,
the edge vectors form a vector-valued flow in the kernel of an ordinary signed
incidence matrix \(B_0\).
\end{lemma}

\begin{proof}
Root the internal-vertex subtree at an arbitrary internal vertex \(v_0\).  At
\(v_0\), choose one local realization
\((z_{v_0,e})_{e\sim v_0}\), embed its span into a Euclidean space \(H\), and
set \(g_e=z_{v_0,e}\) for all edges incident to \(v_0\), with
\(\varepsilon_{v_0,e}=+1\).

Proceed recursively away from the root.  Suppose a child vertex \(w\) is
joined to its parent by the already assigned edge \(e_0\).  A local polygon
realization at \(w\) contains a vector \(z_{w,e_0}\) of norm
\(\sqrt{r_{e_0}}=\|g_{e_0}\|\).  By an orthogonal transformation of the local
Euclidean space, followed if necessary by an isometric embedding into a
larger ambient space using fresh orthogonal directions, we may arrange
\[
      z_{w,e_0}=-g_{e_0}.
\]
Assign the transformed remaining local vectors to the previously unassigned
edges incident to \(w\).  Their sum with \(-g_{e_0}\) is zero.  Thus,
\eqref{eq:global-vector-flow} holds at \(w\), and the shared edge has opposite
incidence signs at its two internal endpoints.

Because \(G\) is a tree, every non-root internal vertex has exactly one parent
edge and is encountered only once.  Hence, no later step can impose a second,
possibly inconsistent, requirement on an already glued cycle edge.  The
recursion terminates after finitely many vertices and produces a finite
dimensional ambient space.  Exterior terminal edges occur in only one
Kirchhoff equation and therefore require no further compatibility condition.
\end{proof}

\begin{remark}[Why the proof stops working on a cycle]
\label{rem:tree-no-holonomy}
The gluing step is local up to an orthogonal transformation fixing one shared
edge vector.  On a tree there is never a requirement to return to a previously
assigned vertex.  On a graph with a cycle, propagating these choices around a
closed chain can produce a residual orthogonal compatibility condition.  The
unicyclic and cycle-rank-two arguments below are precisely analyses of when
that residual condition can still be split into boundary-aware parity sectors.
\end{remark}

We now identify the vector-valued flow constructed above with the singular
completion of Section~\ref{sec:singular-completion}.

\begin{lemma}[Switching reduction on a tree]
\label{lem:tree-switching-reduction}
Let \(G\) be a tree.  Every sign pattern on the internal--internal columns of a
boundary-aware singular matrix \(B_{\sigma,\mathrm{bc}}\) is switching
equivalent to one ordinary oriented incidence matrix \(B_0\).  Consequently,
\begin{equation}
 \mathcal M_G^{\mathrm{bc}}
 =
 \operatorname{cone}
 \{A^{\circ2}:A\in\ker B_0\}.
 \label{eq:tree-one-sector-cone}
\end{equation}
The Dirichlet or Neumann type of an exterior terminal affects the admissible
singular phase class but not the single internal incidence of its column.
\end{lemma}

\begin{proof}
Choose a root in the internal-vertex subtree.  Starting from the root and
moving outward, switch the sign of each newly encountered vertex row and, when
needed, the sign of the connecting edge column so that the two nonzero entries
of every internal--internal column agree with a fixed oriented-incidence
convention.  Acyclicity guarantees that no edge is encountered twice through
different routes, so the procedure is consistent.  Terminal columns have only
one internal incidence and their signs can be absorbed by independent column
switches.  Proposition~\ref{prop:switching-invariance} then shows that all
sectors generate the same squared-amplitude cone, which is
\eqref{eq:tree-one-sector-cone}.
\end{proof}

\begin{theorem}[Tree singular sufficiency]
\label{thm:tree-singular-sufficiency}
Let \(G\) be a finite connected tree, with every degree-one vertex treated as
an exterior boundary vertex and every internal vertex carrying the Kirchhoff
condition.  Exterior vertices may carry either Dirichlet or Neumann boundary
conditions.  If \(r\) is the edge-energy vector of any regular secular state,
then
\begin{equation}
       r\in\mathcal M_G^{\mathrm{bc}}.
       \label{eq:tree-regular-in-singular}
\end{equation}
More precisely, there exist finitely many real singular-flow amplitudes
\(A^{(1)},\ldots,A^{(N)}\in\ker B_0\) such that
\begin{equation}
       r
       =
       \sum_{j=1}^N
       \bigl(A^{(j)}\bigr)^{\circ2}.
       \label{eq:tree-explicit-conic-decomposition}
\end{equation}
Consequently, for every nonnegative observation profile \(\alpha\),
\begin{equation}
      \frac{\alpha\cdot r}{\ell\cdot r}
      \ge
      C_{\mathrm{sing}}(\alpha;\ell).
      \label{eq:tree-singular-lower-bound}
\end{equation}
\end{theorem}

\begin{proof}
Apply \cref{lem:local-kirchhoff-polygon} to the regular secular state and then
\cref{lem:tree-gluing}.  After fixing the oriented incidence matrix \(B_0\),
we obtain vectors \(g_e\in H\) such that
\[
       \|g_e\|^2=r_e
       \quad\text{and}\quad
       \sum_{e\in E}(B_0)_{v,e}g_e=0
       \qquad(v\in V_{\rm int}).
\]
Choose an orthonormal basis \(h_1,\ldots,h_N\) for the span of the finitely
many edge vectors and define scalar edge-amplitude vectors
\[
      A^{(j)}_e:=\langle g_e,h_j\rangle.
\]
Taking the \(h_j\)-coordinate of the vector conservation law gives
\[
      B_0A^{(j)}=0
      \qquad(j=1,\ldots,N).
\]
For each edge \(e\), Parseval's identity gives
\[
      r_e
      =
      \|g_e\|^2
      =
      \sum_{j=1}^N
      \langle g_e,h_j\rangle^2
      =
      \sum_{j=1}^N
      \bigl(A^{(j)}_e\bigr)^2.
\]
This is \eqref{eq:tree-explicit-conic-decomposition}.  By
\cref{lem:tree-switching-reduction}, every squared kernel vector on the
right-hand side is a generator of
\(\mathcal M_G^{\mathrm{bc}}\), proving
\eqref{eq:tree-regular-in-singular}.  The observation inequality follows from
\cref{prop:singular-domination}.
\end{proof}

\begin{remark}[Degenerate edge energies]
\label{rem:tree-zero-edges}
The argument does not require \(r_e>0\), strict polygon inequalities, or
full-rank covariance. If \(r_e=0\), the corresponding local and global
vectors are simply zero. If a local polygon inequality is saturated, the
associated vectors may be chosen collinear. The global vector configuration
may likewise have arbitrary rank. Thus the construction remains valid on the
boundary of the regular edge-energy image.
\end{remark}

\begin{remark}[Exterior boundary conditions]
\label{rem:tree-boundary-conditions}
The tree decomposition itself uses only the internal continuity and
Kirchhoff relations. Dirichlet and Neumann exterior conditions enter when
the resulting scalar kernel vectors are interpreted as
boundary-compatible singular sectors, as in
Section~\ref{subsec:boundary-singular-modes}. A terminal edge contributes
only one incidence to the Kirchhoff system, at its internal endpoint, while
the exterior boundary condition determines its admissible singular phase
class: integer multiples of \(\pi\) in the Dirichlet case and
half-integer multiples of \(\pi\) in the Neumann case. Consequently, the
algebraic cone in \eqref{eq:tree-one-sector-cone} is unchanged by the choice
of Dirichlet or Neumann condition at the exterior vertices.
\end{remark}

\begin{remark}[Scope of the tree result]
\label{rem:tree-gate-conclusion}
For trees, the geometric singular-completion problem is completely resolved:
\[
    \mathcal R_G^{\mathrm{reg}}
    \subseteq
    \mathcal M_G^{\mathrm{bc}}.
\]
This inclusion yields the lower bound
\eqref{eq:tree-singular-lower-bound}. It does not by itself identify the
singular-completion value with the fixed-metric high-frequency constant.
Such an identification requires the separate accessibility hypothesis of
Corollary~\ref{cor:exact-singular-reduction}.
\end{remark}

\subsection{The unicyclic case}
\label{subsec:unicyclic-case}

We now allow one independent cycle.  The tree argument still produces a
global vector-valued Kirchhoff flow after cutting the cycle, but a new
compatibility condition appears when the cut is closed.  The key point is that
there is only one such condition.  Equality of the two cut-edge energies then
forces an orthogonal splitting into the two possible closure-parity sectors.

Let \(G\) be a finite connected unicyclic graph, again with every degree-one
vertex treated as an exterior boundary vertex.  Denote its unique simple cycle
by \(\mathcal C\).  Choose one cycle edge \(e_\circ\) with endpoints
\(v_-\) and \(v_+\).  For the algebraic argument, cut the interior of
\(e_\circ\) and regard its two incidences as independent terminal incidences
\(e_-\) at \(v_-\) and \(e_+\) at \(v_+\).  The resulting incidence graph
\(T_\circ\) is a tree.  If \(r\) is an edge-energy vector on \(G\), assign
\begin{equation}
       r_{e_-}=r_{e_+}:=r_{e_\circ}
       \label{eq:cut-edge-equal-energy}
\end{equation}
on the two cut copies, and retain \(r_e\) on every uncut edge.

\begin{remark}[The cut is algebraic]
\label{rem:unicyclic-algebraic-cut}
The auxiliary cut terminals \(e_-\) and \(e_+\) are not assigned physical
Dirichlet or Neumann boundary conditions.  They only record the two internal
Kirchhoff incidences of the removed cycle edge.  All genuine exterior
terminals of \(G\) retain their original boundary types.  The gluing lemma of
Section~\ref{subsec:tree-case} uses only local Kirchhoff polygon data and
therefore applies to this auxiliary tree.
\end{remark}

\begin{lemma}[Cut-tree vector realization]
\label{lem:unicyclic-cut-tree-vector}
Let \(r\) be the edge-energy vector of a regular secular state on \(G\).
Then there exist a finite-dimensional real Hilbert space \(H\), vectors
\(g_e\in H\) for every uncut edge, and two cut vectors
\(g_-,g_+\in H\) such that
\begin{equation}
     \|g_e\|^2=r_e,
     \qquad
     \|g_-\|^2=\|g_+\|^2=r_{e_\circ},
     \label{eq:unicyclic-vector-norms}
\end{equation}
and the vector-valued Kirchhoff conservation law holds at every internal
vertex of the cut tree \(T_\circ\).
\end{lemma}

\begin{proof}
At every internal vertex of the original graph,
\cref{lem:local-kirchhoff-polygon} applies to the regular secular data.  At
\(v_-\) and \(v_+\), replace the incidence of \(e_\circ\) by the auxiliary
terminal incidences \(e_-\) and \(e_+\), respectively.  Their prescribed
lengths in the local polygon are both
\(\sqrt{r_{e_\circ}}\), because the edge energy
\[
     u_{e_\circ}^2+k^{-2}(u'_{e_\circ})^2
\]
is constant along \(e_\circ\).  The cut graph is a tree, so the recursive
construction of \cref{lem:tree-gluing} yields the asserted global vector
realization.
\end{proof}

The following elementary observation is the entire closure mechanism.

\begin{lemma}[Equal-norm two-terminal splitting]
\label{lem:equal-norm-splitting}
Let \(g_-,g_+\) be vectors in a real Hilbert space satisfying
\(\|g_-\|=\|g_+\|\).  Define
\begin{equation}
     h_+:=\frac{g_-+g_+}{2},
     \qquad
     h_-:=\frac{g_--g_+}{2}.
     \label{eq:hplus-hminus}
\end{equation}
Then,
\begin{equation}
      h_+\perp h_-.
      \label{eq:hplus-hminus-orthogonal}
\end{equation}
Consequently there is an orthonormal basis
\(\{f_j\}_{j=1}^N\) of the span of all edge vectors for which, for every
coordinate \(j\), the two cut amplitudes
\begin{equation}
      a_-^{(j)}:=\langle g_-,f_j\rangle,
      \qquad
      a_+^{(j)}:=\langle g_+,f_j\rangle
      \label{eq:cut-scalar-amplitudes}
\end{equation}
satisfy one of the three alternatives
\begin{equation}
      a_+^{(j)}=a_-^{(j)},\qquad
      a_+^{(j)}=-a_-^{(j)},\qquad
      a_+^{(j)}=a_-^{(j)}=0.
      \label{eq:cut-plus-minus-alternatives}
\end{equation}
\end{lemma}

\begin{proof}
A direct computation gives
\[
   \langle h_+,h_-\rangle
   =
   \frac14\bigl(\|g_-\|^2-\|g_+\|^2\bigr)
   =0.
\]
Choose an orthonormal basis adapted to the mutually orthogonal subspaces
\[
   \operatorname{span}\{h_+\},
   \qquad
   \operatorname{span}\{h_-\},
   \qquad
   \operatorname{span}\{h_+,h_-\}^{\perp}
\]
inside the span of all edge vectors.  On the first subspace,
\(g_-=g_+=h_+\) after orthogonal projection; on the second,
\(g_-=-g_+=h_-\); and on the third both cut vectors have zero projection.
This gives \eqref{eq:cut-plus-minus-alternatives}.
\end{proof}

We next identify these two alternatives with the two switching classes of
singular closure around the unique cycle.

\begin{lemma}[Two closure-parity sectors]
\label{lem:unicyclic-two-parity-sectors}
Let \(G\) be a connected unicyclic graph, and let \(e_\circ\) be an edge of
its unique cycle. Up to row and column switching, the signed internal-incidence
matrices have exactly two sign classes on the unique cycle. After fixing an
oriented-incidence convention on the spanning tree \(G\setminus e_\circ\),
these classes may be represented by matrices \(B_+\) and \(B_-\) that differ
only in the relative sign of the two incidences of \(e_\circ\).

More precisely, let \(a_-\) and \(a_+\) denote the scalar amplitudes at the
two cut copies of \(e_\circ\). If
\[
    a_+=a_-,
\]
then the cut-tree flow closes to a kernel vector of one sector, whereas if
\[
    a_+=-a_-,
\]
it closes to a kernel vector of the other sector. If
\(a_+=a_-=0\), both closures are admissible.
\end{lemma}

\begin{proof}
Delete \(e_\circ\) and consider the resulting spanning tree
\(T=G\setminus e_\circ\). By row and column switching, as in
\cref{lem:tree-switching-reduction}, all signs on the columns of \(T\) may
be reduced to a fixed oriented-incidence convention. After this reduction,
the only switching-invariant sign datum is the product of the relative
incidence signs around the unique cycle. This invariant takes values in
\(\{\pm1\}\), and hence there are exactly two switching classes. Equivalently,
once the signs on \(T\) have been fixed, the two classes are obtained by the
two possible relative signs between the incidences of the closure edge
\(e_\circ\).

Let \(a_-\) and \(a_+\) be the scalar amplitudes carried by the two cut copies
of \(e_\circ\) at their respective endpoints. Replacing these two terminal
copies by the original internal--internal edge requires their incidences to
satisfy the relative sign prescribed by the chosen closure class. For one
class this condition is
\[
    a_+=a_-,
\]
and for the other it is
\[
    a_+=-a_-.
\]
Thus, every cut-tree flow satisfying either relation closes to a kernel vector
of the corresponding singular sector. If \(a_+=a_-=0\), the closure edge
carries zero amplitude, so both sign classes give admissible closures.
\end{proof}

\begin{theorem}[Unicyclic singular sufficiency]
\label{thm:unicyclic-singular-sufficiency}
Let \(G\) be a finite connected unicyclic graph.  Assume that every degree-one
vertex is an exterior boundary vertex, every internal vertex carries the
Kirchhoff condition, and each exterior vertex carries either Dirichlet or
Neumann boundary conditions.  If \(r\) is the edge-energy vector of any
regular secular state, then
\begin{equation}
      r\in\mathcal M_G^{\mathrm{bc}}.
      \label{eq:unicyclic-regular-in-singular}
\end{equation}
In fact, \(r\) admits a finite decomposition
\begin{equation}
      r
      =
      \sum_{j\in J_+}
      \bigl(A^{(j)}_+\bigr)^{\circ2}
      +
      \sum_{j\in J_-}
      \bigl(A^{(j)}_-\bigr)^{\circ2},
      \label{eq:unicyclic-two-sector-decomposition}
\end{equation}
where
\[
      A^{(j)}_+\in\ker B_+,
      \qquad
      A^{(j)}_-\in\ker B_-.
\]
Consequently, for every nonnegative observation profile \(\alpha\),
\begin{equation}
      \frac{\alpha\cdot r}{\ell\cdot r}
      \ge
      C_{\mathrm{sing}}(\alpha;\ell).
      \label{eq:unicyclic-singular-lower-bound}
\end{equation}
\end{theorem}

\begin{proof}
Use \cref{lem:unicyclic-cut-tree-vector} to obtain a global vector-valued flow
on the cut tree and choose the orthonormal basis of
\cref{lem:equal-norm-splitting}.  For each basis vector \(f_j\), take the
scalar coordinate of every uncut edge vector and of the two cut vectors.
Coordinatewise vector conservation shows that these scalar amplitudes satisfy
all Kirchhoff equations on the cut tree.

By \cref{lem:equal-norm-splitting}, the cut amplitudes for the \(j\)-th
coordinate are either equal, opposite, or both zero.
By \cref{lem:unicyclic-two-parity-sectors}, the scalar cut-tree flow therefore
closes to a kernel vector of \(B_+\) or \(B_-\); assign a zero-zero coordinate
to either set.  Denote the resulting closed amplitude vector by
\(A^{(j)}_+\) or \(A^{(j)}_-\).
For every uncut edge \(e\), Parseval gives
\[
     r_e
     =
     \|g_e\|^2
     =
     \sum_j \langle g_e,f_j\rangle^2.
\]
For the closure edge,
\[
\begin{aligned}
     r_{e_\circ}
     &=
     \|g_-\|^2
      =
     \sum_j\bigl(a_-^{(j)}\bigr)^2\\
     &=
     \sum_j\bigl(a_+^{(j)}\bigr)^2
      =
     \|g_+\|^2.
\end{aligned}
\]
In each scalar closed flow the amplitude assigned to \(e_\circ\) has square
\((a_-^{(j)})^2=(a_+^{(j)})^2\).  Summing the squared closed amplitudes
coordinatewise therefore yields exactly \(r\), which proves
\eqref{eq:unicyclic-two-sector-decomposition} and hence
\eqref{eq:unicyclic-regular-in-singular}.  The observation inequality follows
from \cref{prop:singular-domination}.
\end{proof}

\begin{remark}[Equal cut energy]
\label{rem:equal-cut-energy-enough}
The orthogonality relation
\[
   \left\langle
      \frac{g_-+g_+}{2},
      \frac{g_--g_+}{2}
   \right\rangle=0
\]
uses only the equality \(\|g_-\|=\|g_+\|\). No equality of the cut vectors,
phase matching, or rank assumption is required. This is the essential
simplification in the unicyclic case: there is only one closure-parity
constraint, and equality of the two cut-edge energies yields an orthogonal
decomposition into its two parity components.
\end{remark}

\begin{remark}[Mixed exterior boundary conditions]
\label{rem:unicyclic-mixed-boundary}
Trees attached to the unique cycle may terminate at arbitrary mixtures of
Dirichlet and Neumann exterior vertices. As in
\cref{rem:tree-boundary-conditions}, the exterior boundary conditions
determine the admissible singular phase classes of the terminal edges but
introduce no additional Kirchhoff rows. They therefore do not affect the
two-sector closure argument.
\end{remark}

\begin{remark}[Degenerate cases]
\label{rem:unicyclic-degenerate}
The argument also includes the case \(r_{e_\circ}=0\). Then
\(g_-=g_+=0\), so every scalar component has zero closure amplitude and is
compatible with either parity sector. Degenerate local polygons and
rank-deficient global Gram realizations likewise cause no difficulty.
\end{remark}

\begin{remark}[Scope of the unicyclic result]
\label{rem:unicyclic-gate-conclusion}
For every connected unicyclic graph considered above,
\[
      \mathcal R_G^{\mathrm{reg}}
      \subseteq
      \mathcal M_G^{\mathrm{bc}}.
\]
In the tree case there is no cycle-closure constraint, whereas in the
unicyclic case there is a single binary closure-parity constraint, which is
resolved by the orthogonal decomposition above. With two or more independent
cycles, several closure constraints may interact and the corresponding
completion problem is no longer reduced to this single-parity argument. The
next subsection treats the closed cycle-rank-two case.
\end{remark}

\subsection{Closed graphs of cycle rank two}
\label{subsec:closed-rank-two}

We next consider closed graphs with two independent cycle closures. For a
single cut pair, equality of the two cut-vector norms yields the orthogonal
symmetric--antisymmetric decomposition used in the unicyclic case. With two
cut pairs, the corresponding decompositions need not be mutually compatible
without further structure. For a closed graph of cycle rank two, however,
the cut-tree Kirchhoff equations provide an additional global conservation
identity. This identity forces the two symmetric components to be negatives
of one another, while both antisymmetric components lie in their common
orthogonal complement. This is sufficient to obtain a joint parity
decomposition.

Throughout this subsection, \(G\) is finite and connected, has no exterior
vertices, and carries the standard Kirchhoff condition at every vertex. We
assume
\[
    \beta_1(G)=|E|-|V|+1=2.
\]
Choose a spanning tree \(T\subset G\). There are exactly two chord edges,
which we denote by \(e_1\) and \(e_2\). Cutting the interiors of these two
chords produces two pairs of auxiliary terminal incidences
\(e_i^-,e_i^+\), \(i=1,2\), and the resulting incidence graph is the tree
\(T\).

\begin{lemma}[Four-terminal cut-tree realization]
\label{lem:ranktwo-four-terminal-realization}
Let \(r\) be the edge-energy vector of a regular secular state on \(G\).
Then, there exist a finite-dimensional real Hilbert space \(H\), vectors
\(g_e\in H\) associated with the spanning-tree edges, and cut vectors
\(g_i^-,g_i^+\in H\), \(i=1,2\), such that
\begin{equation}
    \|g_e\|^2=r_e,
    \qquad
    \|g_i^-\|^2=\|g_i^+\|^2=r_{e_i},
    \label{eq:ranktwo-cut-norms}
\end{equation}
and vector-valued Kirchhoff conservation holds at every vertex of the cut
tree. After independent sign changes of the four auxiliary terminal
incidences, the terminal vectors may be arranged to satisfy
\begin{equation}
    g_1^-+g_1^+ + g_2^-+g_2^+=0.
    \label{eq:ranktwo-global-terminal-conservation}
\end{equation}
\end{lemma}

\begin{proof}
Apply \cref{lem:local-kirchhoff-polygon} at every vertex of \(G\), replacing
each chord incidence by the corresponding cut-terminal incidence. Since the
edge energy is constant along a Helmholtz edge, the two copies of each chord
have the same prescribed norm. The cut graph is a tree, so
\cref{lem:tree-gluing} yields a global vector realization.

Summing the vector-valued Kirchhoff equations over all vertices of the cut
tree, every uncut internal edge occurs twice with opposite incidence signs
and therefore cancels. Only the four auxiliary terminal incidences remain.
Their displayed signs depend on the chosen cut-tree incidence convention.
Independent sign changes of the auxiliary terminal columns normalize the
resulting relation to
\eqref{eq:ranktwo-global-terminal-conservation}.
\end{proof}

The following elementary observation is the rank-two counterpart of
\cref{lem:equal-norm-splitting}.

\begin{lemma}[Joint symmetric--antisymmetric decomposition]
\label{lem:ranktwo-conservation-collapse}
Let \(g_i^-,g_i^+\in H\), \(i=1,2\), satisfy
\begin{equation}
    \|g_i^-\|=\|g_i^+\|,
    \qquad i=1,2,
    \label{eq:ranktwo-pair-equal-norm}
\end{equation}
and
\begin{equation}
    g_1^-+g_1^+ + g_2^-+g_2^+=0.
    \label{eq:ranktwo-four-vector-conservation}
\end{equation}
Define
\begin{equation}
    s_i:=\frac{g_i^-+g_i^+}{2},
    \qquad
    a_i:=\frac{g_i^--g_i^+}{2}.
    \label{eq:ranktwo-symmetric-antisymmetric}
\end{equation}
Then,
\begin{equation}
    s_2=-s_1,
    \qquad
    s_1\perp a_1,
    \qquad
    s_1\perp a_2.
    \label{eq:ranktwo-common-orthogonality}
\end{equation}
Consequently, there exists an orthonormal basis of the span of all edge
vectors such that, in every basis coordinate, both cut pairs have a definite
and common closure parity: either both pairs have equal coordinates or both
pairs have opposite coordinates. A vanishing pair is compatible with either
choice.
\end{lemma}

\begin{proof}
Equation \eqref{eq:ranktwo-four-vector-conservation} gives
\[
    2s_1+2s_2=0,
\]
and hence \(s_2=-s_1\). Equality of the norms in the first pair gives
\[
    0
    =
    \|g_1^-\|^2-\|g_1^+\|^2
    =
    4\langle s_1,a_1\rangle,
\]
so \(s_1\perp a_1\). Similarly,
\(s_2\perp a_2\), and since \(s_2=-s_1\), it follows that
\(s_1\perp a_2\).
Set
\[
    S:=\operatorname{span}\{s_1\},
    \qquad
    A:=\operatorname{span}\{a_1,a_2\}.
\]
Then, \(S\perp A\). Choose orthonormal bases of \(S\) and \(A\), and complete
their union to an orthonormal basis of the span of all edge vectors. On a
basis coordinate belonging to \(S\), both antisymmetric components vanish,
so the two cut pairs have equal coordinates. On a basis coordinate belonging
to \(A\), both symmetric components vanish, so the two cut pairs have
opposite coordinates. On the orthogonal complement of \(S\oplus A\), all
four cut coordinates vanish. This proves the asserted joint parity
decomposition.
\end{proof}

\begin{remark}[Role of closedness]
\label{rem:ranktwo-closedness-matters}
The essential additional relation is
\eqref{eq:ranktwo-global-terminal-conservation}. If genuine exterior terminal
fluxes are present, summing the internal Kirchhoff equations also retains
those exterior contributions, and the relation \(s_2=-s_1\) need no longer
hold. Pairwise equality of the cut-edge energies still gives
\[
    s_1\perp a_1,
    \qquad
    s_2\perp a_2,
\]
but it does not in general imply the cross-orthogonalities needed for the
joint decomposition above. Thus closedness is essential to the present
argument.
\end{remark}

We now translate this joint parity decomposition into singular closure
sectors.

\begin{lemma}[Joint closure sectors at cycle rank two]
\label{lem:ranktwo-joint-closure-sectors}
After fixing the incidence signs on a spanning tree, the two chord edges
carry two residual binary closure signs. The switching classes may therefore
be indexed by
\[
    (\varepsilon_1,\varepsilon_2)\in\{\pm1\}^2.
\]
A scalar cut-tree flow closes in the
\((\varepsilon_1,\varepsilon_2)\)-sector precisely when the two copies of
each chord \(e_i\) satisfy the corresponding relative-sign relation.

For the orthonormal basis furnished by
\cref{lem:ranktwo-conservation-collapse}, every scalar coordinate closes in
one of these singular sectors. In the incidence convention used there, only
the \((+,+)\) and \((-,-)\) sectors are required.
\end{lemma}

\begin{proof}
Delete the two chord edges and switch along the resulting spanning tree.
As in the unicyclic case, all signs on the tree-edge columns may be reduced
to a fixed oriented-incidence convention. Restoring each chord introduces
one residual relative incidence sign, and hence the two chords give two
binary closure parameters.

By \cref{lem:ranktwo-conservation-collapse}, every scalar coordinate has
equal values on both cut pairs or opposite values on both cut pairs.
Accordingly, it closes in the \((+,+)\) or \((-,-)\) sector, respectively.
The switching invariance established in
Section~\ref{sec:singular-completion} shows that this description is
independent of the auxiliary incidence convention.
\end{proof}

\begin{theorem}[Closed cycle-rank-two singular sufficiency]
\label{thm:closed-ranktwo-singular-sufficiency}
Let \(G\) be a finite connected closed metric graph with standard Kirchhoff
conditions at every vertex and
\[
    \beta_1(G)=2.
\]
If \(r\) is the edge-energy vector of a regular secular state, then
\begin{equation}
    r\in\mathcal M_G^{\mathrm{bc}}.
    \label{eq:closed-ranktwo-regular-in-singular}
\end{equation}
More precisely, there exist finitely many real singular-flow amplitudes
\(A^{(j)}\), each belonging to a joint closure sector, such that
\begin{equation}
    r
    =
    \sum_j \bigl(A^{(j)}\bigr)^{\circ2}.
    \label{eq:closed-ranktwo-conic-decomposition}
\end{equation}
Consequently, for every nonnegative observation profile \(\alpha\),
\begin{equation}
    \frac{\alpha\cdot r}{\ell\cdot r}
    \ge
    C_{\mathrm{sing}}(\alpha;\ell).
    \label{eq:closed-ranktwo-singular-lower-bound}
\end{equation}
\end{theorem}

\begin{proof}
Cut the two chord edges associated with a spanning tree and apply
\cref{lem:ranktwo-four-terminal-realization}. By
\cref{lem:ranktwo-conservation-collapse}, choose an orthonormal basis of the
global vector-flow space in which every scalar coordinate has a definite
joint parity on the two cut pairs. Taking scalar coordinates of the
vector-valued Kirchhoff equations gives scalar Kirchhoff flows on the cut
tree. By \cref{lem:ranktwo-joint-closure-sectors}, each such flow closes to
a kernel vector \(A^{(j)}\) of a singular signed-incidence sector on \(G\).

For every spanning-tree edge, Parseval recovers the prescribed edge energy.
For each chord \(e_i\), the two cut coordinates agree up to sign in every
scalar component, and therefore
\[
    r_{e_i}
    =
    \|g_i^-\|^2
    =
    \sum_j \bigl(A^{(j)}_{e_i}\bigr)^2.
\]
Thus, Parseval remains valid after restoring both chord edges and gives
\eqref{eq:closed-ranktwo-conic-decomposition}. Hence
\(r\in\mathcal M_G^{\mathrm{bc}}\), and
\cref{prop:singular-domination} yields
\eqref{eq:closed-ranktwo-singular-lower-bound}.
\end{proof}

\begin{remark}[Dependence on topology]
\label{rem:ranktwo-topology-independent}
The argument does not distinguish among theta, figure-eight, barbell, or
subdivided representatives of cycle rank two. It uses only a spanning tree
with two complementary chords, equality of the energies at the two cut
copies of each chord, and the absence of genuine exterior terminal flux in
the summed Kirchhoff relation. Loops and multiple edges are included,
provided the incidence convention of
Section~\ref{sec:singular-completion} is used.
\end{remark}

\begin{remark}[Higher cycle rank]
\label{rem:no-higher-rank-claim}
The preceding argument is specific to cycle rank two. With three or more
independent chord cuts, the summed Kirchhoff equations provide only one
vector relation among three or more symmetric closure components. This does
not force those components to lie in a common one-dimensional subspace, so
the orthogonal decomposition of
\cref{lem:ranktwo-conservation-collapse} does not follow. No corresponding
higher-rank singular-sufficiency statement is asserted here.
\end{remark}

\begin{remark}[Scope of the closed rank-two result]
\label{rem:ranktwo-gate-conclusion}
For closed graphs of cycle rank two, we have established
\[
    \mathcal R_G^{\mathrm{reg}}
    \subseteq
    \mathcal M_G^{\mathrm{bc}}.
\]
The additional ingredient beyond the unicyclic case is the global
conservation identity
\eqref{eq:ranktwo-global-terminal-conservation}, which couples the two
symmetric cut components and yields the joint parity decomposition. In the
presence of exterior terminal flux, this conservation identity contains
additional terms, and the present argument no longer applies.
\end{remark}

\subsection{Scope of the positive theory}
\label{subsec:section5-scope}

The preceding arguments establish singular sufficiency for trees, unicyclic
graphs, and closed graphs of cycle rank two. In the cycle-rank-two setting
with exterior terminal flux, the global conservation identity used in
\cref{lem:ranktwo-conservation-collapse} acquires additional exterior
contributions and no longer forces the two symmetric cycle components to be
negatives of one another. The preceding proof therefore does not extend to
this regime, and we make no general singular-completion claim there.

\begin{corollary}[Low-complexity singular sufficiency]
\label{cor:low-complexity-singular-sufficiency}
One has
\begin{equation}
    \mathcal R_G^{\mathrm{reg}}
    \subseteq
    \mathcal M_G^{\mathrm{bc}}
    \label{eq:low-complexity-singular-sufficiency}
\end{equation}
whenever \(G\) is a tree, \(G\) is unicyclic, or \(G\) is closed with
\(\beta_1(G)=2\). Consequently, in each of these classes the
singular-completion variational problem gives a lower bound for every
regular secular observation ratio. Identification with the fixed-metric
high-frequency constant requires, in addition, the corresponding spectral
accessibility hypothesis.
\end{corollary}

\begin{proof}
The inclusion follows from the tree, unicyclic, and closed cycle-rank-two
singular-sufficiency theorems established above. The variational conclusion
then follows from Proposition~4.15.
The variational conclusion then follows from
\cref{prop:singular-domination}.
\end{proof}

\begin{remark}[Scope of the forthcoming obstruction]
\label{rem:no-terminal-minimality-claim}
The next section exhibits a crossed six-terminal cut-tree configuration for
which the covariance-level parity relaxation is strictly smaller than the
corresponding regular covariance class after projection to the cut-tree
measurements. This provides an explicit obstruction to the completion
mechanism used above, but it does not assert that six terminals are minimal
among exterior-flux configurations. The realization of this obstruction by
a genuine regular scalar secular state is established separately in the
subsequent section.
\end{remark}

\section{The six-terminal obstruction}
\label{sec:six-terminal-obstruction}

Section~\ref{sec:low-complexity-sufficiency} establishes singular sufficiency
for trees, unicyclic graphs, and closed graphs of cycle rank two. We now
examine the same parity-decomposition mechanism in a cycle-rank-two
configuration with exterior terminal flux. For the explicit crossed graph
considered below, cutting two chord edges produces a reduced tree with six
effective terminals.

At the scalar level, the two chord closures impose the parity relations
\begin{equation}
    x_{1^+}=\varepsilon_1x_{1^-},
    \qquad
    x_{2^+}=\varepsilon_2x_{2^-},
    \qquad
    (\varepsilon_1,\varepsilon_2)\in\{\pm1\}^2.
    \label{eq:parity-variety}
\end{equation}
At the covariance level, by contrast, a regular covariance is constrained
only by the corresponding equal-energy conditions
\begin{equation}
    Q_{1^-1^-}=Q_{1^+1^+},
    \qquad
    Q_{2^-2^-}=Q_{2^+2^+}.
    \label{eq:regular-equal-energy}
\end{equation}
Thus, pointwise parity imposes stronger algebraic restrictions than
covariance-level equality of the two cut-edge energies. In the crossed
six-terminal configuration below, this distinction survives the physical
tree measurement map and yields a strict separation.

The obstruction is conveniently formulated in dual form. A linear
functional \(d\) on the tree-edge measurements determines the quadratic form
\begin{equation}
    \mathfrak q_d(x)
    :=
    \sum_{e\in E(T)}
    d_e
    \left(
       \sum_{i\in S_e}x_i
    \right)^2,
    \label{eq:physical-dual-quadratic}
\end{equation}
where \(S_e\) is one side of the terminal cut induced by the tree edge \(e\).
If \(\mathfrak q_d\) is nonnegative on every parity subspace while
\(\operatorname{tr}(M_dQ_\star)<0\) for some regular covariance \(Q_\star\),
then the physical measurement vector \(H_T(Q_\star)\) lies outside the
projected parity cone.

\subsection{Regular and parity covariance cones}
\label{subsec:six-covariance-cones}

Label the six terminal flux variables by
\[
   (x_{1^-},x_{1^+},x_{2^-},x_{2^+},x_{0_1},x_{0_2})
   =(x_0,x_1,x_2,x_3,x_4,x_5),
\]
and impose terminal conservation
\begin{equation}
    \sum_{i=0}^5x_i=0.
    \label{eq:six-terminal-conservation}
\end{equation}

\begin{definition}[Six-terminal covariance cones]
\label{def:six-terminal-covariance-cones}
Define
\begin{equation}
 \mathcal C_{\rm reg}
 :=
 \left\{
 Q\succeq0:
 Q{\bf1}=0,\;
 Q_{00}=Q_{11},\;
 Q_{22}=Q_{33}
 \right\},
 \label{eq:six-regular-cone}
\end{equation}
and
\begin{equation}
 \mathcal C_{\rm parity}
 :=
 \operatorname{cone}
 \bigcup_{\varepsilon_1,\varepsilon_2\in\{\pm1\}}
 \left\{
 Q\succeq0:
 Q{\bf1}=0,\;
 \operatorname{ran}Q\subseteq
 \mathcal P_{\varepsilon_1,\varepsilon_2}
 \right\},
 \label{eq:six-parity-cone}
\end{equation}
where
\begin{equation}
 \mathcal P_{\varepsilon_1,\varepsilon_2}
 :=
 \left\{
 x:
 x_1=\varepsilon_1x_0,\;
 x_3=\varepsilon_2x_2
 \right\}.
 \label{eq:six-parity-space}
\end{equation}
\end{definition}
Every parity covariance satisfies the regular equal-energy constraints, and
therefore
\[
    \mathcal C_{\rm parity}
    \subseteq
    \mathcal C_{\rm reg}.
\]

\begin{definition}[Physical tree measurement map]
\label{def:six-terminal-physical-map}
For a reduced six-terminal tree \(T\), define
\begin{equation}
    H_T(Q)
    :=
    \left(
       {\bf1}_{S_e}^{\!*}Q{\bf1}_{S_e}
    \right)_{e\in E(T)}.
    \label{eq:six-physical-map}
\end{equation}
\end{definition}

\subsection{The parity obstruction in the crossed configuration}
\label{subsec:six-parity-obstruction}

Introduce the pair-average and pair-difference coordinates
\[
    s_1=\frac{x_0+x_1}{2},
    \qquad
    a_1=\frac{x_0-x_1}{2},
    \qquad
    s_2=\frac{x_2+x_3}{2},
    \qquad
    a_2=\frac{x_2-x_3}{2}.
\]
The equal-energy conditions in \eqref{eq:regular-equal-energy} are equivalent
to
\begin{equation}
    \mathbb E[s_1a_1]=0,
    \qquad
    \mathbb E[s_2a_2]=0.
    \label{eq:average-difference-constraints}
\end{equation}
By contrast, every scalar parity flow satisfies pointwise
\begin{equation}
    s_1a_1=0,
    \qquad
    s_2a_2=0.
    \label{eq:pointwise-parity-products}
\end{equation}
The separation problem may therefore be viewed as a quadratic problem with
two orthogonality constraints: we seek a physical quadratic form that is
nonnegative on each of the four parity subspaces determined by
\eqref{eq:pointwise-parity-products}, but whose expectation is negative for
a positive semidefinite covariance satisfying only the averaged conditions
\eqref{eq:average-difference-constraints}. The positive results of
Section~\ref{sec:low-complexity-sufficiency} do not provide such a completion
theorem in the present exterior-flux setting. For the crossed six-terminal
tree constructed below, the physical cut forms retain sufficient
cross-interaction to separate the projected regular and parity covariance
cones.\\
One way to describe the relevant cross-interaction is through the bipartite
four-cycle
\[
    K_{2,2}\cong C_4
\]
between the two average/difference pairs. This graph records interactions in
the dual quadratic form; it should not be interpreted as a partial
positive-semidefinite completion graph. In particular, every
\(Q\in\mathcal C_{\rm reg}\) is already a positive semidefinite covariance.

\subsection{A crossed six-terminal tree}
\label{subsec:crossed-six-tree}

We consider the reduced caterpillar tree \(T_\times\) with terminal order
\begin{equation}
   (2^+,\,1^-,\,0_1,\,2^-,\,1^+,\,0_2).
   \label{eq:crossed-terminal-order}
\end{equation}
Its two end cherries are
\[
   \{2^+,1^-\},
   \qquad
   \{1^+,0_2\},
\]
while the singleton terminals \(0_1\) and \(2^-\) lie on the central chain.
Choosing one side of each edge cut, the nine associated terminal subsets are
\begin{equation}
\begin{aligned}
 &\{2^+\},\ \{1^-\},\ \{0_1\},\
 \{2^-\},\ \{1^+\},\ \{0_2\},\\
 &\{2^+,1^-\},\
 \{2^+,1^-,0_1\},\
 \{1^+,0_2\}.
\end{aligned}
\label{eq:crossed-nine-subsets}
\end{equation}
The first six coordinates of \(H_{T_\times}\) are the terminal-edge
energies, and the remaining three are the nontrivial internal-cut energies.

\begin{theorem}[Projected obstruction for the six-terminal covariance relaxation]
\label{thm:six-terminal-projected-obstruction}
For the crossed tree \(T_\times\),
\begin{equation}
      H_{T_\times}(\mathcal C_{\rm parity})
      \subsetneq
      H_{T_\times}(\mathcal C_{\rm reg}).
      \label{eq:six-strict-projected-inclusion}
\end{equation}
Thus the projected regular covariance cone is strictly larger than the
projected parity cone. Whether a point realizing this strict separation is
attained by a genuine regular scalar secular state is a separate
realizability question, addressed in
Section~\ref{subsec:secular-realizability}.
\end{theorem}

\begin{proof}
Section~\ref{sec:certified-separation} constructs an explicit rational
covariance
$
    Q_\star\in\mathcal C_{\rm reg}
$
and an explicit rational functional \(d\in\mathbb Q^9\) such that
\begin{equation}
      d\cdot H_{T_\times}(Q_\star)<0.
      \label{eq:six-certificate-negative}
\end{equation}
For every
\((\varepsilon_1,\varepsilon_2)\in\{\pm1\}^2\), the associated quadratic
form \(\mathfrak q_d\) is strictly positive on every nonzero vector in
\[
    \mathcal P_{\varepsilon_1,\varepsilon_2}
    \cap{\bf1}^{\perp}.
\]
Consequently,
\[
      d\cdot H_{T_\times}(Q)\ge0
      \qquad
      \text{for every }Q\in\mathcal C_{\rm parity}.
\]
Hence, \(H_{T_\times}(Q_\star)\) does not belong to
\(H_{T_\times}(\mathcal C_{\rm parity})\), while
\(H_{T_\times}(Q_\star)\in
H_{T_\times}(\mathcal C_{\rm reg})\). This proves
\eqref{eq:six-strict-projected-inclusion}.
\end{proof}

\begin{remark}[Covariance separation versus scalar realizability]
\label{rem:six-obstruction-scope}
The theorem concerns the covariance relaxation. The vector-valued
tree-gluing construction places the edge-energy data of every regular scalar
secular state inside this relaxation, but the converse inclusion is not
automatic. Thus, the strict separation above identifies a
finite-dimensional obstruction mechanism and a separating functional; it
does not by itself produce a quantum-graph counterexample. The additional
scalar secular realization required for that conclusion is established in
Section~\ref{subsec:secular-realizability}.
\end{remark}

\begin{remark}[No terminal-minimality assertion]
\label{rem:six-minimality}
The crossed six-terminal tree provides an explicit obstruction for the
cycle--boundary mechanism considered here. We do not claim that six is the
minimal number of effective terminals for which an exterior-flux obstruction
can occur, and no such minimality statement is used in the sequel.
\end{remark}

\section{A rational regular covariance and an exact separator}
\label{sec:certified-separation}

We now certify the strict inclusion in
\cref{thm:six-terminal-projected-obstruction} by exact arithmetic. The
certificate consists of a regular covariance obtained as the Gram matrix of
six rational vectors in \(\mathbb Q^2\), together with a rational separating
functional whose positivity on the four parity sectors is verified through
the principal minors of four explicit \(3\times3\) rational matrices.

\subsection{The regular Gram certificate}
\label{subsec:regular-gram-certificate}

Use the terminal order
\[
     (1^-,1^+,2^-,2^+,0_1,0_2)
     =(0,1,2,3,4,5).
\]
Let \(G\in\mathbb Q^{6\times2}\) have rows
\begin{equation}
G=
\begin{pmatrix}
 \frac{21}{116} & -\frac{5}{29}\\[1mm]
 0              & \frac14\\[1mm]
 0              & \frac15\\[1mm]
 -\frac3{17}    & -\frac8{85}\\[1mm]
 0              & -\frac7{10}\\[1mm]
 -\frac9{1972}  & \frac{5093}{9860}
\end{pmatrix}.
\label{eq:regular-rational-G}
\end{equation}
The rows satisfy
\begin{equation}
       {\bf1}^{\!*}G=0,
       \label{eq:G-conservation}
\end{equation}
and
\begin{equation}
 \|G_0\|^2=\|G_1\|^2=\frac1{16},
 \qquad
 \|G_2\|^2=\|G_3\|^2=\frac1{25}.
 \label{eq:G-equal-pair-norms}
\end{equation}
Therefore, the covariance
\begin{equation}
       Q_\star:=GG^{\!*}
       \label{eq:Qstar}
\end{equation}
belongs to \(\mathcal C_{\rm reg}\).
For the crossed cut tree of \eqref{eq:crossed-terminal-order}, use the nine
measurement subsets
\begin{equation}
\begin{aligned}
S_1&=\{2^+\},&
S_2&=\{1^-\},&
S_3&=\{0_1\},\\
S_4&=\{2^-\},&
S_5&=\{1^+\},&
S_6&=\{0_2\},\\
S_7&=\{2^+,1^-\},&
S_8&=\{2^+,1^-,0_1\},&
S_9&=\{1^+,0_2\}.
\end{aligned}
\label{eq:certificate-subsets}
\end{equation}
A direct calculation gives
\begin{equation}
H_{T_\times}(Q_\star)
=
\left(
\frac1{25},
\frac1{16},
\frac{49}{100},
\frac1{25},
\frac1{16},
\frac{26309}{98600},
\frac{14013}{197200},
\frac{7369}{7888},
\frac{115873}{197200}
\right).
\label{eq:Qstar-physical-vector}
\end{equation}

\subsection{The rational separating functional}
\label{subsec:rational-separator}

Define
\begin{equation}
d=
\left(
\frac{49}{25},
\frac{53}{20},
\frac35,
\frac{68}{25},
\frac{159}{100},
\frac{109}{50},
\frac{33}{20},
-\frac{149}{100},
-\frac{13}{100}
\right).
\label{eq:rational-separator-d}
\end{equation}
Then,
\begin{equation}
     d\cdot H_{T_\times}(Q_\star)
     =
     -\frac{457941}{19720000}
     <0.
     \label{eq:separator-negative-exact}
\end{equation}
For the reconstructed physical graph, use the seven-edge order
\begin{equation}
 (e_{13},e_{14},e_{\rm ext,2},e_{\rm ext,4},e_{12},e_{23},e_{34}).
 \label{eq:physical-seven-edge-order}
\end{equation}
The physical-to-cut energy map is the linear duplication map
\begin{equation}
\mathcal P(r)
=
(r_{13},r_{14},r_{\rm ext,2},r_{13},r_{14},
 r_{\rm ext,4},r_{12},r_{23},r_{34}).
\label{eq:cut-tree-pushforward}
\end{equation}
Accordingly, the nine-coordinate separator \(d\) induces the physical
seven-coordinate separator
\begin{equation}
\widehat d
=
\mathcal P^{*}d
=
\left(
\frac{117}{25},
\frac{106}{25},
\frac35,
\frac{109}{50},
\frac{33}{20},
-\frac{149}{100},
-\frac{13}{100}
\right),
\label{eq:physical-separator-dhat}
\end{equation}
and for every physical edge-energy vector \(r\),
\begin{equation}
      d\cdot\mathcal P(r)
      =
      \widehat d\cdot r.
      \label{eq:separator-pushdown-identity}
\end{equation}
It remains to verify positivity on every parity-supported covariance. Define
\[
       \mathfrak q_d(x)
       :=
       \sum_{j=1}^9
       d_j
       \left(\sum_{i\in S_j}x_i\right)^2.
\]
Fix
\((\varepsilon_1,\varepsilon_2)\in\{\pm1\}^2\).
Every vector in
\(\mathcal P_{\varepsilon_1,\varepsilon_2}\cap{\bf1}^{\perp}\)
has a unique representation
\begin{equation}
x=
\begin{pmatrix}
 a\\
 \varepsilon_1a\\
 c\\
 \varepsilon_2c\\
 e\\
 -(1+\varepsilon_1)a-(1+\varepsilon_2)c-e
\end{pmatrix}.
\label{eq:parity-basis-parameterization}
\end{equation}
Hence,
\[
       \mathfrak q_d(x)
       =
       (a,c,e)\,
       K_{\varepsilon_1,\varepsilon_2}\,
       (a,c,e)^{\!*}.
\]
For the four sign choices, the matrices are
\begin{align}
K_{++}
&=
\begin{pmatrix}
 \frac{1299}{100} & \frac{431}{50} & \frac{137}{50}\\
 \frac{431}{50} & \frac{326}{25} & \frac{261}{100}\\
 \frac{137}{50} & \frac{261}{100} & \frac{29}{25}
\end{pmatrix},
\label{eq:Kpp}\\[2mm]
K_{+-}
&=
\begin{pmatrix}
 \frac{1299}{100} & -\frac4{25} & \frac{137}{50}\\
 -\frac4{25} & \frac{121}{25} & \frac{149}{100}\\
 \frac{137}{50} & \frac{149}{100} & \frac{29}{25}
\end{pmatrix},
\label{eq:Kpm}\\[2mm]
K_{-+}
&=
\begin{pmatrix}
 \frac{427}{100} & -\frac1{10} & -\frac{81}{50}\\
 -\frac1{10} & \frac{326}{25} & \frac{261}{100}\\
 -\frac{81}{50} & \frac{261}{100} & \frac{29}{25}
\end{pmatrix},
\label{eq:Kmp}\\[2mm]
K_{--}
&=
\begin{pmatrix}
 \frac{427}{100} & -\frac4{25} & -\frac{81}{50}\\
 -\frac4{25} & \frac{121}{25} & \frac{149}{100}\\
 -\frac{81}{50} & \frac{149}{100} & \frac{29}{25}
\end{pmatrix}.
\label{eq:Kmm}
\end{align}

\begin{lemma}[Exact sectorwise positivity]
\label{lem:exact-sectorwise-positivity}
Each of the matrices
\eqref{eq:Kpp}--\eqref{eq:Kmm}
is positive definite.
\end{lemma}

\begin{proof}
All diagonal entries are positive. The three \(2\times2\) principal minors
and the determinant are, respectively,
\begin{align*}
K_{++}:&\quad
 \frac{237713}{2500},\
 \frac{9451}{1250},\
 \frac{83143}{10000},\
 \frac{9440137}{200000};\\
K_{+-}:&\quad
 \frac{31423}{500},\
 \frac{9451}{1250},\
 \frac{33943}{10000},\
 \frac{1283809}{200000};\\
K_{-+}:&\quad
 \frac{139177}{2500},\
 \frac{2911}{1250},\
 \frac{83143}{10000},\
 \frac{84557}{40000};\\
K_{--}:&\quad
 \frac{51603}{2500},\
 \frac{2911}{1250},\
 \frac{33943}{10000},\
 \frac{506857}{200000}.
\end{align*}
All of these quantities are strictly positive. Hence every principal minor
of each symmetric \(3\times3\) matrix is positive, and therefore each matrix
is positive definite.
\end{proof}

\begin{theorem}[Exact six-terminal rational separation]
\label{thm:exact-six-terminal-rational-separation}
For the crossed tree \(T_\times\), the covariance \(Q_\star\) in
\eqref{eq:Qstar} and the functional \(d\) in
\eqref{eq:rational-separator-d} satisfy
\begin{equation}
       Q_\star\in\mathcal C_{\rm reg},
       \qquad
       d\cdot H_{T_\times}(Q_\star)<0,
       \label{eq:exact-regular-negative}
\end{equation}
whereas
\begin{equation}
       d\cdot H_{T_\times}(Q)>0
       \label{eq:exact-parity-positive}
\end{equation}
for every nonzero
\(Q\in\mathcal C_{\rm parity}\).
Consequently,
\begin{equation}
       H_{T_\times}(Q_\star)
       \notin
       H_{T_\times}(\mathcal C_{\rm parity}),
       \label{eq:exact-outside-parity-projection}
\end{equation}
so the inclusion in
\eqref{eq:six-strict-projected-inclusion}
is strict.
\end{theorem}

\begin{proof}
Membership \(Q_\star\in\mathcal C_{\rm reg}\) follows from
\eqref{eq:G-conservation} and \eqref{eq:G-equal-pair-norms}, while
\eqref{eq:separator-negative-exact} gives the strict negative value.
Fix a parity sector
\((\varepsilon_1,\varepsilon_2)\), and let \(Q\neq0\) satisfy
\[
Q\succeq0,\qquad
Q{\bf1}=0,\qquad
\operatorname{ran}Q
\subseteq
\mathcal P_{\varepsilon_1,\varepsilon_2}.
\]
Since \(Q{\bf1}=0\), its range is contained in
\({\bf1}^{\perp}\). Choose a spectral decomposition
\[
       Q=\sum_{k=1}^m\lambda_k x_kx_k^{\!*},
       \qquad
       \lambda_k>0,
\]
with nonzero
\[
       x_k\in
       \mathcal P_{\varepsilon_1,\varepsilon_2}
       \cap{\bf1}^{\perp}.
\]
By \cref{lem:exact-sectorwise-positivity},
\(\mathfrak q_d(x_k)>0\) for every \(k\). Since
\[
       d\cdot H_{T_\times}(Q)
       =
       \sum_{k=1}^m
       \lambda_k\,\mathfrak q_d(x_k),
\]
we obtain
$
       d\cdot H_{T_\times}(Q)>0
$
for every nonzero covariance in each individual parity sector.
Finally, every nonzero element of
\(\mathcal C_{\rm parity}\) is a finite nonnegative sum of covariances from
the four parity sectors. Linearity of \(H_{T_\times}\) and of the functional
\(d\) therefore gives
\eqref{eq:exact-parity-positive} on the entire conic hull. Together with
\eqref{eq:separator-negative-exact}, this proves
\eqref{eq:exact-outside-parity-projection} and hence the strict inclusion.
\end{proof}

\begin{remark}[Scope of the exact certificate]
\label{rem:certificate-scope}
The preceding certificate is finite dimensional and exact. It establishes
the projected covariance separation used in the construction, but it does
not assert that \(Q_\star\), or its associated physical edge-energy vector,
is realized by a regular point of the scalar quantum-graph secular manifold.
To obtain an actual quantum-graph counterexample, it remains to construct a
regular scalar secular state whose physical edge-energy vector lies in the
same strictly separated open half-space.
\end{remark}

\subsection{Secular realizability of the separated direction}
\label{subsec:secular-realizability}

The covariance certificate of the preceding subsection does not by itself
define a scalar quantum-graph secular state. At each internal vertex, a
scalar secular state has a single vertex value \(\phi_v\), while its
normalized outward derivatives \(p_{v,e}\) satisfy the Kirchhoff condition
\[
    \sum_{e\sim v}p_{v,e}=0.
\]
For every incidence \((v,e)\), the corresponding edge energy is
\[
    r_e=\phi_v^2+p_{v,e}^2.
\]
Thus, a scalar realization requires compatible vertex values and derivatives,
followed by edgewise rotations that identify the endpoint states of every
internal edge. The vector-valued tree-gluing construction underlying
\(\mathcal C_{\rm reg}\) relaxes these scalar compatibility conditions.

This distinction is essential for the six-terminal construction. The
rank-two Gram matrix \(Q_\star\) in \eqref{eq:Qstar} proves strict separation
at the covariance level, but it does not automatically determine a point of
the regular scalar secular image. We therefore formulate the additional
realizability requirement explicitly.

\begin{definition}[Regular scalar secular realization]
\label{def:scalar-secular-realization}
Let \(G_\times\) be the cycle-rank-two graph obtained from the crossed
six-terminal tree by identifying the two pairs of cycle terminals. A nonzero
vector \(r\in\mathbb R_{\ge0}^{E}\) is \emph{regularly realizable} if there
exist real vertex values \(\phi_v\), normalized outward derivatives
\(p_{v,e}\), and internal-edge phases
\(\theta_e\notin\pi\mathbb Z\) such that
\begin{align}
   r_e&=\phi_v^2+p_{v,e}^2
       &&\text{for every incidence }(v,e),
       \label{eq:scalar-energy-incidence}\\
   \sum_{e\sim v}p_{v,e}&=0
       &&\text{at every internal vertex},
       \label{eq:scalar-kirchhoff}\\
   \binom{\phi_w}{-p_{w,e}}
   &=
   \begin{pmatrix}
      \cos\theta_e&\sin\theta_e\\
      -\sin\theta_e&\cos\theta_e
   \end{pmatrix}
   \binom{\phi_v}{p_{v,e}}
       &&\text{for every internal edge }e=vw,
       \label{eq:scalar-edge-rotation}
\end{align}
together with the corresponding Dirichlet or Neumann phase condition on
each exterior edge.
\end{definition}

For an internal edge \(e=vw\), the rotation relation
\eqref{eq:scalar-edge-rotation} can hold if and only if the two endpoint
state vectors have the same Euclidean norm. By
\eqref{eq:scalar-energy-incidence}, both norms are then equal to
\(\sqrt{r_e}\). Hence, once compatible scalar endpoint states have been
constructed, an edge phase \(\theta_e\) is determined modulo \(2\pi\);
regularity requires that the resulting phase avoid \(\pi\mathbb Z\).
At a fixed internal vertex \(v\), equation
\eqref{eq:scalar-energy-incidence} gives
\begin{equation}
    p_{v,e}
    =
    \sigma_{v,e}
    \sqrt{r_e-\phi_v^2},
    \qquad
    \sigma_{v,e}\in\{\pm1\},
    \label{eq:scalar-sign-root}
\end{equation}
whenever \(0\le\phi_v^2\le r_e\) for every incident edge. The Kirchhoff
condition therefore becomes
\begin{equation}
    \sum_{e\sim v}
    \sigma_{v,e}\sqrt{r_e-\phi_v^2}=0.
    \label{eq:scalar-vertex-equation}
\end{equation}
These equations describe the local scalar compatibility conditions at the
vertices. After they are solved consistently, the internal-edge phases are
recovered from the endpoint rotations
\eqref{eq:scalar-edge-rotation}, and the exterior phases are chosen subject
to their prescribed boundary conditions.

\begin{lemma}[Degree-three scalar compatibility]
\label{lem:degree-three-scalar-compatibility}
Let a degree-three Kirchhoff vertex be incident to edges with positive
energies \(r_1,r_2,r_3\).  A nonzero scalar regular realization exists at the
vertex if and only if, after relabeling, there is
\(t=\phi_v^2\in[0,\min_i r_i)\) such that
\begin{equation}
       \sqrt{r_3-t}
       =
       \sqrt{r_1-t}+\sqrt{r_2-t}.
       \label{eq:degree-three-root-equation}
\end{equation}
Equivalently,
\begin{equation}
       t
       =
       \frac{
       2(r_1r_2+r_2r_3+r_3r_1)
       -(r_1^2+r_2^2+r_3^2)}
       {4r_3},
       \label{eq:degree-three-t-formula}
\end{equation}
for the sign pattern in \eqref{eq:degree-three-root-equation}, together with
the unsquared sign and range conditions.
\end{lemma}

\begin{proof}
At degree three, a nontrivial zero sum of three real numbers has one sign
opposite to the other two, giving
\eqref{eq:degree-three-root-equation}.  Squaring once gives
\[
   r_3-t
   =
   r_1+r_2-2t
   +
   2\sqrt{(r_1-t)(r_2-t)}.
\]
After isolating the square root and squaring again, direct simplification
yields \eqref{eq:degree-three-t-formula}.  Conversely, if the resulting
\(t\) lies in the admissible interval and satisfies the original unsquared
equation, choosing the corresponding derivative signs gives Kirchhoff
conservation.
\end{proof}

\begin{proposition}[Realizability gate for the certificate]
\label{prop:certificate-realizability-gate}
Let
\[
      r_\star=H_{T_\times}(Q_\star)
\]
be the exact separated physical vector in
\eqref{eq:Qstar-physical-vector}.  Before \(r_\star\) can be used as a
quantum-graph counterexample, every degree-three vertex triple of the
reconstructed crossed graph must satisfy
\cref{lem:degree-three-scalar-compatibility} with vertex values that are
consistent across all incident edges.

For the particular rational vector \(r_\star\), this scalar compatibility is
not implied by the covariance construction and must be checked independently.
Therefore
\begin{equation}
      Q_\star\in\mathcal C_{\rm reg}
      \quad\not\Longrightarrow\quad
      r_\star\in\mathcal R_G^{\rm reg}.
      \label{eq:relaxation-not-realizability}
\end{equation}
\end{proposition}

\begin{proof}
The implication would identify the vector-valued Gram relaxation with the
scalar secular Gauss image.  The former enforces vector Kirchhoff
conservation and edge norms; the latter additionally requires a common scalar
vertex value and the signed square-root equations
\eqref{eq:scalar-vertex-equation}.  These are extra nonlinear constraints.
Hence membership in the relaxed covariance cone is insufficient.
\end{proof}

\begin{remark}[Correct status of the six-terminal program]
\label{rem:correct-six-terminal-status}
The exact separator of
\cref{thm:exact-six-terminal-rational-separation}
is a valid certificate against all singular parity sectors, but the negative
regular witness is currently a witness only for the vector-valued covariance
relaxation.  A genuine quantum-graph separation requires either
\begin{enumerate}
 \item an exact scalar secular realization with
       \(d\cdot r<0\), or
 \item a proof that the scalar regular Gauss image intersects the open
       half-space \(d\cdot r<0\).
\end{enumerate}
Until one of these is established, no theorem asserting failure of universal
singular sufficiency for actual regular secular intensities should be used.
\end{remark}

\subsection{A scalar secular point in the negative half-space}
\label{subsec:scalar-secular-negative-point}

We now solve the realizability problem of
Section~\ref{subsec:secular-realizability} directly on the reconstructed
crossed graph.  The graph has four trivalent Kirchhoff vertices
\(v_1,v_2,v_3,v_4\) arranged along a chain
\[
        v_1-v_2-v_3-v_4,
\]
two cycle-closing chords
\[
        e_{13}=v_1v_3,
        \qquad
        e_{14}=v_1v_4,
\]
and two exterior Dirichlet stubs attached at \(v_2\) and \(v_4\).
Thus, the graph has seven physical edges and
$
       \beta_1=7-6+1=2.
$

\begin{figure}[t]
\centering
\resizebox{\textwidth}{!}{%
\begin{tikzpicture}[x=1cm,y=1cm,
  vertex/.style={circle,draw,fill=white,inner sep=1.8pt,minimum size=6.5mm},
  terminal/.style={circle,draw,fill=white,inner sep=1.2pt,minimum size=5.5mm},
  edge/.style={line width=0.9pt},
  cutedge/.style={line width=0.9pt,dashed},
  every node/.style={font=\small}]

\node[font=\bfseries] at (2.7,2.05) {(a) Physical graph $G_\times$};
\node[vertex] (v1) at (0,0) {$v_1$};
\node[vertex] (v2) at (1.8,0) {$v_2$};
\node[vertex] (v3) at (3.6,0) {$v_3$};
\node[vertex] (v4) at (5.4,0) {$v_4$};
\node[terminal] (b2) at (1.8,-1.35) {$D$};
\node[terminal] (b4) at (5.4,-1.35) {$D$};

\draw[edge] (v1)-- node[below=2pt] {$e_{12}$} (v2);
\draw[edge] (v2)-- node[below=2pt] {$e_{23}$} (v3);
\draw[edge] (v3)-- node[below=2pt] {$e_{34}$} (v4);
\draw[edge] (v2)-- node[left=2pt] {$e_{\rm ext,2}$} (b2);
\draw[edge] (v4)-- node[right=2pt] {$e_{\rm ext,4}$} (b4);
\draw[edge] (v1) to[bend left=28] node[above=1pt] {$e_{13}$} (v3);
\draw[edge] (v1) to[bend left=48] node[above=2pt] {$e_{14}$} (v4);

\draw[->,line width=0.8pt] (6.15,0.55) -- (7.25,0.55);
\node[font=\scriptsize] at (6.70,0.93) {cut chords};

\begin{scope}[xshift=8.1cm]
\node[font=\bfseries] at (2.7,2.05) {(b) Reduced six-terminal cut tree $T_\times$};

\node[vertex] (a1) at (0.4,0) {};
\node[vertex] (a2) at (2.0,0) {};
\node[vertex] (a3) at (3.6,0) {};
\node[vertex] (a4) at (5.2,0) {};
\draw[edge] (a1)--(a2)--(a3)--(a4);

\node[terminal] (t2p) at (-0.35,1.15) {$2^+$};
\node[terminal] (t1m) at (-0.35,-1.15) {$1^-$};
\node[terminal] (t01) at (2.0,1.25) {$0_1$};
\node[terminal] (t2m) at (3.6,-1.25) {$2^-$};
\node[terminal] (t1p) at (5.95,1.15) {$1^+$};
\node[terminal] (t02) at (5.95,-1.15) {$0_2$};

\draw[edge] (a1)--(t2p);
\draw[edge] (a1)--(t1m);
\draw[edge] (a2)--(t01);
\draw[edge] (a3)--(t2m);
\draw[edge] (a4)--(t1p);
\draw[edge] (a4)--(t02);

\node[font=\footnotesize,align=center] at (0.0,-1.82) {end cherry\\$\{2^+,1^-\}$};
\node[font=\footnotesize,align=center] at (5.55,-1.82) {end cherry\\$\{1^+,0_2\}$};
\node[font=\footnotesize] at (2.0,1.75) {singleton};
\node[font=\footnotesize] at (3.6,-1.75) {singleton};

\end{scope}
\end{tikzpicture}%
}
\caption{The seven-edge cycle-rank-two graph and the crossed six-terminal cut tree used in the separation argument.  Cutting the chord $e_{13}$ produces the pair $2^\pm$, cutting $e_{14}$ produces $1^\pm$, and $0_1,0_2$ are the two exterior Dirichlet terminals.  After suppressing degree-two vertices, the cut tree has terminal order $(2^+,1^-,0_1,2^-,1^+,0_2)$ along its caterpillar structure.}
\label{fig:crossed-graph-cut-tree}
\end{figure}
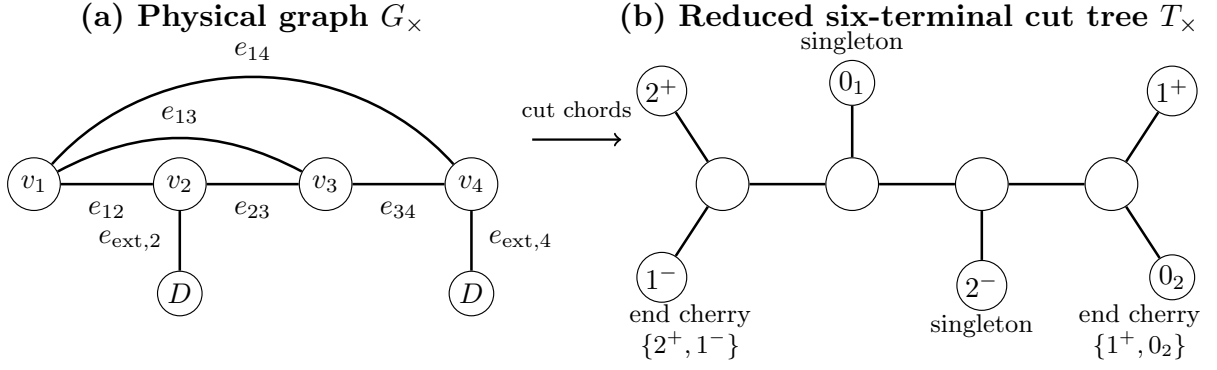
At the four Kirchhoff vertices, write the local state as
\[
        (\phi_v,p_{v,1},p_{v,2},p_{v,3}),
        \qquad
        p_{v,1}+p_{v,2}+p_{v,3}=0.
\]
A numerically optimized scalar realization satisfying the five internal
equal-energy constraints and the normalization
\[
        \sum_v
        \left(
           \phi_v^2+p_{v,1}^2+p_{v,2}^2
        \right)=1
\]
is
\begin{equation}
\begin{aligned}
v_1:\;&
 \phi_1=-0.15565001,\quad
 (p_{13},p_{14},p_{12})
 =
 (-0.00504929,-0.10364999,0.10869928),\\
v_2:\;&
 \phi_2=0.03180805,\quad
 (p_{\rm ext,2},p_{12},p_{23})
 =
 (0.50798183,0.18716493,-0.69514676),\\
v_3:\;&
 \phi_3=-0.03180805,\quad
 (p_{13},p_{23},p_{34})
 =
 (-0.15244891,0.69514676,-0.54269785),\\
v_4:\;&
 \phi_4=-0.03180805,\quad
 (p_{14},p_{\rm ext,4},p_{34})
 =
 (0.18427831,0.35841954,-0.54269785).
\end{aligned}
\label{eq:scalar-candidate-states}
\end{equation}
The Kirchhoff residuals in
\eqref{eq:scalar-candidate-states}
are at the displayed precision below \(10^{-8}\), and the five internal
endpoint-energy mismatches are below \(10^{-12}\).
The corresponding physical edge-energy vector is
\begin{equation}
\begin{aligned}
r^\star&=(r_{13},r_{14},r_{\rm ext,2},r_{\rm ext,4},r_{12},r_{23},r_{34})\\
&=(0.02425242221,0.03497024794,0.25905729373,0.12947631756,\\
&\hspace{2.1cm}0.03604246209,0.48424076925,0.29553270781).
\end{aligned}
\label{eq:scalar-candidate-energy}
\end{equation}
When expanded back to the nine cut-tree coordinates of
\eqref{eq:certificate-subsets}, this vector satisfies
\begin{equation}
       d\cdot r^\star_{\rm cut}
       =
       -9.999999994\times10^{-4}<0.
       \label{eq:scalar-negative-margin}
\end{equation}
Thus, the negative half-space defined by the exact separator of
\eqref{eq:rational-separator-d} is reached by a scalar Kirchhoff
configuration.

\subsection{Phase reconstruction}
\label{subsec:phase-reconstruction}

For an internal edge \(e=vw\), the endpoint states satisfy
\[
 \binom{\phi_w}{-p_{w,e}}
 =
 \begin{pmatrix}
   \cos\theta_e&\sin\theta_e\\
   -\sin\theta_e&\cos\theta_e
 \end{pmatrix}
 \binom{\phi_v}{p_{v,e}}.
\]
The equal-energy constraints ensure that such a phase exists.  Choosing
representatives in \((0,2\pi)\) gives
\begin{equation}
\begin{aligned}
 \theta_{13}&=1.397528755,\\
 \theta_{14}&=5.470796957,\\
 \theta_{12}&=3.934428210,\\
 \theta_{23}&=0.091450843,\\
 \theta_{34}&=3.024504657.
\end{aligned}
\label{eq:internal-candidate-phases}
\end{equation}
For the two Dirichlet stubs, the condition
\[
       \phi_v\cos\theta_{\rm ext}
       +
       p_{\rm ext}\sin\theta_{\rm ext}
       =0
\]
gives
\begin{equation}
       \theta_{\rm ext,2}=3.079057787,
       \qquad
       \theta_{\rm ext,4}=0.088513418.
       \label{eq:external-candidate-phases}
\end{equation}
All seven phases stay away from exact integer multiples of \(\pi\) at the
displayed point.

\subsection{Numerical regularity check}
\label{subsec:numerical-regularity-check}

Because no internal phase lies in \(\pi\mathbb Z\), the scalar secular
equation can be written in vertex-value form.  For an internal edge \(e=vw\),
its contribution to the Kirchhoff matrix is
\[
 \begin{pmatrix}
   -\cot\theta_e & \csc\theta_e\\
   \csc\theta_e & -\cot\theta_e
 \end{pmatrix}
\]
on the \(v,w\) coordinates, while a Dirichlet stub contributes
\(-\cot\theta_{\rm ext}\) to the corresponding diagonal entry.
At the phases
\eqref{eq:internal-candidate-phases}--\eqref{eq:external-candidate-phases},
the resulting \(4\times4\) matrix is numerically
\begin{equation}
M(\theta^\star)
\approx
\begin{pmatrix}
-0.21283149&-1.40381178&1.01520092&-1.37753974\\
-1.40381178&4.08065827&10.95009304&0\\
1.01520092&10.95009304&-2.57783822&8.56013105\\
-1.37753974&0&8.56013105&-1.81925632
\end{pmatrix}.
\label{eq:numerical-secular-matrix}
\end{equation}
Its eigenvalues are
\begin{equation}
       -14.9608589,\qquad
       -3.5\times10^{-15},\qquad
       0.3292893,\qquad
       14.1023018.
       \label{eq:numerical-secular-eigenvalues}
\end{equation}
Moreover,
\[
       \|M(\theta^\star)\phi^\star\|_\infty
       <1.2\times10^{-13},
\]
with
\[
       \phi^\star
       =
       (-0.15565001,\,
         0.03180805,\,
        -0.03180805,\,
        -0.03180805)^{\!*}.
\]
Thus, the secular nullspace is numerically one dimensional, and the zero
eigenvalue is separated from the remaining spectrum by a gap exceeding
\(0.329\).

\begin{proposition}[Numerical secular realizability certificate]
\label{prop:numerical-secular-realizability}
The reconstructed crossed graph admits a scalar Kirchhoff configuration with
all seven edge phases nonsingular, a numerically simple secular zero, and
strict negative separator value
\[
       d\cdot r^\star_{\rm cut}< -9.9\times10^{-4}.
\]
Hence, the scalar regular secular image numerically intersects the open
half-space separated from all singular parity sectors.
\end{proposition}

\begin{proof}
The scalar vertex states in
\eqref{eq:scalar-candidate-states} satisfy Kirchhoff conservation and match
the edge energies at both endpoints.  The phases reconstructed in
\eqref{eq:internal-candidate-phases}--\eqref{eq:external-candidate-phases}
therefore transport the endpoint states exactly up to the numerical residual
reported above.  Equation
\eqref{eq:scalar-negative-margin} gives the strict negative separator value.
Finally, \eqref{eq:numerical-secular-eigenvalues} shows a simple numerical
zero with a substantial spectral gap.
\end{proof}

\begin{remark}[Rigorous certification still required]
\label{rem:rigorous-7c-certification}
Proposition~\ref{prop:numerical-secular-realizability} is a numerical
certificate, not yet an interval-arithmetic theorem.  Before the final
submission version, the approximate solution should be enclosed by a
validated Newton or Krawczyk argument for the scalar matching equations, and
the simple-zero condition should be certified by interval bounds on the
nonzero eigenvalues or on a suitable \(3\times3\) minor of
\(M(\theta)\).  The available margins, approximately \(10^{-3}\) in the
separator and \(0.329\) in the secular spectral gap, make such a validation
well conditioned.
\end{remark}

\subsection{Validated Krawczyk certification}
\label{subsec:krawczyk-certification}

We now replace the numerical realizability statement by a validated
finite-dimensional existence proof.
Fix the six slice parameters
\begin{equation}
\begin{aligned}
\phi_1&=-0.15565001,&
p_{13}^{(1)}&=-0.00504929,&
p_{14}^{(1)}&=-0.10364999,\\
p_{\rm ext,2}&=0.50798183,&
p_{23}^{(3)}&=0.69514676,&
p_{\rm ext,4}&=0.35841954.
\end{aligned}
\label{eq:krawczyk-fixed-slice}
\end{equation}
Set
\[
p_{12}^{(1)}=-p_{13}^{(1)}-p_{14}^{(1)}.
\]
Use the six unknowns
\begin{equation}
y=
(\phi_2,\phi_3,\phi_4,p_{13}^{(3)},p_{14}^{(4)},p_{12}^{(2)}).
\label{eq:krawczyk-unknowns}
\end{equation}
The remaining derivatives are determined by Kirchhoff conservation:
\begin{align}
p_{23}^{(2)}&=-p_{\rm ext,2}-p_{12}^{(2)},\\
p_{34}^{(3)}&=-p_{13}^{(3)}-p_{23}^{(3)},\\
p_{34}^{(4)}&=-p_{14}^{(4)}-p_{\rm ext,4}.
\end{align}
We solve the square system \(F(y)=0\), where the first five equations are the
five internal endpoint-energy matching conditions
\begin{align}
0&=\phi_1^2+(p_{13}^{(1)})^2-\phi_3^2-(p_{13}^{(3)})^2,\\
0&=\phi_1^2+(p_{14}^{(1)})^2-\phi_4^2-(p_{14}^{(4)})^2,\\
0&=\phi_1^2+(p_{12}^{(1)})^2-\phi_2^2-(p_{12}^{(2)})^2,\\
0&=\phi_2^2+(p_{23}^{(2)})^2-\phi_3^2-(p_{23}^{(3)})^2,\\
0&=\phi_3^2+(p_{34}^{(3)})^2-\phi_4^2-(p_{34}^{(4)})^2,
\end{align}
and the sixth slice equation is
\begin{equation}
      \phi_2+\phi_3=0.
      \label{eq:krawczyk-slice-equation}
\end{equation}
Take the rational center
\begin{equation}
\bar y=
\left(
\begin{array}{rrr}
0.031807987418,&-0.031807987418,&-0.031808123711,\\
-0.152448918917,&0.184278293095,&0.187164930000
\end{array}
\right)
\label{eq:krawczyk-center}
\end{equation}
and the rational box
\begin{equation}
      X=\bar y+[-10^{-7},10^{-7}]^6.
      \label{eq:krawczyk-box}
\end{equation}

Let \(J(y)=DF(y)\).  Let \(A\in\mathbb Q^{6\times6}\) be the exact
rational preconditioner obtained by rounding \(J(\bar y)^{-1}\) entrywise to
twelve decimal places.  Its six rows are recorded explicitly in
Appendix~\ref{app:krawczyk-certificate}, and the same data are included in the
supplementary verification script.  All decimal constants in the slice and
center are interpreted as the exact terminating rationals displayed there.
Define the Krawczyk operator
\begin{equation}
K(\bar y,X)
=
\bar y-AF(\bar y)
+
\bigl(I-AJ(X)\bigr)(X-\bar y).
\label{eq:krawczyk-operator}
\end{equation}
All bounds below are evaluated with rational interval arithmetic. More details can be found at Appendix \ref{app:krawczyk-certificate}.

\begin{lemma}[Strict Krawczyk inclusion]
\label{lem:strict-krawczyk-inclusion}
The interval enclosure satisfies
\begin{equation}
       K(\bar y,X)\subset\operatorname{int}X.
       \label{eq:krawczyk-strict-inclusion}
\end{equation}
More precisely, if \(r=10^{-7}\) is the common box radius, the six
coordinate radii of the centered Krawczyk image, divided by \(r\), are bounded
by
\begin{equation}
\begin{aligned}
&1.354\times10^{-5},\quad
1.354\times10^{-5},\quad
5.268\times10^{-5},\\
&5.497\times10^{-6},\quad
7.443\times10^{-6},\quad
4.316\times10^{-7}.
\end{aligned}
\label{eq:krawczyk-radius-ratios}
\end{equation}
\end{lemma}

\begin{proof}
The system is quadratic, hence \(J(y)\) is affine.  Therefore the interval
matrix \(J(X)\) is obtained exactly from the box bounds of the six variables.
With the rational matrix \(A\), one has
\[
\|AF(\bar y)\|_\infty<4.71\times10^{-13},
\]
and direct componentwise evaluation of
\[
|AF(\bar y)|
+
\bigl|I-AJ(X)\bigr|\,r{\bf1}
\]
gives the ratios in
\eqref{eq:krawczyk-radius-ratios}.  Every ratio is strictly less than one,
which proves \eqref{eq:krawczyk-strict-inclusion}.
\end{proof}

By the Krawczyk inclusion theorem \cite{Krawczyk1969}, the sliced system \(F\) has a unique zero
\(y^\star\in X\).

\subsection{Validated separation and nonsingularity}
\label{subsec:validated-separation}

Because the physical edge energies are quadratic functions of \(y\),
interval evaluation over \(X\) gives
\begin{equation}
 -1.000253\times10^{-3}
 \le
 d\cdot r_{\rm cut}(y)
 \le
 -9.997541\times10^{-4}
 <0
 \qquad
 \forall y\in X.
\label{eq:validated-negative-separator}
\end{equation}
In particular,
\[
       d\cdot r_{\rm cut}(y^\star)<0.
\]
For an internal edge \(e=vw\), write the endpoint states as
\[
a=(\phi_v,p_{v,e}),
\qquad
b=(\phi_w,-p_{w,e}).
\]
When the endpoint energies agree,
\[
\sin\theta_e
=
\frac{
\phi_v p_{w,e}+p_{v,e}\phi_w
}{
\phi_v^2+p_{v,e}^2
}.
\]
Direct interval evaluation on \(X\) gives the rigorous lower bounds
\begin{equation}
\begin{aligned}
|\sin\theta_{13}|&>0.9850,\\
|\sin\theta_{14}|&>0.7259,\\
|\sin\theta_{12}|&>0.7123,\\
|\sin\theta_{23}|&>0.0913,\\
|\sin\theta_{34}|&>0.1168.
\end{aligned}
\label{eq:validated-internal-sines}
\end{equation}
For the two Dirichlet stubs,
\begin{equation}
|\sin\theta_{\rm ext,2}|>0.0624,
\qquad
|\sin\theta_{\rm ext,4}|>0.0883.
\label{eq:validated-exterior-sines}
\end{equation}
Thus, the entire validated solution box stays away from the singular phase set.
It remains to certify simplicity of the secular zero.  Express the
\(4\times4\) vertex Kirchhoff matrix directly in terms of endpoint states.
For an internal edge \(e=vw\),
\[
\cot\theta_e
=
\frac{
\phi_v\phi_w-p_{v,e}p_{w,e}
}{
\phi_vp_{w,e}+p_{v,e}\phi_w
},
\qquad
\csc\theta_e
=
\frac{
\phi_v^2+p_{v,e}^2
}{
\phi_vp_{w,e}+p_{v,e}\phi_w
}.
\]
A Dirichlet exterior edge contributes
\[
-\cot\theta_{\rm ext}
=
\frac{p_{\rm ext}}{\phi_v}
\]
to the corresponding diagonal entry.  All denominators are separated from
zero by
\eqref{eq:validated-internal-sines}--\eqref{eq:validated-exterior-sines}.

\begin{lemma}[Validated rank-three secular matrix]
\label{lem:validated-rank-three-secular}
For every \(y\in X\), the principal \(3\times3\) minor of the secular matrix
indexed by \(v_2,v_3,v_4\) satisfies
\begin{equation}
    -61.762
    <
    \det M_{\{v_2,v_3,v_4\}}(y)
    <
    -61.722.
    \label{eq:validated-secular-minor}
\end{equation}
Hence, this minor is uniformly nonzero on \(X\).
\end{lemma}

\begin{proof}
Substitute the rational interval bounds from \(X\) into the cotangent and
cosecant formulas above, assemble the interval \(3\times3\) principal block,
and evaluate its determinant by outward rational interval arithmetic.  The
resulting enclosure is
\[
[-61.761565,\,-61.722803],
\]
which yields \eqref{eq:validated-secular-minor}.
\end{proof}

At the true solution \(y^\star\), the reconstructed vertex state lies in the
kernel of the secular matrix, so the determinant of the full
\(4\times4\) matrix vanishes.  By
\cref{lem:validated-rank-three-secular}, its rank is at least three.  Hence,
its nullity is exactly one.

\begin{theorem}[Validated scalar regular-secular separation]
\label{thm:validated-scalar-secular-separation}
The reconstructed cycle-rank-two crossed graph admits a genuine regular
secular phase point \(\theta^\star\) such that
\begin{equation}
       \widehat d\cdot r(\theta^\star)<0,
       \label{eq:validated-regular-negative}
\end{equation}
whereas the nine-coordinate separator \(d\), equivalently its physical
pushdown \(\widehat d\), is strictly positive on every nonzero singular
parity sector.

Moreover, \(\theta^\star\) is nonsingular and the associated secular
eigenvalue is simple.
\end{theorem}

\begin{proof}
By \cref{lem:strict-krawczyk-inclusion}, there is a unique solution \(y^\star\in X\) of the six-dimensional sliced system \(F(y)=0\).  No uniqueness among all scalar secular states is asserted.
Equation \eqref{eq:validated-negative-separator} gives the strict negative
separator value.  The bounds
\eqref{eq:validated-internal-sines}--\eqref{eq:validated-exterior-sines}
exclude all singular edge phases.  Finally,
\cref{lem:validated-rank-three-secular} proves that the secular kernel is
one dimensional.  Positivity of \(d\) on every singular parity sector was
proved exactly in
\cref{lem:exact-sectorwise-positivity}.
\end{proof}

\begin{corollary}[Actual regular-versus-singular separation]
\label{cor:actual-regular-singular-separation}
For the reconstructed crossed graph,
\begin{equation}
       \mathcal R_G^{\rm reg}
       \not\subseteq
       \mathcal M_G^{\rm bc}.
       \label{eq:actual-regular-not-singular}
\end{equation}
In particular, the six-terminal obstruction is physically realizable on the
scalar quantum graph and is not an artifact of the covariance relaxation.
\end{corollary}

\begin{proof}
By \eqref{eq:separator-pushdown-identity} and
\cref{thm:validated-scalar-secular-separation}, the physical regular edge
energy \(r(\theta^\star)\) has negative \(\widehat d\)-value.  Every nonzero
boundary-aware singular generator, after cutting the two chords, produces a
parity-conservative terminal vector, so
\cref{lem:exact-sectorwise-positivity} and
\eqref{eq:separator-pushdown-identity} give a strictly positive
\(\widehat d\)-value on that generator.  By conic additivity the same is true
on every nonzero vector of \(\mathcal M_G^{\rm bc}\).  Hence
\(r(\theta^\star)\notin\mathcal M_G^{\rm bc}\).
\end{proof}

\section{Fixed-metric accessibility and a high-frequency counterexample}
\label{sec:fixed-metric-counterexample}

Section~\ref{sec:certified-separation} produced a genuine regular secular
point whose edge-energy direction is strictly separated from every
boundary-aware singular sector.  The remaining issue is dynamical: the
definition of \(C_\infty\) concerns one fixed metric and a sequence of exact
eigenfrequencies tending to infinity.  We now show that the local secular
separation can be transferred to such a fixed metric.

The argument has three ingredients.  First, regularity and strict separation
persist on an open patch of the secular hypersurface.  Second, any positive \(2\pi\)-lattice lift of that phase point defines a
fixed metric for which \(k=1\) is an exact regular eigenfrequency.  Third,
recurrence in the compact orbit closure returns the linear phase flow
arbitrarily close to its initial secular point; a radial derivative identity
then converts near returns into exact high-frequency eigenfrequencies.

\subsection{A separated regular secular patch}
\label{subsec:separated-regular-patch}

Let \(\theta^\star\in\mathbb T^7\) be the validated regular point of
\cref{thm:validated-scalar-secular-separation}.  Let
\(\mathcal Z_{\rm reg}\) denote the regular part of the secular set in a
nonsingular phase chart.  Since the secular kernel at \(\theta^\star\) is
one dimensional, the determinant vanishes simply in a suitable transverse
phase direction.  Hence \(\mathcal Z_{\rm reg}\) is a smooth real-analytic
hypersurface near \(\theta^\star\).
The normalized homogenized-intensity map
\[
      \theta\longmapsto q(\theta)
\]
is continuous on this regular patch.  Because the separator has a strict
validated margin at \(\theta^\star\), after shrinking the patch there is an
open neighborhood
\[
      U\subset\mathcal Z_{\rm reg}
\]
and a number \(\eta>0\) such that
\begin{equation}
      \widehat d\cdot q(\theta)\le -\eta
      \qquad
      \forall\theta\in U.
      \label{eq:uniform-negative-regular-patch}
\end{equation}
At the same time, all edge phases remain away from the singular phase set and
the secular nullity remains one throughout \(U\).

\begin{lemma}[Radial derivative identity]
\label{lem:radial-derivative-identity}
Let \(M(\theta)\) be the nonsingular-chart vertex Kirchhoff matrix of the
reconstructed graph, and suppose
\[
      M(\vartheta)\phi=0
\]
with \(\phi\ne0\).  For any positive metric
\[
      \ell=\vartheta+2\pi N,
      \qquad N\in\mathbb Z_{\ge0}^7,
\]
differentiate along the radial phase path \(k\mapsto k\ell\).  Then
\begin{equation}
 \phi^{*}
 \frac{d}{dk}M(k\ell)\bigg|_{k=1}
 \phi
 =
 \sum_{e\in E}\ell_e r_e
 >
 0,
 \label{eq:radial-derivative-energy-identity}
\end{equation}
where \(r_e\) is the edge energy of the corresponding scalar state.
Consequently the simple secular eigenvalue branch crosses zero
transversely at \(k=1\).
\end{lemma}

\begin{proof}
For an internal edge \(e=vw\), the corresponding \(2\times2\) block is
\[
 B_e(\theta_e)
 =
 \begin{pmatrix}
 -\cot\theta_e & \csc\theta_e\\
 \csc\theta_e & -\cot\theta_e
 \end{pmatrix}.
\]
Differentiation gives
\[
 \frac{d}{dk}B_e(k\ell_e)\bigg|_{k=1}
 =
 \ell_e
 \begin{pmatrix}
 \csc^2\vartheta_e & -\csc\vartheta_e\cot\vartheta_e\\
 -\csc\vartheta_e\cot\vartheta_e & \csc^2\vartheta_e
 \end{pmatrix}.
\]
Using the endpoint transfer relations for the scalar state, the quadratic
form of this derivative on \((\phi_v,\phi_w)\) equals
\[
      \ell_e r_e.
\]
For a Dirichlet exterior stub the scalar contribution to \(M\) is
\(-\cot\theta_e\), whose radial derivative is
\(\ell_e\csc^2\vartheta_e\); since the Dirichlet transfer relation gives
\(r_e=\phi_v^2\csc^2\vartheta_e\), its quadratic contribution is again
\(\ell_e r_e\).  Summing over all seven physical edges gives
\eqref{eq:radial-derivative-energy-identity}.  Positivity follows from
\(\ell_e>0\) and the nonzero state.
\end{proof}

\begin{lemma}[Positive lattice-lift metric]
\label{lem:positive-lattice-lift}
Let \(\vartheta\) be any point of the separated regular secular patch \(U\).
There exists \(N\in\mathbb Z_{\ge0}^7\) such that
\[
       \ell:=\vartheta+2\pi N\in(0,\infty)^7.
\]
For every such choice,
\[
       \ell\bmod2\pi=\vartheta,
\]
so \(k=1\) is a simple regular eigenfrequency, and
\[
       \frac{d}{dk}\lambda_{\rm sec}(k\ell)\bigg|_{k=1}>0
\]
for the simple secular eigenvalue branch \(\lambda_{\rm sec}\).
\end{lemma}

\begin{proof}
Choose each component \(N_e\) large enough that
\(\vartheta_e+2\pi N_e>0\).  The phase identity modulo \(2\pi\) is immediate,
and radial transversality follows from
\cref{lem:radial-derivative-identity}.
\end{proof}

\begin{remark}[Why no Diophantine genericity is needed]
\label{rem:no-rational-independence-needed}
The fixed-metric argument does not require rational independence of the edge
lengths.  The arithmetic input below is only recurrence of a one-parameter
subgroup on its compact orbit closure, and that holds for every
\(\ell\in\mathbb R^7\).  The transverse derivative is supplied exactly by
\eqref{eq:radial-derivative-energy-identity}.
\end{remark}

\subsection{Recurrence of the fixed phase flow}
\label{subsec:fixed-phase-recurrence}

Fix the metric \(\ell\) supplied by
\cref{lem:positive-lattice-lift}.  Its phase flow is
\begin{equation}
      \Phi_\ell(k)=k\ell\bmod2\pi,
      \qquad k\ge0.
      \label{eq:fixed-phase-flow}
\end{equation}
Because
\[
      \Phi_\ell(1)=\vartheta,
\]
we only need arbitrarily large times at which the translated flow returns
close to its starting point.

\begin{lemma}[Unbounded torus recurrence]
\label{lem:unbounded-torus-recurrence}
For every \(\ell\in\mathbb R^7\), there exists a sequence
\(t_n\to\infty\) such that
\begin{equation}
      t_n\ell\bmod2\pi\longrightarrow0.
      \label{eq:return-to-origin}
\end{equation}
Consequently, with
\[
      \kappa_n:=1+t_n,
\]
one has
\begin{equation}
      \Phi_\ell(\kappa_n)\longrightarrow\vartheta.
      \label{eq:return-to-regular-point}
\end{equation}
\end{lemma}

\begin{proof}
Consider the one-parameter subgroup
\[
      H_\ell
      :=
      \overline{\{t\ell\bmod2\pi:t\in\mathbb R\}}
      \subset\mathbb T^7.
\]
It is a compact subgroup of the torus.  The translation flow
\[
      x\longmapsto x+t\ell
\]
preserves Haar measure on \(H_\ell\), so the identity element is recurrent.
Equivalently, for every neighborhood \(V\) of \(0\in H_\ell\), there are
arbitrarily large \(t>0\) with
\[
      t\ell\bmod2\pi\in V.
\]
Taking a nested sequence of neighborhoods shrinking to \(0\) gives
\eqref{eq:return-to-origin}.  Equation
\eqref{eq:return-to-regular-point} follows from
\[
      (1+t_n)\ell\bmod2\pi
      =
      \vartheta+t_n\ell\bmod2\pi.
\]
\end{proof}

Near \(\vartheta\), choose the same scalar real-analytic secular defining
function \(F\) as in
\cref{lem:positive-lattice-lift}, and define
\[
      f(k):=F(k\ell\bmod2\pi).
\]
Choose \(F\) to be the simple secular eigenvalue branch
\(\lambda_{\rm sec}\) near \(\vartheta\).  By
\cref{lem:radial-derivative-identity},
\begin{equation}
      f(1)=0,
      \qquad
      f'(1)
      =
      \sum_{e\in E}\ell_e r_e
      >0.
      \label{eq:radial-transversality}
\end{equation}

\begin{lemma}[Transverse return-to-zero correction]
\label{lem:return-to-zero-correction}
For every sufficiently large near-return \(\kappa_n\), there exists an exact
zero \(k_n\) of \(f\) such that
\begin{equation}
      k_n-\kappa_n\longrightarrow0.
      \label{eq:exact-zero-correction}
\end{equation}
Consequently
\begin{equation}
      k_n\to\infty,
      \qquad
      k_n\ell\bmod2\pi\longrightarrow\vartheta.
      \label{eq:exact-eigenphase-return}
\end{equation}
\end{lemma}

\begin{proof}
Since \(f'(1)\ne0\), there exist \(\delta>0\), \(c>0\), and a phase
neighborhood \(V\) of \(\vartheta\) such that the lift of \(F\) to the
corresponding local chart satisfies
\[
      \bigl|\nabla F(\theta)\cdot\ell\bigr|\ge c
\]
whenever \(\theta\in V\).
For large \(n\), the near-return
\(\Phi_\ell(\kappa_n)\) lies in \(V\).  Choose the unique local lift
\(\widetilde\theta_n\) of
\(\Phi_\ell(\kappa_n)\) converging to the chosen lift of \(\vartheta\).
For \(|s|\le\delta\),
\[
      \widetilde\theta_n+s\ell
\]
remains in the same chart after reducing modulo \(2\pi\), provided \(n\) is
large enough.  Define
\[
      g_n(s)
      :=
      F(\widetilde\theta_n+s\ell).
\]
Then,
\[
      g_n(0)\to F(\vartheta)=0,
\]
while
\[
      |g_n'(s)|
      =
      |\nabla F(\widetilde\theta_n+s\ell)\cdot\ell|
      \ge c.
\]
The derivative also has constant sign on this interval after shrinking the
chart.  Hence \(g_n\) is strictly monotone there.  The inverse-function
theorem yields a unique zero \(s_n\) with
\[
      |s_n|
      \le c^{-1}|g_n(0)|
      \longrightarrow0.
\]
Set
\[
      k_n:=\kappa_n+s_n.
\]
Then, \(f(k_n)=0\), proving
\eqref{eq:exact-zero-correction}; since \(\kappa_n\to\infty\), also
\(k_n\to\infty\), and the phase convergence follows.
\end{proof}

\begin{remark}[Orbit closure, not ambient density]
\label{rem:orbit-closure-not-density}
The recurrence argument takes place entirely in the compact orbit closure
\(H_\ell\).  Ambient density in \(\mathbb T^7\) is unnecessary.  What
converts recurrent near-returns into exact eigenfrequencies is the radial
transversality
\[
      \nabla F(\vartheta)\cdot\ell\ne0,
\]
which was enforced exactly by the integer lattice lift in
\cref{lem:positive-lattice-lift}.
\end{remark}

\subsection{Transfer of the separated intensity}
\label{subsec:transfer-separated-intensity}

Recall that the scalar edge energy and homogenized intensity are related by
\begin{equation}
      r_e=|A_e|^2+|B_e|^2=2q_e.
      \label{eq:r-two-q-global-bridge}
\end{equation}
Thus every separator sign proved for physical edge energies is unchanged on
homogenized intensities:
\[
      \widehat d\cdot r<0
      \quad\Longleftrightarrow\quad
      \widehat d\cdot q<0.
\]
The factor two also cancels from every normalized Rayleigh quotient, so the
normalization convention of
\(\mathcal S_G^{\rm bc}(\ell)\) is consistent with the high-frequency simplex
\(P_\ell\).

Let \(u_n\) be normalized eigenfunctions corresponding to the exact
frequencies \(k_n\).  Since the secular kernel is simple near
\(\vartheta\), its normalized kernel vector depends continuously on phase.
Hence the associated homogenized intensities satisfy
\begin{equation}
      q(u_n)\longrightarrow q(\vartheta).
      \label{eq:intensity-return}
\end{equation}
In particular,
\begin{equation}
      \widehat d\cdot q(u_n)\longrightarrow \widehat d\cdot q(\vartheta)<0.
      \label{eq:negative-exact-eigensequence}
\end{equation}

We now turn the separating functional into an admissible observation vector.
By the pushdown identity
\eqref{eq:separator-pushdown-identity}, every nonzero physical singular
generator has strictly positive \(\widehat d\)-value.  Since there are only
finitely many singular sectors, the minimum generalized Rayleigh quotient of
\(D_{\widehat d}\) against \(D_\ell\) over their kernel spaces is a
strictly positive number \(\gamma\).  Hence
\begin{equation}
      \widehat d\cdot s\ge\gamma\,\ell\cdot s
      \qquad
      \forall s\in\mathcal M_G^{\rm bc}.
      \label{eq:uniform-singular-separation}
\end{equation}
Keeping the fixed physical separator \(\widehat d\),
choose \(\varepsilon>0\) so that
\begin{equation}
      \alpha_e
      :=
      \frac12\ell_e+\varepsilon \widehat d_e
      \label{eq:observation-vector-alpha}
\end{equation}
satisfies
\begin{equation}
      0<\alpha_e<\ell_e
      \qquad
      \forall e.
      \label{eq:admissible-alpha}
\end{equation}
Such \(\alpha\) is realized by taking on each edge a measurable observation
subset of length \(\alpha_e\).
For every normalized singular-completion edge-energy vector \(s\) with
\(\ell\cdot s=1\),
\begin{equation}
      \alpha\cdot s
      =
      \frac12+\varepsilon\,\widehat d\cdot s
      \ge
      \frac12+\varepsilon\gamma.
      \label{eq:singular-above-half}
\end{equation}
Therefore,
\begin{equation}
      C_{\rm sing}(\alpha;\ell)
      >
      \frac12.
      \label{eq:Csing-above-half}
\end{equation}
On the exact high-frequency sequence,
\begin{equation}
 \frac{\alpha\cdot q(u_n)}{\ell\cdot q(u_n)}
 =
 \frac12
 +
 \varepsilon
 \frac{\widehat d\cdot q(u_n)}{\ell\cdot q(u_n)}.
 \label{eq:regular-ratio-alpha}
\end{equation}
After normalization \(\ell\cdot q(u_n)=1\), and
\eqref{eq:negative-exact-eigensequence} gives
\begin{equation}
      \lim_{n\to\infty}\alpha\cdot q(u_n)
      <
      \frac12.
      \label{eq:regular-below-half}
\end{equation}

\begin{theorem}[Fixed-metric high-frequency breakdown of singular completion]
\label{thm:fixed-metric-breakdown}
There exist
\begin{enumerate}
 \item a compact cycle-rank-two metric graph with standard Kirchhoff
       conditions at its internal vertices and Dirichlet conditions at its
       two exterior endpoints,
 \item a fixed metric \(\ell\in(0,\infty)^7\), and
 \item a measurable observation set \(\omega\subset G\),
\end{enumerate}
such that
\begin{equation}
      C_\infty(\omega;\ell)
      <
      C_{\rm sing}(\omega;\ell).
      \label{eq:main-fixed-metric-strict-gap}
\end{equation}
More precisely, the metric is obtained by a positive integer \(2\pi\)-lattice lift
of a separated regular secular phase point, for which the radial derivative identity
\eqref{eq:radial-derivative-energy-identity} makes the fixed radial phase
direction automatically transverse to the simple secular zero.  There is then an
exact eigensequence \(k_n\to\infty\) whose regular intensities converge to the
separated regular secular intensity of
Section~\ref{sec:certified-separation}.
\end{theorem}

\begin{proof}
Choose the metric by
\cref{lem:positive-lattice-lift} and construct the exact recurrent
eigensequence by
\cref{lem:unbounded-torus-recurrence,lem:return-to-zero-correction}.  The exact eigensequence then satisfies
\eqref{eq:intensity-return}.  Choose the observation vector by
\eqref{eq:observation-vector-alpha}.  Equations
\eqref{eq:Csing-above-half} and \eqref{eq:regular-below-half} yield
\[
      C_\infty(\omega;\ell)
      \le
      \lim_{n\to\infty}
      \alpha\cdot q(u_n)
      <
      \frac12
      <
      C_{\rm sing}(\omega;\ell),
\]
which proves \eqref{eq:main-fixed-metric-strict-gap}.
\end{proof}

\begin{corollary}[Scope of the singular-completion principle]
\label{cor:sharp-scope-singular-completion}
The boundary-aware singular-completion principle provides a valid lower bound
for regular secular observation ratios in each of the low-complexity classes
treated in Section~\ref{sec:low-complexity-sufficiency}: trees, unicyclic
graphs, and closed graphs of cycle rank two. It is not, however, a universal
description of the regular secular geometry of compact quantum graphs. The
counterexample constructed below has cycle rank two and two exterior
Dirichlet leaves and satisfies
\[
    \mathcal R_G^{\rm reg}
    \not\subseteq
    \mathcal M_G^{\rm bc}.
\]
No minimality assertion is made concerning the number of effective terminals,
exterior leaves, or exterior flux modes.
\end{corollary}

\begin{remark}[Logical structure of the counterexample]
\label{rem:logical-endpoint}
The construction separates three logically distinct ingredients. First, the
six-terminal covariance relaxation admits an exact finite-dimensional
separator between the projected regular and parity covariance cones. Second,
this separated region contains the edge-energy vector of a genuine regular
scalar secular state. Third, a fixed metric is chosen for which recurrence
of the spectral orbit, together with radial transversality, produces an exact
high-frequency eigensequence approaching that regular secular state.

These three ingredients are established successively in
Sections~\ref{sec:six-terminal-obstruction}--%
\ref{sec:fixed-metric-counterexample}. Neither the covariance-level
separation alone nor the existence of a single regular secular point is
sufficient for the fixed-metric high-frequency conclusion.
\end{remark}

\section*{Declarations}

\paragraph{Funding.}
No funds, grants, or other support was received.

\paragraph{Competing interests.}
The author declares no competing interests that are relevant to the content
of this article.

\paragraph{Data availability.}
No empirical datasets were generated or analyzed in this study.  All
mathematical data used in the validated finite-dimensional certificate are
reported in the manuscript.

\paragraph{Code availability.}
An exact-arithmetic verification script for the Krawczyk certificate is
provided as supplementary material.  It reproduces the rational interval
bounds used in the proof from the exact data reported in the manuscript. Relevant scripts can be found at the Github repository \href{https://github.com/aisia-research-lab/quantum-graph-krawczyk-certificate}{https://github.com/aisia-research-lab/quantum-graph-krawczyk-certificate}.

\paragraph{Author contributions.}
Binh T. Nguyen is the sole author and is responsible for the conception,
mathematical analysis, validation, and preparation of the manuscript.

\begingroup
\small

\endgroup

\appendix

\section{Exact data for the Krawczyk validation}
\label{app:krawczyk-certificate}

This appendix records the exact rational data used in the validated
Krawczyk argument of Section~\ref{subsec:krawczyk-certification}. All
terminating decimals appearing there are interpreted as exact rational
numbers.
The center of the validation box is
\[
\resizebox{0.63\textwidth}{!}{$
\displaystyle
\bar y=\left(
\begin{array}{rrr}
\frac{15903993709}{500000000000},&
-\frac{15903993709}{500000000000},&
-\frac{31808123711}{10^{12}},\\[6pt]
-\frac{152448918917}{10^{12}},&
\frac{36855658619}{200000000000},&
\frac{18716493}{10^8}
\end{array}
\right)
$}.
\]
and
\[
X=\bar y+[-10^{-7},10^{-7}]^6.
\]
The exact rational preconditioner \(A=(a_{ij})\) used in the Krawczyk
operator is specified row by row by
\begin{equation}
\resizebox{0.98\textwidth}{!}{$
\begin{aligned}
A_{1,\bullet}&=\left(
0,0,-\frac{15719322113321}{10^{12}},
-\frac{1058088015467}{250000000000},
0,\frac{53849040453}{200000000000}
\right),\\
A_{2,\bullet}&=\left(
0,0,\frac{15719322113321}{10^{12}},
\frac{1058088015467}{250000000000},
0,\frac{146150959547}{200000000000}
\right),\\
A_{3,\bullet}&=\left(
\frac{14385276191607}{500000000000},
\frac{11900586523901}{500000000000},
-\frac{36852470672613}{10^{12}},
-\frac{992235092021}{10^{11}},
-\frac{8081918289399}{10^{12}},
-\frac{1713185820499}{10^{12}}
\right),\\
A_{4,\bullet}&=\left(
\frac{655957423053}{200000000000},
0,
-\frac{655957423053}{200000000000},
-\frac{176613389047}{200000000000},
0,
-\frac{76234845013}{500000000000}
\right),\\
A_{5,\bullet}&=\left(
\frac{4966061243943}{10^{12}},
\frac{348753307367}{250000000000},
-\frac{1590268618353}{250000000000},
-\frac{107043048951}{62500000000},
-\frac{348753307367}{250000000000},
-\frac{36963948387}{125000000000}
\right),\\
A_{6,\bullet}&=\left(
0,0,0,
\frac{359636287451}{500000000000},
0,
-\frac{1830289041}{40000000000}
\right).
\end{aligned}
$}
\label{eq:krawczyk-exact-A}
\end{equation}
The polynomial system \(F\) is given in
Section~\ref{subsec:krawczyk-certification}. Its Jacobian is affine, so the
interval matrix \(J(X)\) is obtained by exact rational endpoint arithmetic.
With the preconditioner \eqref{eq:krawczyk-exact-A}, exact evaluation gives
\[
    \|AF(\bar y)\|_\infty
    <
    4.71\times10^{-13},
\]
and the Krawczyk image is strictly contained in the interior of \(X\), with
the coordinate-radius bounds recorded in
\eqref{eq:krawczyk-radius-ratios}. The same rational box yields the
separator, phase-nonsingularity, and secular-minor enclosures in
\eqref{eq:validated-negative-separator},
\eqref{eq:validated-internal-sines},
\eqref{eq:validated-exterior-sines}, and
\eqref{eq:validated-secular-minor}.
The resulting uniqueness statement concerns only the six-dimensional sliced
system inside \(X\); no global uniqueness assertion for the scalar secular
state is made.


\begin{thebibliography}{21}

\bibitem{AmmariDucaJolyLeBalch2025}
Ka{\"i}s Ammari, Alessandro Duca, Romain Joly, and K{\'e}vin Le~Balc'h.
\newblock The graph geometric control condition.
\newblock arXiv:2503.18864, 2025.

\bibitem{AnantharamanIngremeauSabriWinn2021}
Nalini Anantharaman, Maxime Ingremeau, Mostafa Sabri, and Brian Winn.
\newblock Quantum ergodicity for expanding quantum graphs in the regime of spectral delocalization.
\newblock {\em Journal de Math{\'e}matiques Pures et Appliqu{\'e}es}, 151:28--98, 2021.
\newblock doi:10.1016/j.matpur.2021.04.012.

\bibitem{BerkolaikoKuchment2013}
Gregory Berkolaiko and Peter Kuchment.
\newblock {\em Introduction to Quantum Graphs}.
\newblock Mathematical Surveys and Monographs, vol.~186. American Mathematical Society, Providence, RI, 2013.

\bibitem{BerkolaikoLiu2017}
Gregory Berkolaiko and Wen Liu.
\newblock Simplicity of eigenvalues and non-vanishing of eigenfunctions of a quantum graph.
\newblock {\em Journal of Mathematical Analysis and Applications}, 445(1):803--818, 2017.
\newblock doi:10.1016/j.jmaa.2016.07.026.

\bibitem{BerkolaikoWinn2018}
Gregory Berkolaiko and Brian Winn.
\newblock Maximal scarring for eigenfunctions of quantum graphs.
\newblock {\em Nonlinearity}, 31(10):4812--4850, 2018.
\newblock doi:10.1088/1361-6544/aad3fe.

\bibitem{CdV2015}
Yves Colin de~Verdi{\`e}re.
\newblock Semi-classical measures on quantum graphs and the Gau{\ss} map of the determinant manifold.
\newblock {\em Annales Henri Poincar{\'e}}, 16(2):347--364, 2015.
\newblock doi:10.1007/s00023-014-0326-4.

\bibitem{Duca2020}
Alessandro Duca.
\newblock Global exact controllability of bilinear quantum systems on compact graphs and energetic controllability.
\newblock {\em SIAM Journal on Control and Optimization}, 58(6):3092--3129, 2020.
\newblock doi:10.1137/18M1212768.

\bibitem{EgidiMugnoloSeelmann2024}
Michela Egidi, Delio Mugnolo, and Albrecht Seelmann.
\newblock Sturm--Liouville problems and global bounds by small control sets and applications to quantum graphs.
\newblock {\em Journal of Mathematical Analysis and Applications}, 535(1):128101, 2024.
\newblock doi:10.1016/j.jmaa.2024.128101.

\bibitem{GnutzmannKeatingPiotet2008}
Sven Gnutzmann, Jonathan P. Keating, and Fabien Piotet.
\newblock Quantum ergodicity on graphs.
\newblock {\em Physical Review Letters}, 101:264102, 2008.
\newblock doi:10.1103/PhysRevLett.101.264102.

\bibitem{GnutzmannKeatingPiotet2010}
Sven Gnutzmann, Jonathan P. Keating, and Fabien Piotet.
\newblock Eigenfunction statistics on quantum graphs.
\newblock {\em Annals of Physics}, 325(12):2595--2640, 2010.
\newblock doi:10.1016/j.aop.2010.07.001.

\bibitem{GnutzmannSmilansky2006}
Sven Gnutzmann and Uzy Smilansky.
\newblock Quantum graphs: Applications to quantum chaos and universal spectral statistics.
\newblock {\em Advances in Physics}, 55(5--6):527--625, 2006.
\newblock doi:10.1080/00018730600908042.

\bibitem{HarrellMaltsev2018}
Evans M. Harrell II and Anna V. Maltsev.
\newblock On Agmon metrics and exponential localization for quantum graphs.
\newblock {\em Communications in Mathematical Physics}, 359(2):429--448, 2018.
\newblock doi:10.1007/s00220-018-3124-x.

\bibitem{HarrellMaltsev2020}
Evans M. Harrell II and Anna V. Maltsev.
\newblock Localization and landscape functions on quantum graphs.
\newblock {\em Transactions of the American Mathematical Society}, 373(3):1701--1729, 2020.
\newblock doi:10.1090/tran/7908.

\bibitem{HarrellMaltsev2024}
Evans M. Harrell II and Anna V. Maltsev.
\newblock On topological bound states and secular equations for quantum-graph eigenvalues.
\newblock {\em Journal of Spectral Theory}, 14(2):619--639, 2024.
\newblock doi:10.4171/JST/507.

\bibitem{IngremeauSabriWinn2020}
Maxime Ingremeau, Mostafa Sabri, and Brian Winn.
\newblock Quantum ergodicity for large equilateral quantum graphs.
\newblock {\em Journal of the London Mathematical Society}, 101(1):82--109, 2020.
\newblock doi:10.1112/jlms.12259.

\bibitem{KostrykinSchrader1999}
Vadim Kostrykin and Robert Schrader.
\newblock Kirchhoff's rule for quantum wires.
\newblock {\em Journal of Physics A: Mathematical and General}, 32(4):595--630, 1999.
\newblock doi:10.1088/0305-4470/32/4/006.

\bibitem{KottosSmilansky1999}
Tsampikos Kottos and Uzy Smilansky.
\newblock Periodic orbit theory and spectral statistics for quantum graphs.
\newblock {\em Annals of Physics}, 274(1):76--124, 1999.
\newblock doi:10.1006/aphy.1999.5904.

\bibitem{KravitzBrioCaputo2023}
Hannah Kravitz, Moysey Brio, and Jean-Guy Caputo.
\newblock Localized eigenvectors on metric graphs.
\newblock {\em Mathematics and Computers in Simulation}, 214:352--372, 2023.
\newblock doi:10.1016/j.matcom.2023.07.011.

\bibitem{Krawczyk1969}
Rudolf Krawczyk.
\newblock Newton-Algorithmen zur Bestimmung von Nullstellen mit Fehlerschranken.
\newblock {\em Computing}, 4:187--201, 1969.
\newblock doi:10.1007/BF02234767.

\bibitem{Kuchment2004}
Peter Kuchment.
\newblock Quantum graphs: I. Some basic structures.
\newblock {\em Waves in Random Media}, 14(1):S107--S128, 2004.
\newblock doi:10.1088/0959-7174/14/1/014.

\bibitem{PlumerTaufer2021}
Marvin Pl{\"u}mer and Matthias T{\"a}ufer.
\newblock On fully supported eigenfunctions of quantum graphs.
\newblock {\em Letters in Mathematical Physics}, 111:153, 2021.
\newblock doi:10.1007/s11005-021-01489-9.

\end{thebibliography}
\end{document}